\documentclass[10pt]{amsart}

\usepackage{fullpage}

\usepackage{amsmath}
\usepackage{amsxtra}
\usepackage{amscd}
\usepackage{amsthm}
\usepackage{amsfonts}
\usepackage{amssymb}
\usepackage{eucal}
\usepackage[all]{xy}
\usepackage{graphicx}
\usepackage{comment}
\usepackage{epsfig}
\usepackage{psfrag}
\usepackage{mathrsfs}
\usepackage{amscd}
\usepackage{rotating}
\usepackage{lscape}
\usepackage{amsbsy}
\usepackage{verbatim}
\usepackage{moreverb}
\usepackage{url}
\usepackage{hyperref}
\usepackage[utf8]{inputenc}
\usepackage{bbm}

\newtheorem{cor}[subsubsection]{Corollary}
\newtheorem{lem}[subsubsection]{Lemma}
\newtheorem{prop}[subsubsection]{Proposition}

\newtheorem{conj}[subsubsection]{Conjecture}
\newtheorem{thm}[subsubsection]{Theorem}

\theoremstyle{remark}
\newtheorem{rem}[subsubsection]{Remark}
\newtheorem{conv}[subsubsection]{Convention}
\newtheorem{example}[subsubsection]{Example}

\theoremstyle{remark}

\numberwithin{equation}{section}

\newcommand{\nc}{\newcommand}
\nc{\renc}{\renewcommand}
\nc{\ssec}{\subsection}
\nc{\sssec}{\subsubsection}
\nc{\on}{\operatorname}
\nc{\ms}{\mathsf}

\nc\ol{\overline}
\nc\wt{\widetilde}
\nc\tboxtimes{\wt{\boxtimes}}
\nc\tstar{\wt{\star}}
\nc{\alp}{\alpha}

\nc{\ZZ}{{\mathbb Z}}
\nc{\NN}{{\mathbb N}}
\nc{\CC}{{\mathbb C}}
\nc{\OO}{{\mathbb O}}
\renc{\SS}{{\mathbb S}}
\nc{\DD}{{\mathbb D}}
\nc{\GG}{{\mathbb G}}
\renewcommand{\AA}{{\mathbb A}}
\nc{\Fq}{{\mathbb F}_q}
\nc{\Fqb}{\ol{{\mathbb F}_q}}
\nc{\Ql}{\ol{{\mathbb Q}_\ell}}
\nc{\id}{\text{id}}
\nc\X{\mathcal X}

\nc{\Hom}{\on{Hom}}
\nc{\Lie}{\on{Lie}}
\nc{\Loc}{\on{Loc}}
\nc{\Pic}{\on{Pic}}
\nc{\Bun}{\on{Bun}}
\nc{\IC}{\on{IC}}
\nc{\Aut}{\on{Aut}}
\nc{\rk}{\on{rk}}
\nc{\Sh}{\on{Sh}}
\nc{\Perv}{\on{Perv}}
\nc{\pos}{{\on{pos}}}
\nc{\Conv}{\on{Conv}}
\nc{\Sph}{\on{Sph}}
\nc{\Sym}{\on{Sym}}
\renewcommand{\dim}{\on{dim}}
\nc{\BunBb}{\overline{\Bun}_B}
\nc{\BunNb}{\overline{\Bun}_N}
\nc{\BunTb}{\overline{\Bun}_T}
\nc{\BunBbm}{\overline{\Bun}_{B^-}}
\nc{\BunBbel}{\overline{\Bun}_{B,el}}
\nc{\BunBbmel}{\overline{\Bun}_{B^-,el}}
\nc{\Buno}{\overset{o}{\Bun}}
\nc{\BunPb}{{\overline{\Bun}_P}}
\nc{\BunBM}{\Bun_{B(M)}}
\nc{\BunBMb}{\overline{\Bun}_{B(M)}}
\nc{\BunPbw}{{\widetilde{\Bun}_P}}
\nc{\BunBP}{\widetilde{\Bun}_{B,P}}
\nc{\GUb}{\overline{G/U}}
\nc{\GUPb}{\overline{G/U(P)}}

\nc{\Hhom}{\underline{\on{Hom}}}
\nc\syminfty{\on{Sym}^{\infty}}
\nc\xl{\ol{x}}
\nc\thl{\ol{\theta}}
\nc\nul{\ol{\nu}}
\nc\mul{\ol{\mu}}
\nc\Sum\Sigma
\nc{\oX}{\overset{\circ}{X}{}}
\nc{\hl}{\overset{\leftarrow}h{}}
\nc{\hr}{\overset{\rightarrow}h{}}
\nc{\M}{{\mathcal M}}
\nc{\N}{{\mathcal N}}
\nc{\F}{{\mathcal F}}
\nc{\D}{{\mathcal D}}
\nc{\Q}{{\mathcal Q}}
\nc{\Y}{{\mathcal Y}}
\nc{\G}{{\mathcal G}}
\nc{\E}{{\mathcal E}}
\nc{\CalC}{{\mathcal C}}
\nc\Dh{\widehat{\D}}
\renewcommand{\O}{{\mathcal O}}
\nc{\C}{{\mathcal C}}
\nc{\K}{{\mathcal K}}
\renewcommand{\H}{{\mathcal H}}
\renewcommand{\S}{{\mathcal S}}
\nc{\T}{{\mathcal T}}
\nc{\V}{{\mathcal V}}
\renc{\P}{{\mathcal P}}
\nc{\A}{{\mathcal A}}
\nc{\B}{{\mathcal B}}
\nc{\U}{{\mathcal U}}

\nc{\Gr}{{\on{Gr}}}

\nc{\frn}{{\check{\mathfrak u}(P)}}

\nc{\fC}{\mathfrak C}
\nc{\p}{\mathfrak p}
\nc{\q}{\mathfrak q}
\nc\f{{\mathfrak f}}

\nc{\qo}{{\mathfrak q}}
\nc{\po}{{\mathfrak p}}
\nc{\s}{{\mathfrak s}}
\nc\w{\text{w}}

\renewcommand{\mod}{{\on{-mod}}}

\nc\mathi\iota
\nc\Spec{\on{Spec}}
\nc\Mod{\on{Mod}}
\nc{\tw}{\widetilde{\mathfrak t}}
\nc{\pw}{\widetilde{\mathfrak p}}
\nc{\qw}{\widetilde{\mathfrak q}}
\nc{\jw}{\widetilde j}

\nc{\grb}{\overline{\Gr_{X^{\fset}}}}
\nc{\I}{\mathcal I}
\renewcommand{\i}{\mathfrak i}

\nc{\much}{{\check\mu}}
\nc{\omegach}{{\check\omega}}
\nc{\nuch}{{\check\nu}}
\nc{\lambdach}{{\check\lambda}}
\nc{\alphach}{{\check\alpha}}

\nc{\rhoch}{{\check\rho}}
\nc{\ch}{{\check\fh}}

\nc{\Hb}{\overline{\H}}

\nc{\BA}{{\mathbb{A}}}
\nc{\BC}{{\mathbb{C}}}
\nc{\BE}{{\mathbb{E}}}
\nc{\BG}{{\mathbb{G}}}
\nc{\BM}{{\mathbb{M}}}
\nc{\BO}{{\mathbb{O}}}
\nc{\BD}{{\mathbb{D}}}
\nc{\BN}{{\mathbb{N}}}
\nc{\BP}{{\mathbb{P}}}
\nc{\BR}{{\mathbb{R}}}
\nc{\BZ}{{\mathbb{Z}}}
\nc{\BS}{{\mathbb{S}}}
\nc{\BV}{{\mathbb{V}}}

\nc{\CA}{{\mathcal{A}}}
\nc{\CB}{{\mathcal{B}}}

\nc{\CE}{{\mathcal{E}}}
\nc{\CF}{{\mathcal{F}}}
\nc{\CH}{{\mathcal{H}}}

\nc{\CL}{{\mathcal{L}}}
\nc{\CG}{{\mathcal{G}}}
\nc{\CM}{{\mathcal{M}}}
\nc{\CN}{{\mathcal{N}}}
\nc{\CK}{{\mathcal{K}}}
\nc{\CO}{{\mathcal{O}}}
\nc{\CP}{{\mathcal{P}}}
\nc{\CQ}{{\mathcal{Q}}}
\nc{\CR}{{\mathcal{R}}}
\nc{\CS}{{\mathcal{S}}}
\nc{\CT}{{\mathcal{T}}}
\nc{\CU}{{\mathcal{U}}}
\nc{\CV}{{\mathcal{V}}}
\nc{\CW}{{\mathcal{W}}}
\nc{\CX}{{\mathcal{X}}}
\nc{\CY}{{\mathcal{Y}}}
\nc{\CZ}{{\mathcal{Z}}}
\nc{\CI}{{\mathcal{I}}}

\nc{\csM}{{\check{\mathcal A}}{}}
\nc{\oM}{{\overset{\circ}{\mathcal M}}{}}
\nc{\obM}{{\overset{\circ}{\mathbf M}}{}}
\nc{\oCA}{{\overset{\circ}{\mathcal A}}{}}
\nc{\obA}{{\overset{\circ}{\mathbf A}}{}}
\nc{\ooM}{{\overset{\circ}{M}}{}}
\nc{\osM}{{\overset{\circ}{\mathsf M}}{}}
\nc{\vM}{{\overset{\bullet}{\mathcal M}}{}}
\nc{\nM}{{\underset{\bullet}{\mathcal M}}{}}
\nc{\oD}{{\overset{\circ}{\mathcal D}}{}}
\nc{\obC}{{\overset{\circ}{\mathbf C}}{}}
\nc{\obD}{{\overset{\circ}{\mathbf D}}{}}
\nc{\oA}{{\overset{\circ}{\mathbb A}}{}}
\nc{\op}{{\overset{\bullet}{\mathbf p}}{}}
\nc{\oU}{{\overset{\bullet}{\mathcal U}}{}}
\nc{\oZ}{{\overset{\circ}{\mathcal Z}}{}}
\nc{\ofZ}{{\overset{\circ}{\mathfrak Z}}{}}
\nc{\oF}{{\overset{\circ}{\fF}}}

\nc{\fa}{{\mathfrak{a}}}
\nc{\fb}{{\mathfrak{b}}}
\nc{\fc}{{\mathfrak{c}}}
\nc{\fd}{{\mathfrak{d}}}
\nc{\ff}{{\mathfrak{f}}}
\nc{\fg}{{\mathfrak{g}}}
\nc{\fgl}{{\mathfrak{gl}}}
\nc{\fh}{{\mathfrak{h}}}
\nc{\fj}{{\mathfrak{j}}}
\nc{\fl}{{\mathfrak{l}}}
\nc{\fm}{{\mathfrak{m}}}
\nc{\fn}{{\mathfrak{n}}}
\nc{\fu}{{\mathfrak{u}}}
\nc{\fp}{{\mathfrak{p}}}
\nc{\fr}{{\mathfrak{r}}}
\nc{\fs}{{\mathfrak{s}}}
\nc{\fz}{{\mathfrak{z}}}
\nc{\fsl}{{\mathfrak{sl}}}
\nc{\hsl}{{\widehat{\mathfrak{sl}}}}
\nc{\hgl}{{\widehat{\mathfrak{gl}}}}
\nc{\hg}{{\widehat{\mathfrak{g}}}}
\nc{\chg}{{\widehat{\mathfrak{g}}}{}^\vee}
\nc{\hn}{{\widehat{\mathfrak{n}}}}
\nc{\chn}{{\widehat{\mathfrak{n}}}{}^\vee}

\nc{\fA}{{\mathfrak{A}}}
\nc{\fB}{{\mathfrak{B}}}
\nc{\fD}{{\mathfrak{D}}}
\nc{\fE}{{\mathfrak{E}}}
\nc{\fF}{{\mathfrak{F}}}
\nc{\fG}{{\mathfrak{G}}}
\nc{\fK}{{\mathfrak{K}}}
\nc{\fL}{{\mathfrak{L}}}
\nc{\fM}{{\mathfrak{M}}}
\nc{\fN}{{\mathfrak{N}}}
\nc{\fP}{{\mathfrak{P}}}
\nc{\fU}{{\mathfrak{U}}}
\nc{\fV}{{\mathfrak{V}}}
\nc{\fZ}{{\mathfrak{Z}}}

\nc{\bb}{{\mathbf{b}}}
\nc{\bc}{{\mathbf{c}}}
\nc{\bd}{{\mathbf{d}}}
\nc{\bbf}{{\mathbf{f}}}
\nc{\be}{{\mathbf{e}}}
\nc{\bg}{{\mathbf{g}}}
\nc{\bi}{{\mathbf{i}}}
\nc{\bj}{{\mathbf{j}}}
\nc{\bn}{{\mathbf{n}}}
\nc{\bo}{{\mathbf{o}}}
\nc{\bp}{{\mathbf{p}}}
\nc{\bq}{{\mathbf{q}}}
\nc{\bt}{{\mathbf{t}}}
\nc{\bu}{{\mathbf{u}}}
\nc{\bv}{{\mathbf{v}}}
\nc{\bx}{{\mathbf{x}}}
\nc{\bs}{{\mathbf{s}}}
\nc{\by}{{\mathbf{y}}}
\nc{\bw}{{\mathbf{w}}}
\nc{\bA}{{\mathbf{A}}}
\nc{\bK}{{\mathbf{K}}}
\nc{\bB}{{\mathbf{B}}}
\nc{\bC}{{\mathbf{C}}}
\nc{\bG}{{\mathbf{G}}}
\nc{\bD}{{\mathbf{D}}}
\nc{\bH}{{\mathbf{H}}}
\nc{\bM}{{\mathbf{M}}}
\nc{\bN}{{\mathbf{N}}}
\nc{\bO}{{\mathbf{O}}}
\nc{\bT}{{\mathbf{T}}}
\nc{\bV}{{\mathbf{V}}}
\nc{\bW}{{\mathbf{W}}}
\nc{\bX}{{\mathbf{X}}}
\nc{\bZ}{{\mathbf{Z}}}
\nc{\bS}{{\mathbf{S}}}

\nc{\sA}{{\mathsf{A}}}
\nc{\sB}{{\mathsf{B}}}
\nc{\sC}{{\mathsf{C}}}
\nc{\sD}{{\mathsf{D}}}
\nc{\sF}{{\mathsf{F}}}
\nc{\sG}{{\mathsf{G}}}
\nc{\sK}{{\mathsf{K}}}
\nc{\sM}{{\mathsf{M}}}
\nc{\sO}{{\mathsf{O}}}
\nc{\sW}{{\mathsf{W}}}
\nc{\sQ}{{\mathsf{Q}}}
\nc{\sP}{{\mathsf{P}}}
\nc{\sV}{{\mathsf{V}}}
\nc{\sS}{{\mathsf{S}}}
\nc{\sT}{{\mathsf{T}}}
\nc{\sZ}{{\mathsf{Z}}}
\nc{\sfp}{{\mathsf{p}}}
\nc{\sll}{{\mathsf{l}}}
\nc{\sr}{{\mathsf{r}}}
\nc{\bk}{{\mathsf{k}}}
\nc{\sg}{{\mathsf{g}}}

\nc{\sff}{{\mathsf{f}}}
\nc{\sfb}{{\mathsf{b}}}
\nc{\sfc}{{\mathsf{c}}}
\nc{\sd}{{\mathsf{d}}}
\nc{\se}{{\mathsf{e}}}

\nc{\BK}{{\bar{K}}}

\nc{\tA}{{\widetilde{\mathbf{A}}}}
\nc{\tB}{{\widetilde{\mathcal{B}}}}
\nc{\tg}{{\widetilde{\mathfrak{g}}}}
\nc{\tG}{{\widetilde{G}}}
\nc{\TM}{{\widetilde{\mathbb{M}}}{}}
\nc{\tO}{{\widetilde{\mathsf{O}}}{}}
\nc{\tU}{{\widetilde{\mathfrak{U}}}{}}
\nc{\TZ}{{\tilde{Z}}}
\nc{\tx}{{\tilde{x}}}
\nc{\tbv}{{\tilde{\bv}}}
\nc{\tfP}{{\widetilde{\mathfrak{P}}}{}}
\nc{\tz}{{\tilde{\zeta}}}
\nc{\tmu}{{\tilde{\mu}}}

\nc{\urho}{\underline{\rho}}
\nc{\uB}{\underline{B}}
\nc{\uC}{{\underline{\mathbb{C}}}}
\nc{\ui}{\underline{i}}
\nc{\uj}{\underline{j}}
\nc{\ofP}{{\overline{\mathfrak{P}}}}
\nc{\oB}{{\overline{\mathcal{B}}}}
\nc{\og}{{\overline{\mathfrak{g}}}}
\nc{\oI}{{\overline{I}}}

\nc{\eps}{\varepsilon}
\nc{\hrho}{{\hat{\rho}}}

\nc{\one}{{\mathbf{1}}}
\nc{\two}{{\mathbf{t}}}

\nc{\Rep}{{\mathop{\operatorname{\rm Rep}}}}
\nc{\Tot}{{\mathop{\operatorname{\rm Tot}}}}
\nc{\Ker}{{\mathop{\operatorname{\rm Ker}}}}
\nc{\Hilb}{{\mathop{\operatorname{\rm Hilb}}}}
\nc{\End}{{\mathop{\operatorname{\rm End}}}}
\nc{\Ext}{{\mathop{\operatorname{\rm Ext}}}}
\nc{\CHom}{{\mathop{\operatorname{{\mathcal{H}}\it om}}}}
\nc{\GL}{{\mathop{\operatorname{\rm GL}}}}
\nc{\gr}{{\mathop{\operatorname{\rm gr}}}}
\nc{\Id}{{\mathop{\operatorname{\rm Id}}}}
\nc{\de}{{\mathop{\operatorname{\rm def}}}}
\nc{\length}{{\mathop{\operatorname{\rm length}}}}
\nc{\supp}{{\mathop{\operatorname{\rm supp}}}}

\nc{\Cliff}{{\mathsf{Cliff}}}
\nc{\Fl}{\on{Fl}}
\nc{\Fib}{{\mathsf{Fib}}}
\nc{\Coh}{{\on{Coh}}}
\nc{\QCoh}{{\on{QCoh}}}
\nc{\IndCoh}{{\on{IndCoh}}}
\nc{\FCoh}{{\mathsf{FCoh}}}

\nc{\reg}{{\text{\rm reg}}}

\nc{\cplus}{{\mathbf{C}_+}}
\nc{\cminus}{{\mathbf{C}_-}}
\nc{\cthree}{{\mathbf{C}_*}}
\nc{\Qbar}{{\bar{Q}}}
\nc\Eis{\on{Eis}}
\nc\Eisb{\ol\Eis{}}
\nc\Eisr{\on{Eis}^{rat}{}}
\nc\wh{\widehat}
\nc{\Def}{\on{Def_{\check{\fb}}(E)}}
\nc{\barZ}{\overline{Z}{}}
\nc{\barbarZ}{\overline{\barZ}{}}
\nc{\barpi}{\overline\pi}
\nc{\barbarpi}{\overline\barpi}
\nc{\barpip}{\overline\pi{}^+}
\nc{\barpim}{\overline\pi{}^-}

\nc{\fq}{\mathfrak q}

\nc{\fqb}{\ol{\fq}{}}
\nc{\fpb}{\ol{\fp}{}}
\nc{\fpr}{{\fp^{rat}}{}}
\nc{\fqr}{{\fq^{rat}}{}}

\nc{\hattimes}{\wh\otimes}

\nc{\bh}{{\bar{h}}}
\nc{\bOmega}{{\overline{\Omega(\check \fn)}}}

\nc{\seq}[1]{\stackrel{#1}{\sim}}

\nc{\cT}{{\check{T}}}
\nc{\cG}{{\check{G}}}
\nc{\cM}{{\check{M}}}
\nc{\cB}{{\check{B}}}
\nc{\cP}{{\check{P}}}

\nc{\ct}{{\check{\mathfrak t}}}
\nc{\cg}{{\check{\fg}}}
\nc{\cb}{{\check{\fb}}}
\nc{\cn}{{\check{\fn}}}
\nc{\cp}{{\check{\fp}}}
\nc{\cm}{{\check{\fm}}}

\nc{\cmu}{{\check\mu}}
\nc{\cnu}{{\check\nu}}
\nc{\ceta}{{\check\lambda}}

\nc{\imathb}{{\ol{\imath}}}
\nc{\rlr}{\overset{\longrightarrow}{\underset{\longrightarrow}\longleftarrow}}

\nc{\LocSys}{\on{LocSys}}

\nc{\Vect}{\ms{Vect}}
\nc{\Whit}{\ms{Whit}}
\nc{\bTb}{\ol{\on{CT}}}
\nc{\Ran}{\on{Ran}}
\nc{\bTr}{\on{CT}^{rat}{}}
\nc\jmathr{\jmath^{rat}{}}
\nc{\ux}{\underline{x}}
\nc{\calpha}{{\check\alpha}}
\nc{\ind}{{\mathsf{ind}}}
\nc{\oblv}{{\mathsf{oblv}}}
\nc{\StinftyCat}{\on{DGCat}}
\nc{\inftygroup}{\infty\on{-Grpd}}

\nc{\fset}{\on{fSet}}
\nc{\LocSysG}{\LocSys_{\cG}}
\nc{\Sing}{\on{Sing}}
\nc{\dr}{{\on{dR}}}
\nc{\Dmod}{\on{D-mod}}
\nc{\Ind}{\on{Ind}}
\nc{\Sat}{\on{Sat}}
\nc{\Ho}{\on{Ho}}
\nc{\Res}{\on{Res}}
\nc{\sotimes}{\overset{!}\otimes}
\nc{\mmod}{{\on{-}}{\mathbf{mod}}}
\nc{\Maps}{\on{Maps}}
\nc{\CMaps}{{\mathcal Maps}}
\nc{\bMaps}{{\mathbf{Maps}}}

\nc{\dgSch}{\on{DGSch}}
\nc{\dgindSch}{\on{DGindSch}}
\nc{\indSch}{\ms{indSch}}
\nc{\IndSch}{\ms{IndSch}}
\nc{\Sch}{\mathsf{Sch}}
\nc{\affdgSch}{\on{DGSch}^{\on{aff}}}
\nc{\affSch}{\ms{Sch}^{\ms{aff}}}

\nc{\inftypic}{\infty\on{-PicGrpd}}
\nc{\inftyCat}{\infty\ms{-Cat}}
\nc{\MoninftyCat}{\infty\on{-Cat}^{Mon}}
\nc{\SymMoninftyCat}{\infty\on{-Cat}^{\on{SymMon}}}
\nc{\SymMonStinftyCat}{\on{DGCat}^{\on{SymMon}}}
\nc{\MonStinftyCat}{\on{DGCat}^{Mon}}

\nc{\csupp}{\supp}
\nc{\Arth}{\on{Arth}}
\nc{\ArthG}{{\on{Arth}_\cG}}
\nc{\ul}{\underline}

\nc{\m}{\mu}

\nc{\xto}{\xrightarrow}

\renewcommand{\dim}{\mathsf{dim}}
\newcommand{\codim}{\mathsf{codim}}
\nc{\corr}{\mathsf{corr}}
\nc{\ft}{\mathsf{ft}}
\nc{\fty}{\mathsf{ft}}

\nc{\cores}{\mathsf{cores}}
\nc{\coind}{\mathsf{coind}}

\nc{\actotimes}{\overset{*!}\otimes}

\nc{\hto}{\hookrightarrow}
\nc{\bP}{\mathbf{P}}

\nc{\VV}{\mathbb V}

\nc{\rrep}{\on{-} \mathbf{rep}}

\nc{\ccomod}{\on - \mathbf{comod}}

\nc{\citep}{\cite}

\nc{\Tate}{\mathsf{Tate}}
\nc{\rec}{\mathsf{rec}}
\newcommand{\HC}{\mathsf{HC}}

\renewcommand{\exp}{\mathit{exp}}
\newcommand{\pro}{\mathsf{pro}}

\nc{\flip}{\text{<}}
\nc{\IFT}{\mathsf{IFT}}

\nc{\biota}{\boldsymbol{\iota}}

\nc{\LC}{\mathsf{LC}} 

\nc{\ins}{\mathsf{ins}}
\nc{\ev}{\mathsf{ev}}

\nc{\inv}{\mathsf{inv}}

\nc{\rel}{\mathit{rel}}
\newcommand{\R}{\mathcal{R}}

\nc{\ren}{\mathit{ren}}
\nc{\diam}{\diamond}

\nc{\Heis}{\mathsf{Heis}}
\nc{\Corr}{\mathsf{Corr}}
\nc{\Coalg}{\ms{Coalg}}
\nc{\Alg}{\ms{Alg}}
\nc{\footcite}{\footnote}
\nc{\fch}{\check{f}}
\nc{\bL}{\mathbf{L}}

\nc{\pr}{\mathsf{pr}}
\nc{\Fun}{\on{Fun}}
\nc{\Cu}{\C_{univ}}

\nc{\coact}{\mathsf{coact}}

\nc{\ppart}{(\!(t)\!)}
\nc{\bphi}{\boldsymbol{\phi}}
\nc{\botimes}{\boldsymbol{\otimes}}

\nc{\chiral}[1]{\langle #1 \rangle}

\nc{\chiralx}{\chiral{X^I}}
\nc{\Gch}{\check{G}}
\nc{\FT}{\sF\sT}
\nc{\qcoh}{\QCoh}
\nc{\RepGran}{\Rep(G)_{\Ran}}
\nc{\RepGchran}{\Rep(G^\vee)_{\Ran}}

\nc{\rr}{\xymatrix{ \ar@<-0.1ex>[r] \ar@<.8ex>[r]  & } }
\nc{\rrr}{\xymatrix{ \ar@<.2ex>[r] \ar@<.9ex>[r] \ar@<-0.5ex>[r] & } }

\nc{\Z}{\CZ}
\nc{\CCC}{\wt\CC}
\renewcommand{\Q}{\QCoh}
\nc{\calN}{\N}
\nc{\calW}{\mathcal{W}}
\nc{\calF}{\mathcal{F}}
\nc{\calH}{\mathcal{H}}
\nc{\calO}{\mathcal{O}}
\nc{\Comod}{\on{Comod}}
\nc{\bGamma}{\mathbf{\Gamma}}
\nc{\coGamma}{\on{co-}\bGamma}

\nc{\Jets}{\on{Jets}}
\nc{\act}{\mathsf{act}}
\nc{\Av}{\mathsf{Av}}
\nc{\Ad}{\on{Ad}}
\nc{\BGRan}{BG_{\Ran}}

\nc{\uscolim}[1]{\underset{#1}{\on{colim}} \, }
\nc{\usotimes}[1]{\underset{#1}{\otimes}}
\nc{\ustimes}[1]{\underset{#1}{\times}}
\nc{\colim}{\on{colim}}

\nc{\cpt}{{\on{cpt}}}
\nc{\dR}{{\on{dR}}}
\nc{\Div}{{\on{Div}}}
\nc{\DGCat}{\ms{DGCat}}
\nc{\DGCatcont}{\on{DGCat}_{cont}}
\nc{\glob}{{\on{glob}}}
\nc{\GS}{{\on{GS}}}
\nc{\heart}{{\heartsuit}}
\nc{\loc}{{\on{loc}}}

\nc{\Nout}{\bN_{out}}

\nc{\ran}{\on{Ran}}
\renewcommand{\op}{{\on{op}}}
\nc{\pt}{{\ms{pt}}}
\nc{\PreStk}{{\ms{PreStk}}}
\nc{\Cat}{{\on{Cat}}}
\nc{\DGSchaff}{{\on{DGSch}^{\text{aff}}}}
\nc{\ShvCat}{{\mathsf{ShvCat}}}
\nc{\Sect}{{\on{Sect}}}
\nc{\Stab}{{\on{Stab}}}
\nc{\tstr}{{\on{t-str}}}

\nc{\kk}{\mathbbm{k}} 

\renewcommand{\Dmod}{\mathfrak{D}}

\nc{\Whitglob}{\on{Whit}_{\on{glob}}}

\nc{\virg}[1]{``#1"}

\nc{\restr}[2]{\left. #1 \right |_{#2}}
\nc{\uprestr}[2]{\left. #1 \right |^{#2}}
\nc{\bLoc}{{\mathbf{Loc}}}

\newcommand{\GO}{G({\mathcal O})}

\nc{\squigto}{\rightsquigarrow}

\nc{\rev}{\mathit{rev}}

\nc{\ICoh}{\IndCoh}

\nc{\coev}{\mathsf{coev}}

\begin{document}

\title{Loop group actions on categories and Whittaker invariants}

\author{Dario Beraldo}

\begin{abstract}

The present paper is divided in three parts.

In the first one, we develop the theory of $\fD$-modules on ind-schemes of pro-finite type. This allows to define $\fD$-modules on (algebraic) loop groups and, consequently, the notion of \emph{strong loop group action} on a DG category.

In the second part, we construct the functors of \emph{Whittaker invariants} and \emph{Whittaker coinvariants}, which take as input a DG category acted on by $G\ppart$, the loop group of a reductive group $G$. 
Roughly speaking, the Whittaker invariant category of $\C$ is the full subcategory $\C^{N \ppart, \chi} \subseteq \C$ consisting of objects that are $N\ppart$-invariant \emph{against} a fixed non-degenerate character $\chi:N \ppart \to \GG_a$ of conductor zero. (Here $N$ is the maximal unipotent subgroup of $G$.) The Whittaker coinvariant category $\C_{N \ppart, \chi}$ is defined by a dual construction.

In the third part, we construct a functor $\Theta: \C_{N \ppart, \chi} \to \C^{N \ppart, \chi}$, which depends on a choice of dimension theory for $G\ppart$. We conjecture this functor to be an equivalence.
After developing the Fourier-Deligne transform for Tate vector spaces, we prove this conjecture for $G= GL_n$. We show that both Whittaker categories can be obtained by taking invariants of $\C$ with respect to a very explicit pro-unipotent group subscheme (not indscheme!) of $G \ppart$.
\end{abstract}

\maketitle

\tableofcontents

\section{Introduction}

Categorical (or higher) representation theory is the study of symmetries of categories. In mathematical terms, such symmetries are encoded by the notion of \emph{group action} on a category. 
Our interest in this abstract concept originates from the \emph{local geometric Langlands correspondence}, a conjecture put forward by E. Frenkel and D. Gaitsgory (see, e.g., \cite{FG2}, \cite{FG3}, \cite{FG6}, or \cite{F} and \cite{quantum-langlands} for reviews).
We shall recall this conjecture in Sect. \ref{ssec:local-geom-langlands}: the curious reader might jump there directly and return to the beginning of the introduction only if necessary.

\sssec{}

Before defining the notion of \virg{group action on a category}, let us clarify what we mean by ``group" and what we mean by ``category".

\medskip

Our geometric context is the world of algebraic geometry over an algebraically closed field $\kk$ of characteristic zero. Hence, our geometric objects will be schemes, ind-schemes, stacks, etc., always defined over $\kk$.
Accordingly, by ``group" we mean ``group scheme" or, more generally, ``group ind-scheme" (defined over $\kk$). 

In fact, as the title of the paper suggests, our favorite example will be the group ind-scheme $G \ppart$, the \emph{loop group} of a reductive group $G$.

\medskip

As for the categorical context, by the term \virg{category} we mean a differential graded (DG) category over $\kk$ which is co-complete (i.e., it contains all colimits).
The collections of such categories, together with colimit-preserving functors among them, forms an $\infty$-category, denoted by $\DGCat$.
\footnote{
The foundational basis of such notions is contained in the books \cite{L0}, \cite{L1}.
A succinct review, to which we shall refer for notations and main results, is \cite{DG}.
}
This set-up is extremely convenient for performing algebraic operations on categories, directly generalizing standard operations of linear algebra. 
Most notably, there is a tensor product that makes $\DGCat$ into a symmetric monoidal $\infty$-category.

With this said, a reader not interested in the details can follow this introduction without knowing any of the technical aspects of $\DGCat$. The two concepts that are important to know, or accept, are the following. 
The first one is that of \emph{dualizability of a category}, see \cite[Section 2]{DG}. 
The second one is that of \emph{module for a monoidal category $\A$}, see \cite[Section 4]{DG}. 
Roughly speaking, this is a category $\M$ equipped with an \virg{action} functor $\A \times \M \to \M$ (commuting with colimits in each variable) satisfying the natural compatibility conditions. We denote by $\A \mmod$ the $\infty$-category of (left) $\A$-modules.
\footnote{The $\infty$-category of comodules over a comonoidal category is defined similarly.}

\ssec{Group actions on categories}

Having agreed on the terminology, let us now make sense of the notion of \emph{category equipped with a $G$-action}, where for the moment we assume that $G$ is an affine algebraic group (in particular, $G$ is of finite type).

To do so, let us mimic the classical setting: the structure of a $G$-representation on a vector space $V$ consists of a coaction on $V$ of the coalgebra $\Gamma(G, \O_G)$ of functions on $G$ with convolution coproduct (i.e., pull-back along the multiplication).

\medskip

In our context, one has to replace functions on the group with sheaves: there are thus at least two possible notions of group action, accounting for the two standard sheaf-theoretic contexts of \emph{quasi-coherent sheaves} and \emph{$\fD$-modules} on $G$, both equipped with convolution.

\medskip

By definition, categories with a \emph{weak} $G$-action are comodules for $(\QCoh(G), m^*)$, whereas categories with a \emph{strong} $G$-action (or \emph{infinitesimally trivialized} action, or \emph{de Rham} action) are comodules for $(\Dmod(G), m^!)$. 
We occasionally denote them by $G \on{-}\mathbf{rep}^{weak}$ and $G \on{-}\mathbf{rep}$, respectively.
Here $m^*$ (resp., $m^!$) is the quasi-coherent (resp., $\fD$-module) pull-back along the multiplication $m: G \times G \to G$.

\medskip

Both $\QCoh(G)$ and $\Dmod(G)$ are canonically self-dual; under this duality, the coproducts appearing above get sent to the convolution products in the context of quasi-coherent sheaves or $\fD$-modules, respectively. As a consequence,
$$
G \on{-}\mathbf{rep}^{weak} \simeq (\qcoh(G), \star) \mmod \hspace{.8 cm} G \on{-}\mathbf{rep} \simeq (\Dmod(G), \star) \mmod,
$$
where both convolution products have been denoted by $\star$.

\medskip

In this paper, we privilege strong actions. The reason comes from our interest in the local Langlands correspondence, which explicitly involves strong actions of loop groups on categories. 
Hereafter, unless specified otherwise, by the term ``action" we shall mean ``strong action".

\medskip

The main source of examples of categorical $G$-representations comes from geometry: if $X$ is a scheme of finite type with a $G$-action, then $\Dmod(X)$ is a module category for $(\Dmod(G), \star)$, via push-forward along the action map $G \times X \to X$. In other words, $\Dmod(X)$ carries a $G$-action. Likewise, $\QCoh(X)$ carries a weak $G$-action.

\sssec{$\fD$-modules on loop groups}

As already mentioned, our goal is to study (strong) actions of $G \ppart$, the loop group of an affine algebraic group $G$, on categories.
The first thing to notice is that $G \ppart$ is not at all of finite type (unless $G$ is the trivial group): indeed, $G \ppart$ is infinite dimensional in two \virg{directions}, corresponding to the fact that the field of Laurent series $\kk \ppart$ \virg{goes off} to infinity in two directions. 

Nevertheless, we wish to define $G \ppart \rrep$ copying what we have done above for groups of finite type. Thus, all we need to do is to construct a category $\Dmod(G \ppart)$ and equip it with a convolution monoidal structure.
The details will be given in the main body of the text (Sect. \ref{ssec:D-mods-on-indpro-schemes}), where we discuss the theory of $\fD$-modules on ind-schemes of pro-finite type. 
Accepting for the time being that a good definition of $G \ppart \rrep$ has been given, we are ready to look at the local geometric Langlands conjecture.

\ssec{Local geometric Langlands duality} \label{ssec:local-geom-langlands}

The very goal of the local geometric Langlands correspondence is to study the $\infty$-category $G \ppart \rrep$ of categories equipped with a strong action of $G\ppart$, where here $G$ is a reductive group.

\medskip

To state the main conjecture, we need to recall the concept of \emph{category fibered over} a stack $\Y$. The correct definition of such notion is the topic of \cite{shvcat}; however, the following shortcut is enough for the purposes of this introduction, see also \cite{Sam-locsys}.
We say that $\C$ is fibered over $\Y$ if $\C$ is equipped with an action of the symmetric monoidal category $(\QCoh(\Y), \otimes)$.
The collection of categories fibered over $\Y$ naturally forms an $\infty$-category, which is, according to our earlier notation, $(\QCoh(\Y), \otimes) \mmod$.

\sssec{}

\nc{\LL}{\mathbb{L}}

The local geometric Langlands conjecture predicts the existence of an explicit equivalence
\begin{equation} \label{equiv:local-langlands}
\LL_G:
G\ppart \rrep 
\xto{\;\; \simeq \;\;}
\big(\QCoh(\LocSys_{\Gch} (\D^\times)), \otimes \big) \mmod,
\end{equation}
where $\LocSys_{\Gch} (\D^\times)$ is the stack of de Rham local systems on the punctured disk $\D^\times = \Spec(\kk \ppart)$, for the \emph{Langlands dual group} $\Gch$.


\medskip

The functor $\LL_G$ realizing the above equivalence is expected to be the operation of \emph{Whittaker invariants}, which is actually the protagonist of the present paper. We postpone its definition to Sect. \ref{ssec:LGL for general G}, as we first need to dwell a bit more on abstract nonsense and discuss the notions of categorical invariants and coinvariants.

\sssec{} \label{sssec:best-hope-LGL}

Let us warn the reader that the conjecture (\ref{equiv:local-langlands}) cannot be literally true as stated. The formulation given here might be called the \virg{best hope} form of the conjecture, in analogy with the same term used in the statement of the \emph{global geometric Langlands conjecture}, see  \cite[Sect. 1.1.1]{AG}.
In fact, as in the global case, it appears that the RHS of (\ref{equiv:local-langlands}) needs to be enlarged to match the LHS.

In this paper, we will not discuss any attempt to correct such discrepancy and, for the most part, use the \virg{best hope} as if it was correct.

\sssec{}

Besides the above issue, an even more substantial problem prevents from formulating the conjecture precisely. Namely, after having defined the candidate functor $\LL_G$ as an arrow
$$
G \ppart \rrep
\longrightarrow
\DGCat,
$$
it will not be at all clear that it factors through the forgetful functor $\QCoh(\LocSys_{\Gch} (\D^\times)) \mmod \to \DGCat$. 
In fact, this assertion has not been established yet, although \cite{FG2}, \cite{FG3}, \cite{quantum-langlands}, \cite{F} give evidence for it. We will not give comment on this problem any further, except for a brief remark in Sect. \ref{sssec:doubly whittaker}.

\ssec{Invariant and coinvariant categories}

Several standard operations with ordinary $G$-representations generalize immediately to the categorical context. Let us discuss the most important ones, in the simplest case where $G$ is of finite type.

\sssec{}

Given an ordinary $G$-representation $V$, we can take its invariants $V^G := \Hom_G( \kk, V)$ and coinvariants $V_G = \kk \otimes_G V$. 
In our categorical framework, we have the notions of weak invariants and coinvariants: for $\C \in G \on{-}\mathbf{rep}^{weak}$,
$$
\C^{G,w} := \Hom_{\QCoh(G)}(\Vect, \C), \, \, \, \, \C_{G,w} := \Vect \usotimes{\QCoh(G)}  \C;
$$
as well as strong invariant and coinvariants: for $\C \in G\on{-}\mathbf{rep}$,
$$
\C^{G} := \Hom_{\Dmod(G)}(\Vect, \C), \, \, \, \, \C_{G} := \Vect \usotimes{\Dmod(G)}  \C.
$$

\begin{example}
If $X$ is a $G$-scheme of finite type, then $\Dmod(X)^G \simeq \Dmod(X/G)$, the category of $\fD$-modules on the quotient stack $X/G$. This follows immediately from smooth descent for $\fD$-modules. It turns out that $\Dmod(X)_G \simeq \Dmod(X/G)$ as well, see Theorem \ref{thm:equiv-Ginv-coinv}. (Parallel statements hold for $\QCoh$ in place of $\fD$.)
\end{example}

\sssec{}

Let us now assume that $G$ is a vector group $A = \AA^n$ of dimension $n$, and use the Deligne-Fourier transform to offer an alternative point of view on $A$-invariants and $A$-coinvariants.
Recall that, by \cite{Laumon}, the Deligne-Fourier transform is an equivalence of monoidal categories
\begin{equation} \label{eqn:FD}
\FT: (\Dmod(A), \star) \xrightarrow{\, \, \, \, \simeq \, \, \, \, } (\Dmod(A^\vee), \stackrel ! \otimes).
\end{equation}
\begin{rem}
In analogy with the classical theory, such equivalence is obtained by a \emph{kernel $\fD$-module} on the product $A \times A^\vee$, the $\fD$-module $Q^!(\exp)$, where $Q: A \times A^\vee \to \GG_a$ is the duality pairing and $\exp$, the substitute of the ``exponential function", is the $\fD$-module encoded by the defining differential equation for $\exp$. Its homomorphism property corresponds to the isomorphism $m^! (\exp) \simeq \exp \boxtimes \exp$.
\end{rem}
Thus, the Fourier transform allows to view a category acted on by $A$ as a category fibered over $A^\vee$: in other words, we have an equivalence
\begin{equation} \label{eqn:FD-mod}
A\on{-\mathbf{rep}}= (\Dmod(A), \star)\mmod \xrightarrow{\, \, \, \, \simeq \, \, \, \, } (\Dmod(A^\vee), \stackrel ! \otimes) \mmod,
\end{equation}
which is the identity on the underlying DG categories.

\begin{rem}
In the main body of the paper, we extend (\ref{eqn:FD}) and thus (\ref{eqn:FD-mod}) to $A = \AA^n \ppart$ or, slightly more generally, to a Tate vector space.
\end{rem}

Invariants and coinvariants of $\C \in A \on{-\mathbf{rep}}$, correspond to \emph{fiber} and \emph{cofiber} of $\C$ over $0 \hookrightarrow A^\vee$, respectively.
The fiber at non-zero $\chi \in A^\vee$ is identified with the category of $(A, \chi)$-invariants of $\C$: objects $c \in \C$ for which the \emph{coaction} (dual to the action) is isomorphic to 
$$
\coact (c) \simeq \chi^! (\exp) \otimes c \, \in \Dmod(A) \otimes \C.
$$
By definition, the fiber (resp., cofiber) of $\C$ at $\chi$ is the category 
$$
\uprestr{\C}\chi := \Hom_{\Dmod(A^\vee)}(\Dmod(\{\chi\}), \C)
\hspace{.6cm}
\big(\text{resp., }\, \restr{\C} \chi:= \Dmod(\{\chi\}) \usotimes{\Dmod(A^\vee)} \C \big).
$$
More generally, one can define \emph{restriction} and \emph{corestriction} of $\C$ along any map of schemes $B \to A^\vee$.

\sssec{} \label{sssec:action by semi-direct product T A}

Assume now that $\C$ is acted on by a semi-direct product $G \ltimes A$, with $A$ as above. Then, the natural action of $G$ on $A^\vee$ yields a canonical identification of the invariant categories $\C^{A, \chi}$ and $\C^{A, \chi'}$ for any $\chi$ and $\chi'$ lying in the same $G$-orbit of $A^\vee$.

\ssec{Local geometric Langlands for $G=GL_2$}

Let us use the above ideas, properly adapted to the loop group case, to discuss local geometric Langlands for $G= GL_2$.

\sssec{}

Suppose we want to study $\C$, a category equipped with an action of $GL_2 \ppart$, 
via the Fourier transform. To start, we need to regard $\C$ as acted on by a (Tate) vector space. Such a vector space is readily available inside $GL_2 \ppart$: take $N \ppart$, the loop group of unipotent upper triangular matrices.

Slighly better, we will view $\C$ as acted on by the semidirect product $\GG_m \ppart \ltimes N \ppart$. (Note that $\GG_m \ltimes N$ is nothing but the mirabolic subgroup of $GL_2$. The mirabolic subgroup of $GL_n$, whose definition is reviewed later, will play a crucial role in this paper.)

\medskip

Observe that $\GG_a \ppart^\vee$ is self-dual via the residue pairing and that $\GG_m \ppart$ acts on it with two orbits.
Hence, according to Sect. \ref{sssec:action by semi-direct product T A}, there are only two characters to consider: the trivial one and any nonzero one of our choice. For the latter, it is natural to pick the residue character 
$$
\Res: \GG_a \ppart \to \GG_a, \, \, \, \, f(t) = \sum f_n t^n \mapsto f_{-1}.
$$

\sssec{}

Thus, we see that, for $\C \in GL_2 \ppart \rrep$, the invariant category
\begin{equation} \label{eqn:Whit-for-GL2}
\Whit(\C) := \C^{\,\GG_a \ppart, \Res},
\end{equation}
called the \emph{Whittaker invariant category of $\C$}, contains most of the information about the original $\C$. 

\sssec{}

In the rough formulation of the local Langlands conjecture, the functor $\LL_{GL_2}$ is set to be the above Whittaker invariant functor $\Whit$. Following up on the remark of Sect. \ref{sssec:best-hope-LGL}, it is now manifest that the conjecture cannot be true: the trivial $G \ppart$-representation $\Vect$ has Whittaker invariant zero.

\ssec{Local geometric Langlands for general $G$} \label{ssec:LGL for general G}

Let us move to the case of $G$ an arbitrary reductive group and adapt the above definition of $\Whit$. This will yield the functor $\LL_G$, which ought to realize the Langlands correspondence in general (modulo the provisos of Section \ref{sssec:best-hope-LGL}).

\sssec{}

The expression (\ref{eqn:Whit-for-GL2}) admits a generalization to an arbitrary reductive $G$ as follows. Let $N$ be a maximal unipotent subgroup of $G$ (for $G=GL_n$, this is the subgroup of unipotent triangular matrices).
It turns out that $N \ppart$  still admits a non-degenerate\footnote{i.e., nonzero on any simple root space} character $\chi$, given by \emph{the sum of residues}. To define it, denote by $\{\alpha_1, \ldots, \alpha_r \}$ the simple roots of $G$, thought of as maps $N \to \GG_a$, and let
\begin{equation} \label{eqn:chi}
\chi: N \ppart \to \GG_a \, \, \, \,  n(t) \mapsto \on{Res} \Big( \sum_{j=1}^r \alpha_j(n(t)) \Big).
\end{equation}
If $\C$ is acted upon by $G \ppart$, we define the Whittaker invariant category of $\C$ to be
$$
\Whit(\C) := \C ^{N \ppart, \chi}.
$$

\sssec{}

When $G$ has rank greater than one, the maximal unipotent subgroup $N \subset G$ is no longer abelian, hence the Fourier-Deligne transform along $N \ppart$ is not available. Consequently, it is not clear how conservative the operation $\C \squigto \Whit(\C)$ is. 
Nevertheless, the \virg{best hope} of the local geometric Langlands conjecture states that $\Whit$ is the functor $\LL_G$ of (\ref{ssec:local-geom-langlands}).

\begin{rem}

We refer to \cite{Outline} and \cite{Dario-ext-whit} for the crucial role of the Whittaker invariant categories in the context of global geometric Langlands. In particular, \cite{Dario-ext-whit} uses a strategy similar to the one of the present paper to show that the \emph{extended Whittaker coefficient} for $G = GL_n$ behaves as if $N$ was a vector group.

\end{rem}

\ssec{Whittaker invariants vs Whittaker coivariants}

Let us now get closer to discussing the main theorem of this paper.
For a $G \ppart$-representation $\C$, we have, along with the Whittaker invariant category, also the \emph{Whittaker coinvariant} category.
Thus, one could propose that the Langlands equivalence $\LL_G$ is instead the functor
$$
\C 
\rightsquigarrow
\C_{N \ppart, \chi},
$$
which attaches to $\C$ its coinvariant Whittaker category.
A priori, this would lead to a different local geometric Langlands correspondence. After Gaitsgory (\cite{quantum-langlands}), we conjecture (and prove for $G = GL_n$) that the categories $\C^{N \ppart, \chi}$ and $\C_{N \ppart, \chi}$ are equivalent.

\sssec{}

More precisely, for any category $\C$ acted on by $G \ppart$, we construct an explicit functor $\Theta: \C_{N \ppart, \chi} \to \C^{N \ppart, \chi}$ and repropose the following conjecture:

\begin{conj} \label{big-conj}
For $\C$ a category acted on by $G \ppart$, the functor $\Theta: \C_{N \ppart, \chi} \to \C^{N \ppart, \chi}$ is an equivalence.
\end{conj}

While the present paper was under the publication process, the above conjecture has been proven by Sam Raskin, see \cite{sam-W-algebras}.

\sssec{}

Note that the functor $\Theta$ goes from coinvariants to invariants: this is the most natural direction to define such a functor. Indeed, coinvariants are realized as a colimit, while invariants as a limit; in general, it is much easier to define a functor from a colimit to a limit than vice versa.


\sssec{}

Let us describe some consequences of the above conjecture.
On one hand, it is tautological that $\C \squigto \Whit(\C)$ commutes with limits, but not that it commutes with colimits. On the other hand, it is tautological that $\C \squigto \C_{N \ppart, \chi}$ commutes with colimits, but not with limits.
Hence, the conjecture implies that Whittaker invariants and coinvariants commute with both limits and colimits. In particular, the functor $\LL_G$ is the functor 
$$
\C \squigto \C \usotimes{\Dmod(G \ppart)} \Dmod(G \ppart)^{N \ppart, \chi}
$$
of tensoring up with the \emph{universal Langlands category} $\Dmod(G \ppart)^{N \ppart, \chi}$.

\sssec{} \label{sssec:doubly whittaker}

Thus, the statement of the local geometric Langlands conjecture can be rephrased as follows. The $G \ppart$-module category $\Dmod(G \ppart)^{N \ppart, \chi}$ admits a compatible right action of $\QCoh(\LocSys_{\Gch} (\D^\times))$ and it yields a Morita equivalence between $(\Dmod(G \ppart), \star)$ and $(\QCoh(\LocSys_{\Gch} (\D^\times)), \otimes)$.

\medskip

In particular, the fully faithfulness of the inverse of $\LL_G$ would imply:

\begin{conj}
There is a monoidal equivalence
$$
\QCoh(\LocSys_{\Gch} (\D^\times)) 
\simeq
{}^{N \ppart, -\chi} \Dmod(G \ppart)^{N \ppart,\chi},
$$
where the RHS is the \virg{doubly} Whittaker invariant category of $\Dmod(G \ppart)$.
\end{conj}

The problem mentioned in Sect. \ref{sssec:best-hope-LGL} should not affect this conlecture, which is therefore expected to be true without any further tweaking.

\ssec{Overview of the results}

In this paper we prove a refined version of Conjecture \ref{conj:Tequiv} for $G= GL_n$. Our main theorem reads:

\begin{thm} \label{thm:intro}
Let 
$$
P 
:=
\left( \begin{array}{cc}
GL_{n-1} & \star  \\
0 & 1 \end{array} \right)
 \subseteq GL_n
$$ 
be the mirabolic subgroup of $GL_n$. For any category $\C$ equipped with a $P \ppart$-action, the functor $\Theta: \C_{N \ppart, \chi} \to \C^{N \ppart, \chi}$ is an equivalence.
\end{thm}

\sssec{}  \label{sssec:strategy-outline-for-main-theorem}

To prove this theorem, we fix an integer $k \geq 1$ and make use of an explicit group-scheme $\bH_k \subset P \ppart$. For $n=2$ and $n=3$, the group in question looks like
$$
\bH_k = \left( \begin{array}{cc}
1+t^k \O & t^{-k}\O  \\
0 & 1 \end{array} \right),
\, \, \,
\bH_k = \left( \begin{array}{ccc}
1+t^k\O & t^{-k}\O & t^{-2k}\O \\
t^{2k}\O & 1+t^k\O & t^{-k}\O \\
0 & 0 & 1 \end{array} \right),
$$
and the generalization to any $n$ is straightforward ($\O$ denotes the ring of formal Taylor series).
The sum of the residues of the entries in the over-diagonal yields a character on $\bH_k$ that we continue to denote $\chi$.

\medskip

The idea is to relate $(\bN, \chi)$-invariants with coinvariants on $\C$ by passing through $(\bH_k, \chi)$-invariants on $\C$. Namely, we will produce natural equivalences
$$
\C_{\bN, \chi} \longleftarrow \C^{\bH_k, \chi} \longleftarrow \C^{\bN, \chi}
$$
and observe that the composition of the inverse functors is our $\Theta$, up to a cohomological shift that depends on $k$.

\medskip

In the final part of this introduction, we will sketch the construction of the equivalence $\C^{\bN, \chi} \simeq \C^{\bH_k, \psi}$.

\begin{rem}
A parallel argument will yield the equivalence $\C_{\bN, \chi} \simeq \C_{\bH_k, \chi}$. 
Having established that, the equivalence $\C_{\bN, \chi} \simeq \C^{\bH_k, \chi}$ arises thanks Theorem \ref{thm:inv=coinv}. The latter states that, for any pro-unipotent group $H$, there is a natural equivalence $\C^{H, \mu} \simeq \C_{H, \mu}$.
\end{rem}

\sssec{}

To construct $\C^{\bN, \chi} \simeq \C^{\bH_k, \psi}$, let us first treat the case of $G= GL_2$ in detail. We proceed in stages.

\begin{itemize}

\item

There is a canonical equivalence
$$
\C^{\bH_k, \chi} \simeq  \Big( \C^{t^{-k}\O, \Res} \Big)^{1 + t^k \O}.
$$
This is nothing but a tautological result on actions by semi-direct groups: if $\C$ is acted on by $K \ltimes H$, then $\C^{K \ltimes H} \simeq (\C^H)^K$, and similarly in the presence of characters.
See Lemma \ref{lem:semidirect-pro}.

\medskip

\item

We use the Fourier transform to obtain
$$
\Big(  \C^{t^{-k}\O, \Res} \Big)^{1+ t^k \O} \simeq \Big( \! \uprestr \C {1+t^k \O} \Big)^{1 + t^k \O}.
$$
The general paradigm, which we are applying here to the inclusion $t^{-k} \O \subset \GG_a \ppart$, is the following.
Let $W \subseteq V$ an inclusion of Tate vector spaces, and $\chi \in V^\vee$ a character. Under the Fourier transform, $\C ^{W, \chi}$ is equivalent to the restriction $\uprestr \C{\chi + W^\perp}$, where $W^\perp$ is the annihilator of $W$ inside $V^\vee$.
See Proposition \ref{prop:FT_char_subspace_indpro}.

\medskip

\item

Next, using the regular action of $1+t^k \O$ on itself, we obtain
$$
\Big( \! \uprestr \C {1+t^k \O} \Big)^{1 + t^k \O} 
\simeq 
 \uprestr \C 1.
$$

\medskip

\item

By the Fourier transform again, we have $\uprestr \C 1 \simeq \C^{\bN, \Res}$.
\end{itemize}

\medskip


The proof for $G=GL_n$ with $n > 2$ uses the same logic, combined with induction on $n$: indeed, $N$ is the semi-direct product $N' \ltimes \AA^{n-1}$, where $N'$ refers to the maximal unipotent subgroup of $GL_{n-1}$.

The only slight difference will be in the third step in the list above. Namely, for $GL_2$ we have encountered a simply transitive action: the action of $1+t^k \O$ on itself. In general, we will encounter actions that are transitive, but not simply transitive, and we will use the following paradigm, see Proposition \ref{prop:sigma_equiv}.

If $\C$ fibers over $X/K$ and $K$ acts on $X$ transitively, then $S=\Stab(x \in X)$ continues to act on $\uprestr \C x$ and $\C^K \simeq (\uprestr \C x)^S$.
%

\ssec{Notation and detailed contents}


Let us explain how this paper is organized.

\sssec{} 

In Sect. \ref{SEC:action_fin_dim} we discuss the foundations of weak and strong group actions on categories, in the finite dimensional situation: we define Hopf monoidal categories, invariants and coinvariants, the Harish-Chandra monoidal category and analyze the relation between weak and strong (co)invariants.

\sssec{}

In Sect. \ref{SEC:Dmod-infinite}, we discuss some foundations of the theory of $\fD$-modules on schemes (and ind-schemes) of pro-finite type. The main examples of such are $\bG$, $\bN$ and variations thereof. There are two categories of $\fD$-modules on $\bG$, dual to each other. The first, $\Dmod^*(\bG)$, carries a convolution monoidal structure; its dual $\Dmod^!(\bG)$ is hence comonoidal and also carries a symmetric monoidal structure via the diagonal. It turns out that $\Dmod^*(\bG) \simeq \Dmod^!(\bG)$, as plain DG categories.

\sssec{}

We proceed in Sect. \ref{SEC:grp-actions} to define loop group actions on categories and the concept of invariants and coinvariants. Since $\bN$ is exhausted by its compact open group sub-schemes, we analyze group actions by pro-unipotent group schemes in great detail. For instance, we define and study natural functors among the original category, the invariant category and the coinvariant category. We show that the latter two are equivalent.

Next, we take up Whittaker actions of $\bN$ on categories: this is a special case of the above theory that accounts for the presence of the character $\chi: \bN \to \GG_a$. For any category acted on by $\bN$, we construct a functor (denoted by $\Theta$, as above) from Whittaker coinvariants to Whittaker invariants. 

\sssec{}

We discuss the abelian theory in Sect. \ref{SEC:loop_vectors}. We first review the theory of the Fourier-Deligne transform for finite dimensional vector spaces (in schemes) and then extend it to $\fD$-modules on $\AA^n \ppart$ (more generally, to $\fD$-modules on a Tate vector space). We prove that it still gives a monoidal equivalence. Finally, we re-interpret the concepts of the previous sections (invariants, coinvariants, averaging) in ``Fourier-transformed" terms.

\sssec{}

In Sect. \ref{SEC:trans_act}, we discuss categories fibering over a $K$-space $X$ and acted on by the group $K$ in a compatible fashion. In this situation, we study how the operations of restriction of $\C$ to $Y \subseteq X$ and taking $K$-invariants interact. Our main result there is Proposition \ref{prop:sigma_equiv}.

\sssec{}

Finally, in Sect. \ref{SEC:GL_n}, we take up the proof of Theorem \ref{thm:intro}. We discuss some combinatorics of $GL_n$ and define some group schemes of $GL_n \ppart$ that will play a central role. Our proofs are on induction on $n$ and rely heavily on the theory of all previous sections.

\ssec*{Acknowledgements}

I am immensely indebted to Dennis Gaitsgory for proposing the problem and generously teaching me most of the techniques to solve it. Many of the ideas described in this paper come directly from his suggestions.
I wish to express my deepest gratitude to Constantin Teleman for several years of patient explanations at UC Berkeley: it was him who introduced me to higher representation theory. 
It is a pleasure to thank Edward Frenkel for his invaluable help and the impact he had on my thinking about the geometric Langlands program. 
I am much obliged to Sam Raskin for several conversations that very strongly influenced the development of this paper.
I also benefited enormously from discussions with Jonathan Barlev, David Ben-Zvi and David Nadler. 

\section{Actions by groups of finite type} \label{SEC:action_fin_dim}

In this section we cover some background on categorical actions of affine groups of finite type. Let $G$ be such a group. We first show that $\Dmod(G)$, as well as $\QCoh(G)$, admits two dual Hopf structures. Next, we analyze the difference between strong and weak invariants: this is controlled by the Harish-Chandra monoidal category $\HC$. Lastly, whenever $G$ is equipped with an additive character $\mu: G \to \GG_a$, we discuss the notion of $\mu$-twisted $G$-actions.

\ssec{Hopf algebras and crossed products} \label{ssec:Hopf}

Given a symmetric monoidal $\infty$-category $(\C,\otimes)$, the $\infty$-category $\ms{Coalg}(\C)$ of its coalgebra objects inherits a symmetric monoidal structure, compatible with the forgetful functor $\Coalg(\C) \to \C$ (see \citep{L1}).
A \emph{Hopf algebra} in $\C$ is, by definition, an object in $\Alg(\Coalg(\C))$. The definition is known to be symmetric under the switch $\Alg \leftrightarrow \Coalg$, so that
\begin{equation} \label{eqn:def-Hopf}
\ms{HopfAlg}(\C) := \Alg(\Coalg(\C)) \simeq \Coalg(\Alg(\C)).
\end{equation}
Consequently, $\ms{HopfAlg}(\C)^\mathit{dualizable}$ (the full subcategory of $\ms{HopfAlg}(\C)$ spanned by those Hopf algebras that are dualizable as objects of $\C$) is closed under duality.

For $\C = (\DGCat, \otimes)$, we obtain the concept of \emph{Hopf monoidal} category (or just Hopf category, for short). 

\sssec{}

Any group object $(G, m)$ in $\C = \on{Set}$ (or $\Sch$, $\IndSch$, $\IndSch^{pro}$ etc.) is a Hopf algebra in $\C$, with multiplication being $m$ and comultiplication being $\Delta: G \to G \times G$. The compatibility between the two sctructures follows at once from commutativity of the diagram
\begin{gather}
\xy
(40,0)*+{ G \times G, }="00";
(0,0)*+{ G }="10";
(40,15)*+{ G \times G \times G \times G  }="01";
(00,15)*+{ G \times G  }="11";
{\ar@{<-}_{m_{13} \times m_{24} } "00";"01"};
{\ar@{<-}_{ \Delta  } "00";"10"};
{\ar@{<-}_{\Delta \times \Delta \,\, \, } "01";"11"};
{\ar@{<-}_{ m} "10";"11"};
\endxy
\end{gather} 
which shows that $m$ is a morphism of coalgebras (and that $\Delta$ is a morphism of algebras).

\begin{example} \label{example: Dmod(G)-Hopf}

For a group scheme $G$ of finite type, we claim that $\H = \Dmod(G)$ can be naturally endowed with the structure of a Hopf category. Indeed, as the functor $\Dmod: (\Sch^{\fty})^\op \to \DGCat$ is symmetric monoidal, it maps algebras (resp., coalgebras) in $\Sch^{\fty}$ to comonoidal (resp., monoidal) categories. It follows that $\Dmod(G)$ is Hopf with comultiplication induced by $m^!$ and multiplication by $\Delta^!$. For clarity, we denote this Hopf category as $(\Dmod(G), \Delta^!, m^!)$. 
\end{example}

\begin{example}
The self-duality $\Dmod(G)^\vee \simeq \Dmod(G)$ transforms $!$-pull-backs into $*$-push-forwards. Hence, $(\Dmod(G), m_*, \Delta_*)$ is also a Hopf category.
\end{example}

\sssec{}

The coalgebra structure on a Hopf category $\H$ allows to form the $\infty$-category $\H\ccomod := \on{Comod}_{\H}(\DGCat)$ of comodules categories for $\H$. The rest of the structure endows $\H \ccomod$ with a monoidal structure compatible with the tensor product of the underlying DG categories: informally, given $\C, \, \E \in \H \ccomod$, their product $\C \otimes \E$ has the following $\H$-comodule structure
$$
\C \otimes \E \xrightarrow{\coact_\C \otimes \coact_\E \, } \C \otimes \E \otimes \H \otimes \H \xrightarrow{ \on{mult_\H} \,} (\C \otimes \E) \otimes \H.
$$
Hence, we can consider algebra objects in $\H \ccomod$, that is, monoidal categories with a compatible coaction of $\H$. 

\sssec{}

Let $\H$ be a Hopf category, which is dualizable as a plain category. Then, as pointed out before, $\H^\vee$ (the dual of $\H$ as a plain category) is naturally a Hopf category. 
Given an object $\B \in \on{Alg}(\H^\vee \ccomod)$, we shall form the monoidal category $\H \ltimes \B$, called the \emph{crossed product algebra} of $\H$ and $\B$. At first approximation, $\H \ltimes \B$ can be described as follows: its underlying category is simply $\B \otimes \H$ and the multiplication is given by
$$
\B \otimes \H \otimes (\B \otimes \H)
\xrightarrow{\act_{\H \curvearrowright \B \otimes \H}}
\B \otimes \B \otimes \H
\xrightarrow{ m_\B}
\B \otimes \H.
$$
To be more precise, consider the obvious adjuction
$$
\mathsf{free}: \DGCat \rightleftarrows \B \mmod (\H^\vee \ccomod): \mathsf{forget},
$$
which satisfies the hypotheses of the monadic Barr-Beck theorem. Thus, $\B \mmod (\H^\vee \ccomod)$ concides with the category of modules for a monad whose underlying functor is $\C \mapsto \B \otimes (\H^\vee)^\vee \otimes \C$.
The monad structure endows $\B \otimes \H$ with an algebra structure, which is tautologically the one displayed above.

\begin{example}

Given a group $G$ as above, consider the Hopf category $\H = (\Dmod(G), m_*, \Delta_*)$.  Let $X$ an scheme (of finite type) acted upon by $G$. We claim that $\B= \Dmod(X)$, equipped with the point-wise tensor product, belongs to $\on{Alg}(\H^\vee \ccomod)$. In fact, the datum of the action $G \times X \xrightarrow{\, \, \act \, \,} X$ yields the coaction of $\H^\vee$ on $\Dmod(X)$, and the required compatibility arises from the commutative diagram
\begin{gather} \label{diag:GactsX}
\xy
(40,15)*+{ X }="00";
(0,15)*+{ G \times X }="10";
(40,0)*+{ X \times X.  }="01";
(00,0)*+{ G \times G \times X \times X  }="11";
{\ar@{->}_{\Delta_X } "00";"01"};
{\ar@{<-}_{ \act  } "00";"10"};
{\ar@{<-}_{\, \, \,\, \act \times \act } "01";"11"};
{\ar@{->}_{ \Delta_{G} \times \Delta_X} "10";"11"};
\endxy
\end{gather} 
Thus, we have a well-defined category $\Dmod(G) \ltimes \Dmod(X)$. This example, or rather its generalization to the ind-pro-setting, will be of importance later.

\end{example}

\sssec{}

A \emph{category over a quotient stack} $X/G$ (for us, always with connection) is by definition an object of $\ShvCat( (X/G)_\dR)$; see \cite{shvcat}.
This notion can be alternatively expressed via the crossed product $\Dmod(G) \ltimes \Dmod(X)$. Indeed:

\begin{prop} \label{prop:ShvCat(X/G)}
With the notation above, recall that $(\Dmod(X), \otimes)$ is an algebra object in $\Dmod(G) \ccomod$. The functor $\bGamma(X_\dR, -): \ShvCat( (X/G)_\dR) \to \DGCat$ upgrades to an equivalence
$$
\ShvCat( (X/G)_\dR) \xrightarrow{\, \, \simeq \, \, } \Dmod(G) \ltimes \Dmod(X) \mmod.
$$
\end{prop}

\begin{proof}
The tautological map $q: X \to X/G$ yields the adjunction
$$
\cores_q:
\ShvCat((X/G)_\dR)
\rightleftarrows
\ShvCat(X_\dR):
\coind_q.
$$
It is clear that $\cores_q$ is conservative and that it commutes with all limits, whence it is \emph{comonadic}. It follows that $\ShvCat((X/G)_\dR)$ is equivalent to the $\infty$-category of comodules for the comonad $\cores_q \circ \coind_q$ acting on $\ShvCat(X_\dR)$. By the $1$-affineness of $X_\dR$, we have $\ShvCat(X_\dR) \simeq \Dmod(X) \mmod$. Under this equivalence, the comonad in question is given by the coalgebra $\Dmod(G) \otimes \Dmod(X) \in \Coalg(\Dmod(X) \mmod)$, which is dual in $\Dmod(X) \mmod$ to the algebra $\Dmod(G) \ltimes \Dmod(X) \in \Alg(\Dmod(X) \mmod)$. Thus, 
$$
\ShvCat((X/G)_\dR) 
\simeq
(\Dmod(G) \ltimes \Dmod(X)) \mod(\Dmod(X) \mmod)
$$
and the latter is tautologically $\Dmod(G) \ltimes \Dmod(X) \mmod$, as claimed.
\end{proof}

\ssec{Groups actions on categories}
In this section, we officially define the notion of strong $G$-action on a category $\C$. For such $\C$, we also define the $G$-invariant and $G$-coinvariant categories $\C^G$ and $\C_G$. 
The corresponding notions in the weak context (weak actions, weak (co)invariants, and the natural functor $\C_{G,w} \to \C^{G,w}$) are defined in a completely analogous manner. 

We provide several examples of categories equipped with strong and work $G$-actions. The most interesting one is the strong action of $G$ on the category of modules over its Lie algebra.

\sssec{}

We say that $G$ acts \emph{strongly} on a (co-complete DG) category $\C$ if the latter is equipped with the structure of a comodule category for $(\Dmod(G), m^!)$. The same datum is equivalent to $\C$ possessing an action of the monoidal category $(\Dmod(G), m_*)$.

\medskip

The totality of categories equipped with a strong $G$-action forms an $\infty$-category, denoted by
$$
G \on{-} \mathbf{rep} := (\Dmod(G), \star) \mmod \simeq (\Dmod(G), m^!) \ccomod.
$$
The theory of \citep{L1} guarantees that $G \on{-} \mathbf{rep}$ admits limits and colimits, both computed object-wise.
By the previous section, $G \on{-} \mathbf{rep} := (\Dmod(G), \star)$ is monoidal: in terms of the coaction, for $\C, \E \in G \on{-} \mathbf{rep}$, the coaction of $\Dmod(G)$ on $\C \otimes \E$ given by 
\begin{equation} \label{eqn:coact on a tensor}
c \otimes e \mapsto \Delta_{G}^! \big(  \coact(c) \boxtimes \coact(e) \big).
\end{equation}

\sssec{}

An example of a category with a $G$-action is the \emph{regular representation}: $\Dmod(G)$, considered as a module over itself. In analogy with this, we sometimes denote the action map $\Dmod(G) \otimes \C \to \C$ by the convolution symbol, $\star$.

\medskip

Another example is the \emph{trivial representation} $\Vect$, the category of complexes of $\kk$-vector spaces, endowed with the (left, as well as right) $G$-action specified by the monoidal functor $\Gamma_\dR : \Dmod(G) \to \Vect$.
More generally, we say that $G$ acts on $\C$ \emph{trivially} if the action $\Dmod(G) \otimes \C \to \C$ is given by $M \otimes c \mapsto \Gamma_\dR(G, M) \otimes c$.

\sssec{}

For $\C \in G \rrep$, we define its (strong) \emph{coinvariant} and \emph{invariant} categories as
$$
\C_{G} := \Vect \usotimes{\Dmod(G)} \C,
\hspace{.7cm}
\C^{G} := \Hom_{\Dmod(G)} ( \Vect,  \C ).
$$
They come with tautological functors $\pr_G: \C \to \C_G$ and $\oblv_G: \C^G \to \C$.

\medskip

\begin{conv}
What we have treated so far is the concept of \emph{left} $G$-action. Right actions are defined in the obvious way. 
Whenever the $G$-action on $\C$ is clear from the context, we write $\C^G$ for the invariant category (regardless of whether the $G$-action is left or right). 
On the contrary, if it is important to distinguish between right and left $G$-action (in the case $\C$ is equipped with both), we will denote by $^G \C$ and $\C^G$ the left and right invariant categories, respectively. The same conventions hold for the coinvariant categories.
\end{conv}

\sssec{} \label{sssec:paradigm_dual actions}

If $\C \in G \rrep$ is dualizable as a plain category, its dual $\C^\vee := \Hom(\C, \Vect)$ inherits a \emph{right} action of $G$, described informally by
$$
\C^\vee \otimes \Dmod(G) \to \C^\vee, \, \, \, \, \phi(-) \otimes F \mapsto \phi(F \star - ).
$$
With this structure,
\begin{lem} \label{lem:duality-inv-coinv}
If $\C$ and ${}_G\C$ are dualizable as DG categories, then $({}_G \C)^\vee \simeq (\C^\vee)^{G}$. Under this equivalence, the tautological functors $\oblv^G: (\C ^\vee)^G \to \C^\vee$ and $\pr_G: \C \to {}_G\C$ are dual to each other.
\end{lem}

\begin{proof}
This is immediate from the bar realizations of ${}_G \C$ and $(\C^\vee)^G$. See formulas (\ref{eqn:simplicial_cat}) and (\ref{eqn:cosimplicial_cat}).
\end{proof}

\sssec{}

To obtain the concept of \emph{weak} action, one changes $\fD$ with $\QCoh$ and consider the push-forwards and pull-back functors of quasi-coherent sheaves. Details are left to the reader. For $\C$ a weak $G$-representation, we shall denote by $\C^{G,w}$ and $\C_{G,w}$ its weak $G$-invariant and coinvariant categories.

\medskip

Note that strong $G$-actions can be thought of as weak actions by the group prestack $G_\dR$. For this, we are using the realization of $\fD$-modules as \emph{left crystals}: $\Dmod(G) = \QCoh(G_\dR)$. 
In particular, we see that $\QCoh(G)$ acts on $\Dmod(G)$ via the monoidal functor $\ind_L: (\QCoh(G), \star) \to (\Dmod(G), \star)$, left adjoint (as well as dual) to the forgetful functor $\oblv_L = \q^*: \Dmod(G)  \to \QCoh(G)$. To prove $\ind_L$ is in fact monoidal, recall that $\oblv_L$ intertwines quasi-coherent with de Rham pull-backs, so that by duality $\ind_L$ intertwines the corresponding push-forward functors.

\sssec{}

The action of $G$ on a scheme $X$ of finite type induces a weak (resp., strong) action of $G$ on $\qcoh(X)$ (resp., $\Dmod(X)$).
More generally, the action of $G$ (resp., $G_\dR$) on an arbitrary prestack $\Y$ gives rise to a weak (resp., strong) $G$-action on $\QCoh(\Y)$. Tautologically, the invariant category is 
$$
\QCoh(\Y)^{G,w} \simeq \QCoh(\Y/G), \text{        (resp., } \QCoh(\Y)^{G,s} \simeq \QCoh(\Y/G_\dR) \text{)}, 
$$
where the quotient is simply the geometric realization of the Cech cosimplicial prestack:\footcite{In other words, we are considering the \emph{prestack quotient}. Recall, however, that $\QCoh(-)$ and $\Dmod(-)$ are insensitive to the operation of sheafification in the flat topology.}
We shall next discuss a fundamental example of this situation, recovering the adjoint action of $G$ on its Lie algebra $\fg$.

\sssec{}

Let $\wh G := \{1\} \times_{G_\dR} G$ denote the formal group of $G$ at $1 \in G$. The left $G_\dR$-action on $ G_\dR / G \simeq \pt/  \wh G $ yields a strong $G$-action on $\QCoh(\pt/ \wh G)$.

We claim that this action is familiar. Let $\fg \mod$ is the category of representations of the Lie algebra of $G$. There is an equivalence $\QCoh(\pt/ \wh G) \simeq \fg \mod$, under which the strong $G$-action just described corresponds to the adjoint action of $G$ on $\fg$.
This follows from the equivalence between Lie algebras and formal groups, which holds over a field of characteristic zero. The correspondence associates to a Lie algebra $\fg$ the formal group $\on{Spf}(U(\fg)^\vee)$ and it is well-known that $U(\fg)$ and $\O(\wh G)$ are dual Hopf algebras. See \cite[Chapters IV.2, IV.3]{GR_corresp} for a thorough treatment.

We will not use any Lie algebra theory in the present paper, and by the symbol ``$\fg \mod$" we understand the category $\QCoh(\pt/ \wh G)$. Clearly,
$$
(\fg \mod)^{G,s} \simeq \QCoh(G \backslash G_\dR)^{G_\dR} \simeq \QCoh( G \backslash G_\dR / G_\dR)   \simeq \Rep(G),
$$
while 
$$
(\fg \mod)^{G,w} \simeq \QCoh( G \backslash G_\dR / G) 
$$ 
is the \emph{Harish-Chandra} category to be studied and used in the next section.

\ssec{Invariants vs coinvariants}

Our task in this section is to show that invariants and coinvariants are naturally identified, both in the strong and weak context. 
Indeed, there are natural functors $\theta_{G,w}: \C_{G,w} \to \C^{G,w}$ and $\theta_G: \C_G \to \C^G$, which will be shown to be equivalences.

In the weak context, the assertion follows easily from the $1$-affineness of $BG$. 
The assertion in the strong context will be reduced to the one for the weak context using the \emph{rigid} monoidal category $\HC$ of Harish-Chandra bimodules, see Theorem \ref{thm:equiv-Ginv-coinv}.

\sssec{} \label{paradigm}
We will use the following general framework. 

Let $\A$ be a monoidal category and $\C$ a left $\A$-module. Denote by $\A^\rev$ the same category $\A$ with the reversed monoidal structure, so that $\A^\rev \mmod$ is the $\infty$-category of right $\A$-modules.
Consider the functors
$$
- \usotimes\A \C : \A^\rev \mmod \to \DGCat, 
\hspace{.6cm}
\Hom_\A (-, \C): (\A \mmod)^\op \to \DGCat.
$$
At the level of morphisms, we denote them by $\phi \rightsquigarrow \phi_\C$ and $\phi \rightsquigarrow \phi^\C$, respectively.

\medskip

The following fact will be repeatedly used throughout the text. Let $\phi: \D \rightleftarrows \E: \psi$ be mutually adjoint functors in $\A^\rev \mmod$. Then, the functors
$$
\phi_\C: \D \usotimes\A \C \rightleftarrows \E \usotimes\A \C: \psi_\C
$$ 
are also mutually adjoint in $\DGCat$. Similarly, if $\phi: \D \rightleftarrows \E: \psi$ is an adjunction in $\A\mmod$, so is
$$
\psi^\C: \Hom_\A( \E,\C) \rightleftarrows \Hom_\A( \D,\C)  : \phi^\C.
$$ 
Both these assertions are clear: e.g., for the first case, the reason is that the unit and counit of an adjunction survive tensoring up.

\begin{lem}
Let $f: X_1 \to X_2$ be a $G$-equivariant map of schemes of finite type equipped with a $G$-action. Then $f_*: \Dmod(X_1) \to \Dmod(X_2)$ and $f^*: \Dmod(X_2) \to \Dmod(X_1)$ (whenever the latter is defined) are both $G$-equivariant.
\end{lem}

\begin{proof}
The first functor is clearly compatible with the $(\Dmod(G), \star)$-action. Thanks to smoothness of the action maps $G \times X_i \to X_i$, the second functor is clearly compatible with the $(\Dmod(G), m^!)$-coaction.
\end{proof}

\sssec{}

In particular, let us apply the above paradigm to the $G$-equivariant adjunction $p^*: \Vect \rightleftarrows \Dmod(G): p_*$. We obtain the two adjunctions
$$
(\pr_G)^L = (p^*)_\C : \C_G = \Vect \usotimes{\Dmod(G)} \C \rightleftarrows \C : (p_*)_\C \simeq \pr_G
$$
$$
\oblv^G \simeq (p_*)^\C : \C^G \rightleftarrows \C: (p^*)^\C =: \Av_*^G.
$$
Tautologically, $\oblv^G$ is conservative and $\pr_G$ is essentially surjective.

\sssec{}

Define the \emph{constant sheaf} of $G$ to be $k_G := p^*(\kk) \in \Dmod(G)$: the $\fD$-module corepresenting de Rham cohomology $\Gamma_\dR := p_*$.

\begin{cor} \label{cor:k_G-induces}
For $\C \in G \rrep$, the action by $k_G$ induces a functor $\theta_G : \C_G \to \C^G$.
Similarly, for $\C \in G \rrep^{weak}$, the action by $\O_G$ induces a functor $\theta_{G, w}: \C_{G,w} \to \C^{G,w}$.
\end{cor}

In the next paragraphs, we will prove that $\theta_{G,w}$ is an equivalence. The analogous assertion for $\theta_G$ is the content of Theorem \ref{thm:equiv-Ginv-coinv}.

\sssec{}

Clearly, the weak (co)invariant category of any weak $G$-representation admits an action of $\Hom_{\QCoh(G)} (\Vect, \Vect) \simeq \Rep(G)$. Thus, there is an adjunction
\begin{equation} \label{adj:RepG-QCoh(G)}
{\rec} : (\Rep(G), \otimes) \mmod \rightleftarrows (\QCoh(G), \star) \mmod: {\inv}^{G,w},
\end{equation}
where $\inv^{G,w}$ is the functor of weak $G$-invariants and ${\rec}$, the so-called ``reconstruction" functor, sends $\E$ to $\Vect \otimes_{\Rep(G)} \E$.
\begin{thm}[Gaitsgory-Lurie] \label{thm:Gaitsgory-Lurie}
These two adjoint functors are mutually inverse equivalences of categories. 
\end{thm}

\begin{proof}
Combine \cite[Theorem 2.2.2]{shvcat} (this is the statement that $BG$ is $1$-affine) with the discussion in \cite[Sect. 10.2]{shvcat}.
\end{proof}

\begin{cor} \label{cor:theta.equiv.for.weak}
For $\C \in G \rrep^{weak}$, the functor $\theta_{G,w}: \C_{G,w} \to \C^{G,w}$ is an equivalence.
\end{cor}

\begin{proof}
Consider $\theta_{G,w}$ as a continuous functor $G \rrep^{weak} \to \DGCat^{\Delta^1}$. By the above theorem,  the operation $\C \rightsquigarrow \C^{G,w}$ commutes with colimits. Thus, we may assume without loss of generality that $\C = \QCoh(G)$.
In this case, under the canonical equivalences $\C_G \simeq \Vect$ and $\C^G \simeq \Vect$, it is immediate to verify that $\theta_{G,w}$ goes over to the identity functor of $\Vect$.
\end{proof}

\sssec{}

Let us now show how to recover strong invariants from weak invariants. 
If $\C$ admits a strong left action of $G$, its \emph{weak} invariant and coinvariant categories ${}^{G,w} \C$ and ${}_{G,w} \C$ can be expressed as
$$
{}^{G,w} \C \simeq \Hom_{\Dmod(G)} \big(
\Dmod(G)_{G,w}, \C
\big)
\hspace{.4cm}
{}_{G,w} \C
\simeq 
{}_{G,w}\Dmod(G) \usotimes{\Dmod(G)} \C.
$$
By Corollary \ref{cor:theta.equiv.for.weak}, we obtain ${}_{G,w}\Dmod(G) \simeq {}^{G,w}\Dmod(G) \simeq \fg \mod$. Thus, the weak invariant and coinvariant categories of $\C$ both possess an evident left action of the monoidal category 
$$
\HC := \Hom_{\Dmod(G)} (\fg \mod, \fg \mod).
$$

\begin{lem}
The left $\Dmod(G)$-module $\fg \mod$ is self-dual.\footnote{More precisely, the left $\Dmod(G)$-module $\Dmod(G)^{G,w}$ is dualizable with dual the right $\Dmod(G)$-module ${}^{G,w}\Dmod(G)$.}
\end{lem}

\begin{proof}
We use some of the theory of $\IndCoh$ on formal completions, see \cite[Chapter III]{GR_corresp}. 
First, we identify 
$$
{}^{G,w}\Dmod(G) \usotimes{\Dmod(G)} \Dmod(G)^{G,w}
\simeq
{}^{G,w}\Dmod(G)^{G,w}
\simeq
\QCoh(G \backslash G_\dR /G).
$$
Next, we use formal smoothness of $G \backslash G_\dR /G$ to further identify (via the so-called functor $\Upsilon$) 
$$
\QCoh(G \backslash G_\dR /G)
\xto{\simeq}
\ICoh(G \backslash G_\dR /G).
$$
In these terms, our coevaluation is the functor $\coev: \Vect \to {}^{G,w}\Dmod(G) \usotimes{\Dmod(G)} \Dmod(G)^{G,w}$ is given by the object 
$$
\delta_*^{\ICoh}(\omega_{BG}),
$$
where $\delta: BG \to G \backslash G_\dR /G$ is the obvious map and $\omega_{BG} \in \ICoh(BG)$ is the dualizing ind-coherent sheaf.
Similarly, we identify 
$$
\Dmod(G)^{G,w}
\otimes
{}^{G,w}\Dmod(G)
\simeq
\ICoh(G_\dR/G \times G \backslash G_\dR).
$$
In these terms, the evaluation $\ev: \Dmod(G)^{G,w}
\otimes
{}^{G,w}\Dmod(G) \to \Dmod(G)$ is given by the $\ICoh$-pull-push along the correspondence
$$
G_\dR/G \times G \backslash G_\dR
\longleftarrow
G_\dR \stackrel G \times G_\dR
\xto{m}
G_\dR.
$$
Observe that $\ev$ is $(G_\dR \times G_\dR)$-equivariant by construction.
The fact that $\coev$ and $\ev$ do form a duality pairing is a straightforwards diagram case, left to the reader.
\end{proof}

The above proof yields a monoidal equivalence 
$$
\HC = 
\Hom_{\Dmod(G)} (\fg \mod, \fg \mod)
\simeq
\ICoh(G \backslash G_\dR /G),
$$
where the RHS is equipped with the convolution monoidal structure.

\begin{prop}
$\HC$ is compactly generated by objects that are left and right dualizable, hence it is rigid.
\end{prop}

\begin{proof}
The $\IndCoh$-pushforward 
$$
\delta_*^{\IndCoh}: \ICoh(BG) \to \ICoh( G \backslash G_\dR / G)
$$
is monoidal and it admits a conservative right adjoint. For the latter assertion, see \cite[Chapter III.3, Proposition 3.1.2]{GR_corresp}.
From this it follows formally that $\ICoh( G \backslash G_\dR / G)$ is compactly generated and rigid provided that so is $\ICoh(BG) \simeq \QCoh(BG)$. However, the latter is obvious.
\end{proof}

We can finally prove the main theorem of this section.

\begin{thm} \label{thm:equiv-Ginv-coinv}
The functor $\theta_G: \C_G \to \C^G$ of Corollary \ref{cor:k_G-induces} is an equivalence.
\end{thm}

\begin{proof}
The functor of weak $G$-invariants on $G \rrep = (\Dmod(G), \star) \mmod$ sits in the adjunction
\begin{equation}
{\wt\rec} : \HC \mmod \rightleftarrows (\Dmod(G), \star) \mmod: {\inv}^{G,w}
\end{equation}
induced by the $(\Dmod(G), \HC)$-bimodule $\fg \mod = \Dmod(G)^{G,w}$.
By Theorem \ref{thm:Gaitsgory-Lurie}, such adjunction consists of mutually inverse equivalences. In particular, we have 
$$
\C^{G} \simeq \Hom_{\HC} ( \Rep(G), \C^{G,w}).
$$
The rigidity of $\HC$, together with the self-duality of $\Rep(G)$, shows that formation of strong $G$-invariant commutes with colimits and tensor products by categories.
We are now in the position to repeat the same argument of Corollary \ref{cor:theta.equiv.for.weak}. Namely, we view $\theta_G$ as a continuous functor $G \rrep \to \DGCat^{\Delta^1}$ and just need to check that $\theta_G$ is an equivalence for $\C = \Dmod(G)$, where it is manifest.
\end{proof}

\sssec{} \label{sssec:normal}

We digress briefly to discuss actions by normal subgroups. Let $i: K \hookrightarrow G$ be a normal subgroup, so that $\Dmod(K)$ acts on $\Dmod(G)$ via $i_*: \Dmod(K) \to \Dmod(G)$. We prove that $G$-invariants and $G$-coinvariants can be taken in two steps.

\begin{lem} \label{lem:normal-inv}
Let $\C$ be a category with a strong right action of $G$. The quotient group $Q:=G/K$ acts on $\C^K$ and $\C_K$ from the right in such a way that
$$
\C^G \simeq  \big(\C^K \big)^Q,
\hspace*{.6cm}
 \C_G \simeq  \left( \C_K \right)_{G/K}.
$$
With the obvious modifications, the statement holds for weak actions and weak (co)invariants.
\end{lem}

\begin{proof}
For coinvariants, we have
$$
\C_K := \C \usotimes{\Dmod(K)} \Vect \simeq \C \usotimes{\Dmod(G)} \Big( \Dmod(G) \usotimes{\Dmod(K)} \Vect \Big) 
 \simeq \C \usotimes{\Dmod(G)} \Dmod(Q),
$$
where the last step uses Theorem \ref{thm:equiv-Ginv-coinv}. The sought-after $Q$-action on $\C_K$ is the one induced by the regular action of $Q$ on itself and, tautologically,
$$
\left( \C_K \right)_{Q}
\simeq \C_K \usotimes{\Dmod(Q)} \Vect
\simeq \C \usotimes{\Dmod(Q)} \Dmod(Q) \usotimes{\Dmod(G)} \Vect 
\simeq \C \usotimes{\Dmod(G)} \Vect   = \C_G.
$$
The version of argument for invariants and for weak actions is entirely analogous.
\end{proof}

\ssec{Twisted group actions} \label{ssec:twisted_act_FD}

Given a category $\C \in G \rrep$, we shall explain how to twist the action by an additive character $\m : G \to \GG_a$.

\sssec{Strong actions on $\Vect$}

Tautologically, a strong action of $G$ on $\Vect$ is given by a monoidal functor $(\Dmod(G), \star) \to \Vect$. Equivalently, by a comonoidal functor $\Vect \to (\Dmod(G), m^!)$. Such functors correspond precisely to \emph{character $\fD$-modules}, i.e. $\fD$-modules $\F$ equipped with an isomorphism
$$
m^! (\F) \simeq \F \boxtimes \F,
\hspace{.4cm}
\mathit{unit}^!(\F) \simeq \kk,
$$ 
satisfying the natural compatibility conditions.\footnote{Observe that character $\fD$-modules belong to the heart of the natural t-structure on $\Dmod(G)$, whence no $\infty$-categorical theory is necessary to define them.}
For such $\F$, the action map $\Dmod(G) \otimes \Vect \to \Vect$ is given by 
$$
M \otimes V \mapsto M \star_\F V := \eps_{G} (M \otimes \F) \otimes V,
$$
where, $\eps_G = \Gamma_\dR \circ \Delta^!$ is the evaluation pairing between $\Dmod(G)$ and $\Dmod(G)^\vee \simeq \Dmod(G)$.

\sssec{}

Consider the exponential (right) $\fD$-module $exp$ on $\GG_a = \AA^1 = \Spec (k[z])$:
\begin{equation} \label{def:exp}
exp = {\fD_{\AA^1}}/{(\partial_z-1) \fD_{\AA^1}}.
\end{equation}
(This is a substitute of the Artin-Schreier sheaf in characteristic zero.) It is a character $\fD$-module,
\begin{equation}\label{eqn:exp_homo}
m^! (exp) \simeq exp \boxtimes exp,
\end{equation}
and the prototype of all the character $\fD$-modules we shall consider. In fact:

\begin{lem} \label{lem:char_pull-back}
Let $\m: G \to \GG_a$ be an additive character of $G$. Then $\m^! (exp) \in \Dmod(G)$ is a character $\fD$-module.
\end{lem}

\nc{\Tw}{\ms{Tw}}

We write $\Vect_{\m}$ to emphasize that $\Vect$ is being considered as a category with the $\Dmod(G)$-coaction corresponding to $\m^! (exp)$.
Recall that $G \on{-}\mathbf{rep}$ is a monoidal $\infty$-category. Hence, we can define the $\m$-twist functor
\begin{equation} \label{eqn:twist}
\Tw_\m : G \rrep \to G \rrep, 
\hspace*{.4cm}
\C \to \C \otimes \Vect_\m.
\end{equation}
Clearly, $\Tw_\m$ is an automorphism with inverse $\Tw_{-\m}$.

\medskip

In terms of the $G$-coaction on $\C$, the $G$-coaction on $\C \otimes \Vect_\m$ is given by
$$
\C   \to  \Dmod(G) \otimes \C, 
\hspace*{.4cm} c \mapsto \m^!(exp) \otimes \coact(c),
$$
where we have identified $\C \simeq \C \otimes \Vect_{\m}$ (as DG categories).

\sssec{}

Define the \emph{$(G,\m)$-invariant} and \emph{$(G,\m)$-coinvariant} categories of $\C$ respectively as
$$
\C^{G,\m} := \Hom_{\Dmod(G)}(\Vect_\m, \C)
\hspace{.6cm}
\C_{G,\m} :=  \Vect_\m \usotimes{\Dmod(G)} \C.
$$
\begin{lem} \label{lem:Whit_with_dualiz_char}
$\C^{G,\m} \simeq\big( \C \otimes \Vect_{-\m} \big)^G$ and  $\,\,\C_{G,\mu} \simeq \big( \C \otimes \Vect_{-\m} \big)_G$.

\end{lem}

We have natural functors
$$
\begin{array}{ll}
& \pr_{G,\m}: \C \xto{\Tw_{-\m}} \C \otimes \Vect_{-\m}
\xto{\pr_G} 
\big( \C \otimes \Vect_{-\m} \big)_G
\simeq 
\C_{G,\m} \vspace*{.3cm} \\

& \oblv^{G,\m}: \C^{G,\m} \simeq \big( \C \otimes \Vect_{-\m} \big)^G
\xto{\oblv^G} 
\C \otimes \Vect_{-\m} 
\xto{\Tw_\m}
\C  \vspace*{.3cm} \\

& \Av_*^{G,\m}: \C \xto{\Tw_{-\m}} \C \otimes \Vect_{-\m}
\xto{\Av^G_*} 
\big( \C \otimes \Vect_{-\m} \big)^G
\simeq 
\C^{G,\m}. \\
\end{array}
$$
As in the untwisted case, $(\oblv^{G,\m}, \Av_*^{G,\m})$ forms an adjoint pair. From Theorem \ref{thm:equiv-Ginv-coinv}, we obtain:

\begin{cor} \label{cor:theta-mu-fin-type}
There is an equivalence $\theta_{G,\m} : \C_{G,\m} \to \C^{G,\m}$ such that $\theta_{G,\m} \circ \pr_{G,\m} \simeq \Av_*^{G,\m}$.
\end{cor}

\section{$\fD$-modules on ind-schemes of pro-finite type} \label{SEC:Dmod-infinite}

Let $(\G,m)$ be a group prestack. Compatibly with the theory of the previous section, we wish to say that $\G$ acts strongly on a category $\C$ if the latter is endowed with an action of the monoidal category $(\Dmod(\G), \star)$. 
Such definition makes sense and behaves well whenever we can provide a construction of $\Dmod(\G)$ as a \emph{dualizable} category endowed with the convolution monoidal functor $m_*$.

\medskip

The ultimate goal of this section is to supply this definition in our cases of interest: $\G = G \ppart$ and $\G = N \ppart$, where $G$ is a reductive group and $N$ its maximal unipotent subgroup. 
To address this, we proceed in two steps (see \citep{KV} for a very similar discussion). 
First, we identify the kind of algebraic structure that $G \ppart$ and $N \ppart$ possess: the answer is that they are \emph{ind-pro schemes}. Roughly speaking, these are prestacks constructed from schemes of finite type out of affine smooth projections and closed embeddings. Secondly, we develop the theory of $\fD$-modules on ind-pro schemes. Applied to the loop group case, such theory gives rise to monoidal categories of $\fD$-modules that enjoy favorable functoriality properies.

\ssec{$\fD$-modules on pro-schemes}

Let $\Sch^{qc, qs}$ be the $1$-category of quasi-compact quasi-separated schemes and $\Sch^{\fty}$ its full subcategory of schemes of finite type.
Between the two lies the $1$-category $\Sch^\pro := \on{Pro}^{\on{aff,sm,s}} \Sch^{\fty}$ of schemes of \emph{pro-finite type}: schemes that can be written as filtered limits of schemes of finite type along affine smooth surjective maps. 
The natural functor $\Sch^\pro \to \Sch^{qc, qs}$ is fully faithful, as shown in Appendix C of \cite{TT}.\footcite{This reference was pointed out by S. Raskin, who has independently developed the theory of $\fD$-modules on schemes of infinite type.} 

\medskip

To shorten the terminology, we refer to objects of $\Sch^\pro$ just as \emph{pro-schemes}. By \emph{pro-group}, we shall mean ``group object in the category of pro-schemes".

\sssec{}

We define $\fD^*$-modules on pro-schemes as follows: we let
$$
\Dmod^* : \Sch^\pro \to \DGCat
$$
be the right Kan extension of $\Dmod: \Sch^{\fty} \to \DGCat$ along the inclusion $\Sch^{\fty} \hookrightarrow \Sch^\pro$. 
Here, $\Dmod: \Sch^{\fty} \to \DGCat$ is the usual functor that assigns $S \rightsquigarrow \Dmod(S)$ and $f \rightsquigarrow f_*$. So, $\Dmod^*$ is covariant by construction: for any morphism $f: X \to Y$ in $\Sch^\pro$, we denote by $f_*$ the corresponding functor $\Dmod^*(X) \to \Dmod^*(Y)$.

\medskip

Explicitly, suppose $Z \in \Sch^\pro$ be presented as 
\begin{equation} \label{eqn:pro-schemeZ}
Z \simeq \lim_{r \in \R} Z^r,
\end{equation}
where $\R^{op}$ is a filtered category. 
(In most cases of interest, $\R = (\NN, <)^{op}$.) 
For any arrow $s \to r$ in $\R^{op}$, let $\pi_{s \to r}: Z^s \to Z^r$ be the corresponding projection. Then, 
\begin{equation} \label{eqn:Dmod*-as-limit}
\Dmod^*(Z) \simeq \lim_{r \in \R, \pi_*}\, \Dmod(Z^r),
\end{equation}
the limit being taken with respect to the pushforwards $(\pi_{s \to r})_*$. 

\sssec{} \label{sssec:Lurie_G0}

On several occasions, we will make use of the following paradigm, due to J. Lurie (cf. \cite{DG}). Let $\C_\bullet: \I \to \DGCat$ be a diagram of categories: for each $\gamma: i \to j$, denote the corresponding functor by $F_\gamma : \C_i \to \C_j$. Let $G_\gamma$ be the possibly discontinuous right adjoint of $F_\gamma$. There is then an equivalence of categories
$$
\Omega: \lim_{i \in \I^\op, G} \C_i \xrightarrow{ \, \, \simeq \, \, } \uscolim{i \in \I, F} \C_i,
$$
under which the tautological functors of ``insertion" and ``evaluation"
$$
\ins_i : \C_i \to \uscolim{i \in \I, F} \C_i \hspace{.6cm} \ev_i : \lim_{i \in \I^\op, G} \C_i  \to \C_i
$$
form an adjoint pair.

\sssec{}

In the case at hand, smoothness of each $\pi_{s \to r}$ implies the existence of the left adjoints $(\pi_{s \to r})^*$, so that
$$
\Dmod^*(Z) \simeq \lim_{\R, \pi_*} \Dmod(Z^r) \simeq \uscolim{\R^{op}, \pi^*}\, \Dmod(Z^r).
$$
The latter expression explains the notation $\Dmod^*$, which is meant to indicate a colimit along $*$-pullback functors. It also shows that $\Dmod^*(Z)$ is compactly generated and it particular. We denote the insertion
$$
\ins_s: \Dmod(Z^s) \to \uscolim{\R^{op}, \pi^*}\, \Dmod(Z^r)
$$
\emph{formally} by $(\pi_{\infty \to s})^*$. Analogously, the evaluation functor will be formally denoted by $(\pi_{\infty \to s})_*$.
Compact objects of $\Dmod^*(Z)$ are those of the form $(\pi_{\infty \to s})^*(M)$, for $M$ compact in $\Dmod(Z^s)$.

\sssec{} \label{sssec:fromLimToColim}

Thanks to the continuity of the push-forward functors $\pi_*$ and the filteredness of $\R^\op$, the isomorphism $ \underset{\R, \pi_*}  \lim\Dmod(Z^r) \to \uscolim{\R^{op}, \pi^*} \Dmod(Z^r)$ can be made explicit (\cite{DG}). 
Namely,
\begin{equation} \label{eqn:fromLimToColim}
\Omega: M = \{M^r \}_r  \mapsto \uscolim{r \in \R^{op}}  \Big( (\pi_{\infty \to r})^*(M^r) \Big),
\end{equation} 
while the inverse equivalence is completely determined by the assignment:
\begin{equation} \label{eqn:fromColimToLim}
\Omega^{-1}: (\pi_{\infty \to r})^* (M^r)  \mapsto  \left\{ \uscolim{k \, : \, k \to r, \, k \to s} (\pi_{k \to s})_* (\pi_{k \to r})^* (M^r)  \right\}_s.
\end{equation}
Clearly, the latter expression greatly simplifies if each pull-back $(\pi_{r_1 \to r_2})^*$ is fully faithful: in that case, the colimit in (\ref{eqn:fromColimToLim}) is independent of $k$.
Thus, we propose the following definition: we say that a pro-scheme $Z$ is \emph{pseudo-contractible} if it admits a presentation as in (\ref{eqn:pro-schemeZ}) with transition maps giving rise to fully faithful $*$-pullback functors. By smoothness, the $!$-pullbacks are also fully faithful.

\sssec{}

The dual of $\Dmod^*(Z)$ is easily computed and it is by the definition the category of $\fD^!$-modules on $Z$:
$$
\Dmod^!(Z):= \big( \Dmod^*(Z) \big)^\vee \simeq \uscolim{\R^{op}, \pi^!} \Dmod(Z^r).
$$
This is a consequence of the fact that $\Dmod(S)^\vee \simeq \Dmod(S)$ via the classical Verdier duality, and that $(\pi_*)^\vee \simeq \pi^!$ under this self-duality.

\medskip

The assigment $Z \rightsquigarrow \Dmod^!(Z)$ upgrades to a contravariant functor mapping $f: X \to Y$ to 
$$
f^! := (f_*)^\vee : \Dmod^!(Y) \to \Dmod^!(X).
$$
As before, let
$$
(\pi_{\infty \to s})^!: \Dmod(Z^s) \to \uscolim{\R^{op}, \pi^!} \Dmod(Z^r)
$$
symbolize the tautological insertion map. Compact objects of $\Dmod^!(Z)$ are those objects of the form $(\pi_{\infty \to s})^!(M)$, for all $s \in \R$ and all $M$ compact in $\Dmod(Z^s)$.

\sssec{}

Since each $\pi_{s \to r}$ is smooth, $(\pi_{s \to r})^! \simeq (\pi_{s \to r})^* [2 d_{rs}]$, where $d_{rs}$ is the dimension of $\pi_{s \to r}$. Thus, the right adjoint of $(\pi_{s \to r})^!$ is isomorphic to $(\pi_{s \to r})_*[-2 d_{rs}]$ (hence, it is continuous).
We can realize $\Dmod^!(Z)$ as a limit:
\begin{equation}\label{eqn:Dmod^!-as-limit}
\Dmod^!(Z) \simeq \lim_{r \in \R, \pi_*[-2d_\pi ]} \Dmod(Z^r).
\end{equation}
Assume a trivialization of the dimension torsor of $Z$ has been specified. This consists of a function $\dim Z^ {\centerdot}: \R \to \ZZ$ such that $\dim (Z^s)- \dim(Z^r) = \dim (\pi_{s \to r})$.\footcite{We also say that $Z$ has been given a \emph{dimension theory}. Note that a canonical dimension theory on $Z$ exists, provided we view each $\dim(Z^r)$ as a locally constant function on $Z^r$. See \cite{Sam:Dmod}.} 
Then, comparing the above formula with (\ref{eqn:Dmod*-as-limit}), we construct the equivalence
\begin{equation} \label{eqn:eta-pro-scheme}
\lambda_Z : \Dmod^!(Z) \xrightarrow{ \, \, \simeq \, \, } \Dmod^*(Z),
\end{equation}
induced by the inverse family of shift functors $ \id [-2 \dim(Z^r)]: \Dmod(Z^r) \to \Dmod(Z^r)$. Equivalently, 
$$
(\pi_{\infty \to r})_* \circ \lambda_Z = \big( (\pi_{\infty \to r})^! \big)^R \circ  [- 2 \dim(Z^r)],
$$
or, upon passage to left adjoints,
\begin{equation} \label{eqn:eta-compat-with-ins}
\lambda_Z \circ (\pi_{\infty \to r})^!  = (\pi_{\infty \to r})^* \circ [-2 \dim(Z^r)].
\end{equation}
\sssec{}

By (\cite{DG}), the duality pairing $\eps_Z$ between $\Dmod^*(Z)$ and $\Dmod^!(Z)$ consists of the assignment
\begin{equation} \label{eqn:dual-pairing-limColim}
\eps_Z ( \{M^s\}_s, (\pi_{\infty \to r})^! N^r ) \simeq \eps_{Z^r} ( M^r, N^r). 
\end{equation}
In other words, $(\pi_{\infty \to r})^!$ is dual to the functor $ (\pi_{\infty \to r})_* :  \Dmod^*(Z) \to \Dmod(Z^r)$.
Alternatively, by (\ref{eqn:fromColimToLim}), we obtain
\begin{eqnarray} \label{eqn:duality_pairing}
\eps_Z ( (\pi_{\infty \to r})^* M^r, (\pi_{\infty \to r})^! N^r ) 
 &=& \uscolim{s \in \R_{/r}} \, \eps_{Z^s} \big( (\pi_{s \to r})^* M^r, (\pi_{s \to r})^! N^r \big)     
  \\
   &=& \uscolim{s \in \R_{/r}} \, \Gamma_\dR \big( Z^s, (\pi_{s \to r})^* (M^r \otimes N^r) \big). \nonumber
\end{eqnarray}


\begin{rem}
If $Z$ is pseudo-contractible, the formula (\ref{eqn:duality_pairing}) simplifies as
$$
\eps_Z ( (\pi_{\infty \to r})^* M^r, (\pi_{\infty \to r})^! N^r )  \simeq \Gamma_\dR (Z^r, M^r \otimes N^r ). 
$$
\end{rem}

\begin{lem}
For any $M^r \in \Dmod(Z^r)_\cpt$, Verdier duality $\DD_Z : \Dmod^!(Z) \xrightarrow{\, \, \simeq \, \, } \Dmod^*(Z)$ is computed ``component-wise", i.e., it sends
$$
\DD_Z: (\pi_{\infty \to r})^! (M^r) \mapsto (\pi_{\infty \to r})^* (\DD_{Z^r} (M^r)).
$$
\end{lem}

\begin{proof}
By definition, it suffices to exhibit a canonical equivalence
\begin{equation} \label{eqn:Verdier_tech}
\Hom_{\Dmod^*(Z)} \Big(  (\pi_{\infty \to r})^* (\DD_{Z^r} (M^r)), P \Big)
\simeq 
\eps_{Z} \big(  ( \pi_{\infty \to r})^! (M^r), P \big),
\end{equation}
for any $P \in \Dmod^*(Z)$ pulled back from a finite-type quotient.
This follows immediately from the adjunction $\big ( (\pi_{\infty \to r})^* , (\pi_{\infty \to r})_* \big )$ for $\fD^*$-modules, combined with the duality between $(\pi_{\infty \to r})^!$ and $(\pi_{\infty \to r})_*$.
\end{proof}

\ssec{Basic functoriality}

In this subsection we work out part of the theory of $\fD^*$ and $\fD^!$-modules on pro-schemes: we discuss various push-forward and pull-back functors, the tensor products, the projection and base-change formulas.

\sssec{}

By definition of right Kan extension, if $f: X \to Y$ is a morphism of pro-schemes, the functor $f_*: \Dmod^*(X) \to \Dmod^*(Y)$ is determined by the formula
\begin{equation} \label{eqn:ev----f_*}
(\pi^Y_{\infty \to r})_* \circ f_* \simeq (f^r)_* \circ (\pi^X_{\infty \to r})_*
\end{equation}
for any morphism $f^r: X^r \to Y^r$ of schemes of finite type covered by $f$.
Below are two main examples.

\medskip

The functor of \emph{de Rham cohomology} $\Gamma_\dR$ of $X \simeq \lim_r X^r$ is defined as
$$
\Gamma_\dR:= (p_X)_* \simeq (p_{X^r})_* \circ (\pi^X_{\infty \to r})_* : \Dmod^*(X) \to \Vect,
$$
Explicitly, if $M \in \Dmod^*(X)$ is represented by the collection $ \{M^r\}_r \in \lim_{r, \pi_*} \Dmod(X^r)$, then $\Gamma_\dr(M) := \Gamma_\dr(X^r, M^r)$, the RHS being independent of $r$.

\medskip

If $i_x: \pt \hookrightarrow X$ is a closed point, the \emph{delta $\fD$-module} at $x$ is given by the usual formula $\delta_{x,X}:=(i_x)_*(\kk) \in \Dmod^*(X)$. In the realization of $\Dmod^*(X)$ as a limit, $\delta_x$ is represented by the collection of $\delta_{x^r,X^r} \in \Dmod(X^r)$, where $x^r$ is the image of $x$ under the projection $X \to X^r$.

\begin{rem}

Contrarily to the finite-type case,  $(i_x)_*$ does not preserve compactness: as pointed out before, compact $\fD^*$-modules on $X$ are (in particular) $*$-pulled back along some projection $X \to X^r$ and $\delta_{x, X}$ is not such. As a consequence, $(i_x)_*$ does not admit a continuous right adjoint. 
However, $(i_x)_*$ is fully faithful, being the limit of the functors $(i_{x^r})_*$, which are fully faithful by Kashiwara's lemma. 

\end{rem}

\sssec{}

Consider again an arbitrary map $f: X \to Y$ of pro-schemes. The left adjoint to $f_*$, denoted $f^*$, is only partially defined. We say that $f^*$ is defined on $M \in \Dmod^*(Y)$ if the functor $\Hom(M, f_*(-))$ is corepresentable (by an object that we denote as $f^*(M)$). 

A sufficient condition for $f^*$ to be defined on all of $\Dmod^*(Y)$ is that each $f^r$ be smooth. Indeed, by (\ref{eqn:ev----f_*}) and the $(\ins, \ev)$ adjunction, $f^*$ is specified by
\begin{equation} \label{eqn:f^*-ins}
(\pi^Y_{\infty \to r})^* \circ (f^r)^* \simeq f^* \circ (\pi^X_{\infty \to r})^*
\end{equation}

In particular, whenever $X^r$ admits a constant sheaf $k_{X^r}$, the functor of de Rham cohomology $p_* : \Dmod^*(X) \to \Vect $ is corepresented by the \emph{constant $\fD^*$-module} 
\begin{equation} \label{eqn:k_X}
k_X := (\pi_{\infty \to r})^*(k_{X^r}),
\end{equation}
where the RHS is independent of $r$. 

\sssec{} 

Let us now discuss the functor $f^!: \Dmod^!(Y) \to \Dmod^!(X)$, dual to $f_*$. By (\ref{eqn:dual-pairing-limColim}), one readily gets
\begin{equation}\label{eqn:f^!}
f^! \circ  (\pi^Y_{\infty \to r})^! = (\pi^X_{\infty \to r})^! \circ (f^r)^!.
\end{equation}
For instance, consider $p: X \to \pt$. Then, $p^!: \Vect \to \Dmod^!(X)$ gives the \emph{dualizing sheaf}, $\omega_X := p^! (\kk)$.
Explicitly,
$$
\omega_X \simeq (\pi_{\infty \to r})^! (\omega_{X^r}),
$$
being clear that the RHS does not depend on $r$.

\sssec{}

The functor $f_!$, left adjoint to $f^!$, is only partially defined. For instance, here is a typically infinite dimensional phenomenon.

\medskip

If $X$ is infinite dimensional, $(i_x)^!$ does not have a left adjoint. Indeed, the value of the hypothetical left adjoint on $\kk$ would have to be compact, hence of the form $(\pi_{\infty \to r})^! (F)$ for some $r$ and some $F \in \Dmod(X^r)_{\cpt}$. It is easy to see that this causes a contradiction. For simplicity, assume that $\R = (\NN, <)^{op}$ and that all $\pi^!$ are fully faithful. For any $s \to r$, adjunction forces
$$ 
\Hom_{\Dmod^!(X)} \big( (\pi_{\infty \to r})^! (F), (\pi_{\infty \to s})^! (-) \big) \simeq (i_{x^s})^!
$$
as functors from $\Dmod(X^s)$ to $\Vect$. It follows that $(\pi_{s\to r})^! (F) \simeq \delta_{x^s}$, which is absurd if $\pi_{s \to r}$ is of positive dimension.

\begin{rem} \label{rem:p^*-for-indschemes}
A similar logic shows that $(p_\X)^*: \Vect \to \Dmod(\X)$ is not defined whenever $\X$ is a (genuine) ind-scheme of ind-finite type. For instance, $\AA^{\infty, \ind} := \colim_n \AA^n$ does not admit a constant sheaf. This fact is ``Fourier dual" to the non-existence of the $!$-pushforward along $i_0 : \pt \to \AA^{\infty, \pro} := \lim_n \AA^n$.
\end{rem}

\sssec{}

The above example shows that $f_!$ may not be defined even if all $(f^r)_!$ are. Matters simplify for $f= p: X 	\to \pt$, with $X$ pseudo-contractible. Then
$$
(p^r)^! \simeq \big( (\pi_{\infty\to r})^! \big)^R \circ p^!,
$$
by contruction and fully faithfulness of $(\pi_{\infty \to r})^!$. Passing to the left adjoints, we obtain
$$
(p^r)_! \simeq  p_! \circ (\pi_{\infty \to r})^!:
$$
in other words,
$$
p_! (M) \simeq \uscolim{r \in \R^\op} (p^r)_! (M^r)
$$
for $M=\{M^r\}_r \in \Dmod^!(X) = \lim_{(\pi^!)^R} \Dmod(X^r)$,  provided that each $(p^r)_!$ is defined on $M^r$.

\sssec{}

Let $f: X \to Y$ be a map of pro-schemes for which the equivalences $\lambda_X$ and $\lambda_Y$ have been specified (for instance, if $X$ and $Y$ are limit of smooth schemes). We shall occasionally use the \emph{renormalized} push-forward
$$
f_*^\ren: \Dmod^!(X) \to \Dmod^!(Y), \, \, \, f_*^\ren := \lambda_Y^{-1} \circ f_* \circ \lambda_X.
$$

\sssec{}

Suppose that $f: X \to Y$ can be presented as the limit of maps $f^r : X^r \to Y^r$ such that all the squares
\begin{gather} \label{diag:square-for-pro-schemes}
\xy
(0,0)*+{ Y^s }="00";
(25,0)*+{  Y^r }="10";
(00,15)*+{  X^s }="01";
(25,15)*+{  X^r }="11";
{\ar@{<-}_{ f^s} "00";"01"};
{\ar@{->}^{\pi} "00";"10"};
{\ar@{<-}_{f^r} "10";"11"};
{\ar@{->}^{\pi} "01";"11"};
\endxy
\end{gather} 
are Cartesian (equivalently, $f$ is finitely presented, \cite{Sam:Dmod}). In this case, base-change yields commutative squares
\begin{gather}
\xy
(0,0)*+{ \Dmod(Y^s) }="00";
(25,0)*+{  \Dmod(Y^r) }="10";
(00,15)*+{  \Dmod(X^s) }="01";
(25,15)*+{  \Dmod(X^r) }="11";
{\ar@{->}_{ (f^s)^!} "00";"01"};
{\ar@{->}^{\pi_* } "00";"10"};
{\ar@{->}_{(f^r)^! } "10";"11"};
{\ar@{->}^{\pi _*} "01";"11"};
(50,0)*+{ \Dmod(Y^s) }="20";
(75,0)*+{  \Dmod(Y^r), }="30";
(50,15)*+{  \Dmod(X^s) }="21";
(75,15)*+{  \Dmod(X^r) }="31";
{\ar@{->}_{ (f^s)_*} "21";"20"};
{\ar@{->}^{\pi^! } "30";"20"};
{\ar@{<-}_{(f^r)_* } "30";"31"};
{\ar@{<-}^{\pi ^!} "21";"31"};
\endxy
\end{gather} 
which allow to define the functors $f^\flip : \Dmod^*(Y) \to \Dmod^*(X)$ and $f_+ : \Dmod^!(X) \to \Dmod^!(Y)$
by the formulas
$$
 (\pi^X_{\infty \to r})_* \circ  f^\flip :=  (f^r)^! \circ  (\pi^Y_{\infty \to r})_*
$$
and
$$
f_+ \circ (\pi^X_{\infty \to r})^! := (\pi^Y_{\infty \to r})^! \circ (f^r)_*.
$$
These functors are easily seen to be independent of the presentations.
Also it is immediate from the finite-type case that the pairs of functors $(f_*, g^\flip)$ and $(f_+, g^!)$ satisfy the base-change formula (see \cite{Sam:Dmod} for a thorough treatment). 

\sssec{} \label{sssec:closed-embedd-pro}

A map $f: X \to Y$ between pro-schemes is a finitely presented closed embedding (resp., proper) if it is the base-change of a closed embedding (resp., proper map) $f^r: X^r \to Y^r$ for some $r \in \R$ (equivalently, assuming that $r$ is final: for all $r \in \R$). 

The inclusion of a point into $Y \in \Sch^\pro$ of infinite type is a closed embedding  that is \emph{not} finitely presented.

\begin{lem} \label{lem:f_bullet}
Let $f: X \to Y$ be finitely presented. The following statements hold true:
\begin{itemize}
\item
if $f$ is proper, then $f_+$ is left adjoint to $f^!$. If moreover $X$ and $Y$ have dimension theories, then 
$$
f_+ \simeq f_*^\ren [2 \dim_f],
$$
where $\dim_f  := \dim (X^r)- \dim(Y^r)$ is clearly well-defined;

\smallskip
\item

if $f$ is a finitely presented closed embedding, then $f_+$ is fully faithful.
\end{itemize}
\end{lem}

\begin{proof}
Base change along the diagrams (\ref{diag:square-for-pro-schemes}) guarantees that $f_+$ and $f^!$ are compatible with the evaluations for the $!$-categories of $\fD$-modules. All statements follow immediately.
\end{proof}

\sssec{}

The $1$-category $\Sch^\pro$ admits products. Moreover, for two pro-schemes $X$ and $Y$, there are canonical equivalences
\begin{eqnarray} 
\Dmod^*(X) \otimes \Dmod^*(Y) \xrightarrow{\,\,\, \boxtimes \,\,\,} \Dmod^*(X \times Y) \nonumber \\
\Dmod^!(X) \otimes \Dmod^!(Y) \xrightarrow{\,\,\, \boxtimes \,\,\,} \Dmod^!(X \times Y), \label{eqn:boxtimes_equiv}
\end{eqnarray}
which follow at once from dualizability of each $\Dmod(X^r)$ and the fact that $\Dmod^*$ and $\Dmod^!$ can be represented as colimits.

\begin{rem} \label{rem:pro-schemes-fiber-products}

It is easy to see that $\Sch^\pro$ admits fiber products. Indeed, let $X \to Z \leftarrow Y$ be a diagram of pro-schemes.
Let $Z = \lim Z^r$ be a presentation of $Z$. For each $r \in \R$, the composition $X \to Z \twoheadrightarrow Z^r$ factors through a projection $X \twoheadrightarrow X'$, with $X'$ of finite type. We let $X^r := X'$. In this way we contruct compatible pro-scheme presentations of $X$ and $Y$ and $\lim_r (X^r \times_{Z^r} Y^r)$ is a presentation of $X \times_Z Y$,

\end{rem} 

Note that $\Dmod^!(X)$ is symmetric monoidal: indeed, it is a colimit formed along the monoidal functors $(\pi_{s \to r})^!$. By construction, each insertion functor $(\pi_{\infty \to r})^! : \Dmod(X^r) \to \Dmod^!(X)$ is monoidal and, consequently, the tensor product on $\Dmod^!(X)$ is defined as usual: 
$$
\Dmod^!(X) \otimes \Dmod^!(X) \xrightarrow{\, \, \boxtimes \, \, } \Dmod^!(X \times X) \xrightarrow{\, \, \Delta^! \, \, } \Dmod^!(X).
$$

\sssec{}

For any $f: X \to Z$, the functor $f^!: \Dmod^!(Z) \to \Dmod^!(X)$ is symmetric monoidal, hence $(\Dmod^!(Z), \otimes)$ acts on $\Dmod^!(X)$.
The relative analogue of  (\ref{eqn:boxtimes_equiv}) holds true as well:
\begin{lem} \label{lem:BN}
Let $X \to Z \leftarrow Y$ be a diagram of pro-schemes.
There is a canonical equivalence
\begin{equation}
\Dmod^!(X) \usotimes{\Dmod^!(Z)} \Dmod^!(Y) \simeq \Dmod^!(X \times_Z Y).
\end{equation}
\end{lem}

\begin{proof}
The canonical functor is induced by pullback along $X \times_Z Y \to X \times Y$. To prove it is an equivalence, we reduce it to the finite dimensional case, where it is true by a result of \cite{BN}. By fixing compatible presentations of $X, Y, Z$ and $X \times_Z Y$ as in Remark \ref{rem:pro-schemes-fiber-products}, we obtain
$$
\Dmod^!( X \times_Z Y) 
\simeq
 \uscolim{r \in \R^\op} \Dmod(X^r \times_{Z^r} Y^r) 
\simeq
 \uscolim{r \in \R^\op} \Big( \Dmod^!(X^r) \usotimes{\Dmod^!(Z^r)} \Dmod^!(Y^r) \Big) 
 \simeq
\Dmod^!(X) \usotimes{\Dmod^!(Z)} \Dmod^!(Y).
$$
\end{proof}

\sssec{}

By Sect. \ref{sssec:paradigm_dual actions}, there is an induced right action of $\Dmod^!(Z)$ on $\Dmod^*(X)$, which we indicate by $\overset{*!}\otimes$. Since $\Dmod^!(Z)$ is symmetric monoidal, we sometimes consider the latter as a left action: in that case, we use the symbol $\overset{!*} \otimes$.
It is easy to check that
$$
(\pi_{\infty \to r})^! (M) \overset{!*} \otimes  (\pi_{\infty \to r})^* (P) = (\pi_{\infty \to r})^* (M \otimes_{X^r} P)
$$
and that
$$
M \overset {!*} \otimes N \simeq \lambda_X (M \otimes \lambda_X^{-1}(N)),
$$
where the RHS is independent of the choice of $\lambda_X$.

\medskip

Both statements follow formally from the functorial equivalence
$$
\pi^!(M) \otimes \pi^*(P) \simeq \pi^*(M \otimes P),
$$
valid for any smooth map $\pi$ between schemes of finite type.
\begin{lem} \label{lem:proj_formula}
For $f: X \to Y$, the functor $f_*$ is a functor of $\Dmod^!(Y)$-module categories. 
\end{lem}

\begin{proof}
This is an instance of the following general phenomenon. In the setting of Sect. \ref{sssec:paradigm_dual actions}, if a morphism $\alpha: \M \to \N$ is $\A$-linear and $\M, \N$ are dualizable in $\S$, then $\alpha^\vee: \N^\vee \to \M^\vee$ is also $\A$-linear with respect to the dual $\A$-actions.
\end{proof}

\begin{cor} \label{cor:proj_formula_ren}
For $f: X \to Y$ and any choice of dimension theories for $X$ and $Y$, the functor $f_*^\ren : \Dmod^!(X) \to \Dmod^!(Y)$ is $\Dmod^!(Y)$-linear.
\end{cor}

\begin{proof}
The three functors $\lambda_X$, $f_*$ and $(\lambda_Y)^{-1}$ are all $\Dmod^!(Y)$-linear.
\end{proof}

The content of these two results is to give the projection formulas, that is, the equivalences
\begin{equation}\label{eqn:proj_formulas}
f_*(M) \overset{*!} \otimes N \simeq f_*(M \overset{*!} \otimes f^!(N))
\hspace*{.6cm}
f_*^\ren (P) \overset{!} \otimes N \simeq f^\ren_*  \big( P \overset{!} \otimes f^!(N) \big)
\end{equation}
\emph{up to coherent homotopy}.

\begin{rem} \label{rem:eval:compact}
The evaluation pairing (\ref{eqn:duality_pairing}) between $\Dmod^*(X)$ and $\Dmod^!(X)$ can be tautologically rewritten as $\eps_X (M , N) \simeq \Gamma_\dR \big(X, M \overset{*!}\otimes N \big)$.
\end{rem}

\ssec{$\fD$-modules on ind-pro-schemes} \label{ssec:D-mods-on-indpro-schemes}

Let us now extend the above theory to the set-up of ind-pro-schemes.

\sssec{}

By definition, a (classical) \emph{ind-scheme} is a prestack that can be presented as a filtered colimit of quasi-compact and quasi-separated schemes along closed embeddings. A frequently used subcategory of such prestacks is that of ind-schemes of \emph{ind-finite type}: it consists of those ind-schemes formed out of schemes of finite type.

\sssec{}

Before explaining the two theories of $\fD$-modules on ind-pro-schemes, let us recall the definition of the category of $\fD$-modules on ind-schemes of ind-finite type. Let $\IndSch^{\fty}:= \ms{Ind}^{cl}\Sch^{\fty}$ denote the $1$-category of such. (Here, the superscript $cl$ stands for ``{closed embedding}".) 
The functor
$$
\Dmod: \IndSch^{\fty} \longrightarrow \DGCat
$$
is defined to be the left Kan extension of $\Dmod: \Sch^{\fty} \to \DGCat$ along the inclusion $\Sch^{\fty} \hookrightarrow \IndSch^{\fty}$.
Explicitly, if $\Y \in \IndSch^{\fty}$ is written as a filtered colimit $\Y = \colim_{n \in \I} Y_n$, with $Y_n \in \Sch^{\on{ft}}$ and closed embeddings $\iota_{m \to n}: Y_m \hookrightarrow Y_n$, then $\Dmod(\Y) \simeq \colim_{n \in \I} \Dmod(Y_n)$, with respect to the pushforward morphisms $(\iota_{m \to n})_*$.

\sssec{}

We repeat the same process for ind-pro-schemes, with the proviso that the closed embeddings must be of \emph{finite presentation}. Namely, we define the ordinary category $\IndSch^\pro:= \ms{Ind}^{cl}\Sch^\pro$ of \emph{ind-pro-schemes} to be the one comprising ind-schemes that can be formed as colimits of pro-schemes under finitely presented closed embeddings. By \emph{ind-pro-group}, we mean a group object in the category $\IndSch^\pro$.

\medskip

The functor
$$
\Dmod^* : \IndSch^{\pro} \to \DGCat
$$
is defined as the left Kan extension of $\Dmod^*: \Sch^\pro \to \DGCat$ along the inclusion $\Sch^\pro \hookrightarrow \IndSch^\pro$.
Analogously, the functor
$$
\Dmod^! : (\IndSch^\pro)^\op \to \DGCat
$$
is defined to be the right Kan extension of $\Dmod^!: (\Sch^\pro)^\op \to \DGCat$ along the inclusion $(\Sch^\pro)^\op \hookrightarrow (\IndSch^\pro)^\op$. 

\medskip

For $f: X \to Y$ a map in $\IndSch^\pro$, we continue to denote by $f_*: \Dmod^*(X) \to \Dmod^*(Y)$ and $f^!: \Dmod^!(Y) \to \Dmod^!(X)$ the induced functors.

\begin{lem}
The duality $\Dmod^! \simeq (\Dmod^*)^\vee$ and the isomorphism $(f_*)^\vee \simeq f^!$ continue to hold for ind-pro-schemes.
\end{lem}

\begin{proof}
Let $\Y = \colim_{n \in \I} Y_n$ be a presentation of $\Y$ as an ind-pro-scheme. The two categories $\Dmod^!(\Y)$ and $\Dmod^*(\Y)$ are, by construction,
$$
\Dmod^!(\Y) \simeq \lim_{n \in \I^\op, \iota^!} \Dmod^!(Y_n), \hspace{1 cm}
\Dmod^*(\Y) \simeq \uscolim{n \in \I, \iota_*} \Dmod^*(Y_n).
$$
They are evidently dual to each other, thanks to the validity of the present lemma for pro-schemes. The duality $(f_*)^\vee \simeq f^!$ is a formal consequence of this.
\end{proof}

\begin{prop} \label{prop_D^!-symm-monoidal}
The ordinary category $\IndSch^\pro$ is symmetric monoidal via Cartesian product and the functor $\Dmod^!: (\IndSch^\pro)^\op \to \DGCat$ is symmetric monoidal.
\end{prop}

\begin{proof}
The first assertion is obvious. 
As for the second, recall that $\Dmod^!$ is symmetric monoidal as a functor $(\Sch^\pro)^\op \to \DGCat$. Furthermore, for any map of pro-schemes $f: X \to Y$, the pull-back $f^! : \Dmod^!(Y) \to \Dmod^!(X)$ is symmetric monoidal. The combination of these two facts yields the assertion.
\end{proof}

The following result is a generalization of Lemma \ref{lem:BN}. The proof is completely analogous.

\begin{lem} \label{lem:BN-ind}
Let $\X \to \Z \leftarrow \Y$ be a diagram of pro-schemes.
There is a canonical equivalence
\begin{equation}
\Dmod^!(\X) \usotimes{\Dmod^!(\Z)} \Dmod^!(\Y) \simeq \Dmod^!(\X \times_\Z \Y).
\end{equation}
\end{lem}

\sssec{}

Let $\Y = \colim_{n \in \I} Y_n$ an ind-pro-scheme.
By Lemma \ref{lem:f_bullet} and Sect. \ref{sssec:fromLimToColim}, there is an equivalence
$$
\Omega : \Dmod^!(Y) \simeq \uscolim{n \in \I, \iota_+} \Dmod^!(Y_n).
$$
Assume that each $Y_n$ is equipped with a dimension theory and the corresponding self-duality $\lambda_n : \Dmod^!(Y_n) \to \Dmod^*(Y_n)$. For any arrow $m \to n$ in $\I$, the cited lemma yields the commutative diagram
\begin{gather}
\xy
(0,0)*+{ \Dmod^!(Y_m)}="00";
(40,0)*+{  \Dmod^!(Y_n). }="10";
(00,15)*+{   \Dmod^*(Y_m) }="01";
(40,15)*+{  \Dmod^*(Y_n)  }="11";
{\ar@{<-}_{\lambda_m} "01";"00"};
{\ar@{->}^{\iota_+} "00";"10"};
{\ar@{<-}_{\lambda_n[2 \dim_{\iota}]} "11";"10"};
{\ar@{->}^{\iota_*} "01";"11"};
\endxy
\end{gather}
A dimension theory for $\Y$ (or, trivialization of the dimension torsor of $\Y$) is an assignment $Y_n \rightsquigarrow \dim(Y_n) \in \ZZ$ such that $\dim(\iota_{m \to n}) = \dim(Y_m) - \dim(Y_n)$ for any arrow $m \to n$.

\begin{rem}
Any dimension theory $\Lambda$ produces a canonical object $\Lambda(\omega_\Y)$ of $\Dmod^*(\Y)$. Assuming each $Y_n$ is a limit of smooth schemes (so that $Y_n$ is equipped with the canonical dimension theory), we obtain
$$
\Lambda_m(\omega_\Y) \simeq \uscolim{n \in \I_{m/ }} (i_{n\to \infty})_* k_{Y^n} [2 \dim(Y_n)] \in \Dmod^*(\Y) .
$$
\end{rem}

\sssec{}

For any map $f: \X \to \Y$ of ind-pro-schemes, $\Dmod^!(\Y)$ acts on $\Dmod^!(\X)$ by pull-back. We denote the dual action by the usual symbol:
$$
\overset{!*} \otimes : \Dmod^!(\Y) \otimes \Dmod^*(\X) \to \Dmod^*(\X).
$$
As in the pro-scheme case, this action is the one induced by the self-duality $\Lambda: \Dmod^!(\X) \to \Dmod^*(\X)$ (for any choice of $\Lambda$). Consequently, we have the ``projection formulas":

\begin{prop} \label{prop:proj-formula-indschemes}
For $f: \X \to \Y$ as above, the functors $f_*$ and $f_*^\ren := \Lambda^{\Y} \circ f_* \circ (\Lambda^{\X})^{-1}$ are $\Dmod^!(\Y)$-linear.
\end{prop}

\begin{rem}
The evaluation pairing between $\Dmod^*(\X)$ and $\Dmod^!(\X)$ is
\begin{equation} \label{eqn:eval-compact}
\eps_\X (M,N) \simeq p_* \big(M \overset{*!}\otimes N \big) \simeq  p_*^\ren \big(\Lambda^{-1}(M) \otimes N \big).
\end{equation}
\end{rem}

\ssec{$\fD$-modules on $G \ppart$ and $N \ppart$} \label{ssec:D-mods-on-loop-groups}

Finally, let us take up the case of loop groups.

\sssec{}

Recall that, for $G$ an affine algebraic group, the loop group $\bG:=G \ppart$ comes with the canonical ``decreasing" sequence of congruence subgroups $\bG^r$, starting with $\bG^0 = G [[t]]$ and shrinking to the identity element. 
Recall that $\bG/ \bG^0 =: \Gr$, the so-called \emph{affine Grassmannian}, is an ind-scheme of finite type. (Hence, so is each quotient $\bG/\bG^r$.)

The following construction is well-known.

\begin{lem}
The prestack $\bG := G \ppart$ is an ind-pro-scheme.
\end{lem}

\begin{proof}
Consider the (schematic) quotient map to the affine Grassmannian $\q: \bG \to \Gr$ and choose a presentation of $\Gr$ as an ind-scheme of ind-finite type: $\Gr \simeq \colim_{n, \iota} Z_n$. 
Pulling-back each $Z_n$ along $\q$, we obtain an ind-scheme presentation of $\bG$:
$$
\bG \simeq \uscolim{n \in \NN} \q^{-1}(Z_n).
$$
Of course, each $\q^{-1}(Z_n)=Z_n \underset{\Gr}\times \bG$ is of infinite type. However, as $\q$ factors through $\bG \to \bG/\bG^r$ for any $r \in \NN$, we can write:
$$
\q^{-1}(Z_n)  \simeq \lim_{r \in \NN} \big(  Z_n \underset{\Gr}\times \bG/\bG^r \big),
$$
where the limit is taken along the maps induced by the smooth projections $\pi_{s \to r}: \bG/ \bG^s \to \bG/\bG^r$. This is a presentation of $\q^{-1}(Z_n)$ as a pro-scheme, as desired.
\end{proof}

\begin{lem} \label{lem:DmodG}
$$
\Dmod^*(\bG) \simeq \uscolim{r, \pi^*} \Dmod(\bG/\bG^r), \, \, \, \, \, \, \, \Dmod^!(\bG) \simeq \uscolim{r, \pi^!} \Dmod(\bG/\bG^r).
$$
\end{lem}

\begin{proof}
We only prove the first formula, the proof of the second being completely analogous. 
Let $Z_n^r :=   Z_n \underset{\Gr}\times \bG/\bG^r$, so that 
$$
\bG = \uscolim{n,\iota} \lim_{r, \pi} Z_n^r \, \, \hspace{.5cm} \, \, \bG/\bG^r = \colim_{n,\iota} Z_n^r
$$
are presentations of $\bG$ and $\bG/\bG^r$ as an ind-pro-scheme and an ind-scheme, respectively. 
For each $n$ and $r$, consider the evidently Cartesian square:
\begin{gather}
\xy
(0,0)*+{ Z_{n+1} \times_{\Gr} \bG/ \bG^{r+1}}="00";
(40,0)*+{  Z_{n+1} \times_{\Gr} \bG/ \bG^{r}. }="10";
(00,15)*+{   Z_{n} \times_{\Gr} \bG/ \bG^{r+1} }="01";
(40,15)*+{  Z_{n} \times_{\Gr} \bG/ \bG^{r}  }="11";
{\ar@{<-}^{\iota} "00";"01"};
{\ar@{->}^{\pi} "00";"10"};
{\ar@{<-}^{\iota} "10";"11"};
{\ar@{->}^{\pi} "01";"11"};
\endxy
\end{gather}
Consequently, the category of $\fD^*$-modules on $\bG$ is expressed as follows:
$$
\Dmod^*(\bG) 
\simeq  \uscolim{n, \iota_*}  \lim_{r, \pi_*} \Dmod(Z_n^r) 
\simeq  \uscolim{n, \iota_*}  \uscolim{r, \pi^*} \Dmod(Z_n^r)
\simeq \uscolim{r, \pi^*} \uscolim{n, \iota_*} \Dmod(Z_n^r)
\simeq \uscolim{r, \pi^*} \Dmod(\bG/\bG^r), 
$$
where the switch of colimits in the third equivalence is a consequence of base-change along the above square.
\end{proof}


\sssec{}

We now prove that $\Dmod^*(\bG)$ has a convolution monoidal structure. In the next section, we will use this result to define categorical $\bG$-actions. 

\begin{lem} \label{lem:D^*(G)-monoidal}
Let $\G$ be a group ind-scheme of pro-finite type.
The functor $m_*: \Dmod^*(\G \times \G) \to \Dmod^*(\G)$, together with the equivalence
\begin{equation}
\Dmod^*(\G) \otimes \Dmod^*(\G) 
 \xrightarrow{\, \, \boxtimes \, \, \,} \Dmod^* (\G \times \G),
\end{equation}
endows $\Dmod^*(\G)$ with a monoidal structure.
\end{lem}

\begin{proof}
The functor $\Dmod^! : (\IndSch^\pro)^\op \to \DGCat$ is contravariant and, by Proposition \ref{prop_D^!-symm-monoidal}, symmetric monoidal. Hence, it sends algebras in $\IndSch^\pro$ to comonoidal categories: in particular, $\big( \Dmod^!(\G), m^! \big)$ is comonoidal. By duality, we obtain the required statement.
\end{proof}

\section{Strong actions by ind-pro-groups}\label{SEC:grp-actions}

With the theory of $\fD$-modules on ind-pro-schemes in place, we may define strong actions of ind-pro-groups on categories. As in the finite dimensional case, we discuss the (twisted) invariant and coinvariants categories and the natural functors relating them. For a pro-group $H$ admitting a Levi decomposition and a character $\m: H \to \GG_a$, the $(H,\m)$-invariant and $(H,\m)$-coinvariant categories are equivalent via a natural functor that we call $\theta_{H,\m}$ (see Theorem \ref{thm:theta_mu}).

This will be of fundamental importance for the study of Whittaker categories and for the proof of our main theorem. Indeed, approximating $\bN$ by its pro-unipotent subgroups and using the corresponding $\theta$'s, we construct a functor $\Theta_{\bN,\chi}$ that ought to realize an equivalence between $\C_{\bN, \chi}$ and $\C^{\bN, \chi}$.

\ssec{The main definitions}

Let $\G$ be an ind-pro-group. For instance, $\G = G \ppart$ or $\G = N \ppart$ where $G$ is a reductive group and $N$ its maximal unipotent subgroup.
By definition, $\C \in \DGCat$ is acted on by $\G$ if it is endowed with an action of the monoidal category $(\Dmod^*(\G), m_*)$ (see Lemma \ref{lem:D^*(G)-monoidal}). The totality of categories with $\G$ action forms an $\infty$-category, denoted by $\G \on{-} \mathbf{rep}$. 
As in the finite-type case, $\G \on{-} \mathbf{rep}$ is monoidal: indeed, $\Dmod^*(\G)$ is Hopf with multiplication $\star := m_*$ and comultiplication $\Delta_*$. By duality, $\Dmod^!(\G)$ is also a Hopf algebra, with multiplication $\Delta^!$ and comultiplication $m^!$.

\sssec{} \label{sssec:coin-inv-simplicial-constructions}

Since the map $p: \G \to \pt$ is $\G$-equivariant, the pushforward $p_*$ gives a $\G$-action on $\Vect$, the so called trivial action. This allows us to define invariant and coinvariants, as in the finite dimensional case.

\medskip

The coinvariant category
$$
\C_{\G} := \Vect \usotimes{\Dmod^*(\G)} \C
$$
is computed by  the bar resolution of the relative tensor product, i.e. the simplicial category
\begin{equation} \label{eqn:simplicial_cat}
\cdots \,\, \Dmod^*(\G) \otimes \Dmod^*(\G) \otimes \C \rrr \Dmod^*(\G) \otimes \C \rr \C,
\end{equation}
where the maps are given by action, multiplication and trivial action, according to the usual pattern. As in the finite-type context, there is a tautological map
$$
\pr_\G := (p_*)_\C : \C \to \C_\G.
$$

\sssec{}

Analogously, the invariant category
$$
\C^{\G} := \Hom_{\Dmod^*(\G)} ( \Vect,  \C )
$$
is computed by the totalization of 
$$
\C \rr \Hom(\Dmod^*(\G), \C) \rrr \Hom (\Dmod^*(\G) \otimes \Dmod^*(\G), \C) \,\, \cdots.
$$
Moreover, as $\Dmod^*(\G)$ and $\Dmod^!(\G)$ are in duality, the latter becomes
\begin{equation} \label{eqn:cosimplicial_cat}
\C \rr \Dmod^!(\G) \otimes \C \rrr  \Dmod^!(\G) \otimes \Dmod^!(\G) \otimes \C \, \,\cdots .
\end{equation}
We have the conservative tautological map
$$
\oblv^\G := (p_*)^\C : \C^\G \to \C.
$$
Its right adjoint $\Av_*^\G$ (defined for abstract categorical reasons) may not be continuous whenever $\G$ is an ind-scheme. Indeed, $p_*: \Dmod^*(\G) \to \Vect$ may not admit a left adjoint in that case, see Remark \ref{rem:p^*-for-indschemes}.
We also denote by $\Av_!^\G: \C \to \C^\G$ the \emph{partially defined} left adjoint to $\oblv^\G$.

\sssec{Example: $\bN$-actions} \label{sssec:N-actions}

Let $G$ be a reductive group and $N$ its maximal unipotent subgroup. Unlike $\bG$, the ind-pro scheme $\bN := N \ppart$ is exhausted by its compact open subgroups, hence it is an ind-object in the category of pro-unipotent group schemes. 
We can choose a presentation $\bN \simeq \colim_k \bN_k$, as a colimit of groups, indexed by the natural numbers. Then, an $\bN$-action on $\C$ corresponds to a family of compatible actions of $(\Dmod^*(\bN_k),\star)$ on $\C$. 
It follows that
$$
\C_{\bN} \simeq \uscolim{k \in \NN} \C_{\bN_k},
$$
where the functors in the directed system 
$$
\Vect \usotimes{\Dmod^*(\bN_k)} \C \to \Vect \usotimes{\Dmod^*(\bN_{k+1})} \C
$$
come from the push-forwards $i_*: \Dmod^*(\bN_k) \to \Dmod^*(\bN_{k+1})$.
Likewise, 
$$
\C^{\bN} \simeq \lim_{k \in \NN} \C^{\bN_k},
$$
the limit being along the forgetful functors
$$
\Hom_{\Dmod^*(\bN_{k+1})} (\Vect, \C) \to \Hom_{\Dmod^*(\bN_{k})} (\Vect, \C).
$$
See Proposition \ref{prop:N-inv_vs_Nk-inv} below for a more explicit description of the transition maps in both cases.

\sssec{Actions on $\Vect$}

As in the finite-type case, a strong action of $\G$ on $\Vect$ consists of a comonoidal functor $\Vect \to \Dmod^!(\G)$. The latter is equivalent to specifying a character $\fD^!$-module on $\G$. For such $\F \in \Dmod^!(\G)$, the action map $\Dmod^*(\G) \otimes \Vect \to \Vect$ is 
$$
M \otimes V \mapsto M \star_\F V := \eps_{\G} (M \otimes \F) \otimes V,
$$
where, $\eps_\G = \Gamma_\dR \big( - \overset {*!} \otimes - \big)$ is the duality pairing between $\Dmod^*(\G)$ and $\Dmod^!(\G) := \Dmod^*(\G)^\vee$.

\medskip

Let $\mu: \G \to \GG_a$ be any additive character. By the (obvious) ind-pro version of Lemma \ref{lem:char_pull-back}, $\mu^! (exp) \in \Dmod^!(\G)$ is a character $\fD^!$-module.
We write $\Vect_{\mu}$ to emphasize that $\Vect$ is being considered as a category with the $\Dmod^!(\G)$-coaction corresponding to $\mu^!(\exp)$. 

\sssec{}

The automorphism $\Tw_\m$ of $\G \rrep$ is defined as in (\ref{eqn:twist}). Tautologically, the action of $\Dmod^*(\G)$ on $\C \otimes \Vect_\mu$ consists of the composition of the ``old" action of $\Dmod^*(\G)$ on $\C$ with the monoidal automorphism
$$
\Dmod^*(\G) \to \Dmod^*(\G), 
\hspace*{.4cm} 
M \mapsto \mu^!(exp) \overset {!*} \otimes M.
$$
For $\C \in \G \rrep$, we define the $(\G, \mu)$-invariant and $(\G, \mu)$-coinvariant categories as
$$
\C^{\G, \mu} := \Hom_{\Dmod^*(\G)} (\Vect_\m, \C),
\hspace{.6cm}
\C_{\G,\mu} := \Vect_{\m} \usotimes{\Dmod^*(\G)} \C.
$$
Lemma \ref{lem:Whit_with_dualiz_char} and the definitions of $\oblv$, $\pr$, $\Av$ generalize verbatim from the finite dimensional case.

\ssec{Actions by group pro-schemes}

The formulas of Example \ref{sssec:N-actions} show that, in order to understand $\bN$-actions on categories, one should first discuss $\bN_k$-actions. In this short section we study categorical actions by a group pro-scheme $H$ that admits a \emph{Levi decomposition} as $H \simeq H^r \ltimes H^u$, with $H^r$ reductive and $H^u$ pro-unipotent. 

\sssec{}

Since $H$ can be realized as a limit of smooth schemes, the usual adjunction $p^*: \Vect \rightleftarrows \Dmod^*(H): p_*$ is available.

\begin{prop} \label{prop:p^*-equiv-pro}
The functor $p^*: \Vect \to \Dmod^*(H)$ is $H$-equivariant.
\end{prop}

\begin{proof}
It suffices to prove that the diagram 
\begin{gather}
\xy
(40,0)*+{  \Vect }="00";
(0,0)*+{ \Dmod^*(H) \otimes \Vect }="10";
(40,15)*+{ \Dmod^*(H)  }="01";
(00,15)*+{ \Dmod^*(H) \otimes \Dmod^*(H) }="11";
{\ar@{->}_{ p^* } "00";"01"};
{\ar@{<-}_{  \Gamma_\dR  } "00";"10"};
{\ar@{<-}_{\,\,\,m_* } "01";"11"};
{\ar@{->}_{ \id \otimes p^*} "10";"11"};
\endxy
\end{gather} 
commutes. By means of the automorphism $\xi^*$ of $\Dmod^*(H \times H)$ induced by $\xi(g, h) = (g, gh)$, the top arrow can be converted into $p_* \otimes \id$ and commutativity is obvious.
\end{proof}

Using the same logic as in the finite-type case, we obtain:

\begin{itemize}
\item
the functor $\Av_*^H \simeq (p^*)^\C : \C \to \C^H$, right adjoint to $\oblv^H$, is continuous;
\medskip 
\item
the projection $\pr_H : \C \to \C_H$ admits a left adjoint, $(\pr_H)^L$;
\medskip
\item
the composition $\oblv^H \circ \Av_*^H \simeq k_H \star -$ factors through a functor $\theta_H : \C_H \to \C^H$. We shall prove later (Theorem {\ref{thm:theta_mu}}) that this functor is an equivalence.
\end{itemize}

\sssec{}

Let $\mu: H \to \GG_a$ be a character. The composition
$$
\oblv^{H,\m} \circ \Av_*^{H,\m}: \C 
\xto{\Tw_{-\m}}
\C \otimes \Vect_{-\m} 
\xto{k_H \star - }
\C \otimes \Vect_{-\m} 
\xto{\Tw_{\m}}
\C
$$
is immediately seen to be isomorphic to the functor
\begin{equation} \label{eqn:mumu}
\oblv^{H,\mu}\circ \Av_*^{H,\mu} \simeq
\big( (-\mu)^!(exp) \overset{!*}\otimes k_H \big) \star -: \C \to \C
\end{equation}
and it descends to $\theta_{H,\m}: \C_{H,\m} \to \C^{H,\m}$.
To shorten the notation, we set 
\begin{equation} \label{not:short}
(-\mu)_H :=  (-\mu)^!(exp) \overset{!*}\otimes k_H.
\end{equation}
%
%
\begin{thm} \label{thm:theta_mu}
For $H$ as above, the functor $\theta_{H, \mu}: \C_{H, \mu} \to \C^{H, \mu}$ is an equivalence.
\end{thm}

\begin{proof}
The assertion follows from Corollary \ref{cor:theta-mu-fin-type} and Theorem \ref{thm:inv=coinv} below, together with the discussion of Sect. \ref{ssec:semi-direct} regarding semi-direct products.
\end{proof}

\ssec{Actions by pro-unipotent group schemes}

Let us now discuss the pro-unipotent case in greater detail. Let $\mu: H \to \GG_a$ be a character.

\begin{cor}
If $H$ is pro-unipotent, $\oblv^{H,\mu}$ and $(\pr_{H,\mu})^L$ are fully faithful.
\end{cor}

\begin{proof}
By changing $\C$ with $\C \otimes \Vect_{-\m}$, it suffices to prove the claim for $\mu =0$. 
Both assertions follow from the cohomological contractibility of $H$: indeed, $p_* \circ p^* \simeq \Gamma_\dR(H, k_H) \simeq \kk$.
\end{proof}
Thus, for pro-unipotent $H$, we often regard $\C^{H,\mu}$ as a subcategory of $\C$ and $\Av_*^{H,\mu}$ as an endofunctor of $\C$. In the pro-unipotent case, the fact that $\theta_{H,\m}$ is an equivalence is very easy:

\begin{thm} \label{thm:inv=coinv}
Let $H$ be a pro-unipotent group. 
The functors 
$$
\theta_{H,\m} : \C_{H,\m} \rightleftarrows \C^{H,\m}: \pr_{H,\m} \circ \oblv^{H,\m}
$$
are mutually inverse equivalences. In particular, the operation $\C \rightsquigarrow \C^{H,\m}$ commutes with colimits and tensor products by categories.
\end{thm}

\begin{proof}
Without loss of generality, $\m =0$.
We prove that both compositions are naturally isomorphic to the identity functors.
On one hand:
$$
\pr_H \circ \oblv^H \circ \theta_H  (\pr_H (c)) \simeq 
\pr_H ( \Av_*^H (c)) 
\xrightarrow{\, \, \simeq \, \,}
\pr_H (\Gamma_\dR (H, k_H) \otimes c) \simeq \pr_H(c),
$$
for $\pr_H$ coequalizes the given $H$-action and the trivial $H$-action.
On the other hand:
$$
\theta_H \circ \pr_H \circ \oblv^H (c) \simeq
\Av_*^H \circ \oblv^H (c) 
\xrightarrow{\, \, \simeq \, \,}
c,
$$
for $\oblv^H$ is fully faithful.
\end{proof}

\begin{rem} \label{rem:Av_*-twisted-in-prounip-case}
Let $H$ be pro-unipotent and $\lambda_H$ the self-duality associated to the canonical dimension theory of $H$. Then, $\lambda_H^{-1}(k_H) \simeq \omega_H$ and
\begin{equation} \label{eqn:oblvAv-twisted-unip}
\oblv^{H,\m} \circ \Av_*^{H,\m} \simeq  \lambda_H ((-\m^!) exp)  \star -: \C \to \C.
\end{equation}
The notation introduced in (\ref{not:short}) simplifies to $(-\mu)_H := \lambda_H ((-\m)^!exp)$ in the pro-unipotent case.
\end{rem}

\sssec{}

Let $S \to H$ be an inclusion of pro-unipotent pro-groups. 
\footnote{The choice of the letter $S$ for the given subgroup of $H$ is meant to indicate the word \virg{stabilizer}. Indeed, later in this paper $H$ will act on a space and $S$ will be the stabilizer at a point of that space.}
We wish to define functors $\C^H \rightleftarrows \C^S$, for any $\C \in H \rrep$. To this end, note that $p_* :\Dmod^*(H) \to \Vect$ factors through $\Dmod^*(H)_S$ and gives rise to an adjunction
$$
\pr_S \circ p^* : \Vect \rightleftarrows \Dmod^*(H)_S : p_*,
$$
which is $\Dmod^*(H)$-linear. This induces a pair of adjoint functors
$$
\oblv^{H \to S} :\C^H \rightleftarrows  \Hom_{\Dmod^*(H)} \big( \Dmod^*(H)_S, \C \big) \simeq \C^S : \Av_*^{S \to H}.
$$
When no confusion is likely, we indicate these by $\oblv^{\rel}$ and $\Av_*^{\rel}$, respectively. By contractibility, the composition $\oblv^{\rel} \circ \Av_*^{\rel}$ is given by convolution with $k_{H}$. By changing $\C$ with $\C \otimes \Vect_{-\m}$, these constructions generalize immediately to the $\mu$-twisted case.

\sssec{}

Let us prove two technical results to be used in later chapters.

\begin{lem} \label{lem:auxiliary}
Let $H$ be a pro-unipotent group and $S \hto H$ a normal subgroup. 
Then $H/S$ acts on $\C^S$ and
$\C^H \simeq (\C^{S})^{H/S}$. 
\end{lem}

\begin{proof}
It suffices to prove that  $\Dmod^*(H)^{S} \simeq \Dmod(H/S)$ and then repeat the argument of Lemma \ref{lem:normal-inv}.
We may choose a presentation $H \simeq \lim_{r \in \R^\op} H/H^r$, where each $H^r$ is normal in $H$ and is contained in $S$.
Thus, $S$ acts on $H/H^r$ via the projection $S \twoheadrightarrow S/H^r$. We compute
$$
\Dmod^*(H)^{S} \simeq \lim_{r \in \R^\op} \, \Dmod(H/H^r)^{S} \simeq \lim_{r \in \R^\op} \, \Dmod  \Big( (H/H^r)/(S/H^r) \Big).
$$
The RHS is the limit of the constant family equal to $\Dmod(H/S)$.
\end{proof}

\sssec{}

Let $\I$ be an indexing category and $\{H_i\}_{i \in \I}$ an inverse family of pro-unipotent groups, where the transition maps are closed embeddings of finite presentation. Denote by $H_0$ the final object of this family and $H_{\infty} := \lim_i H_i$: this is the intersection of all the $H_i$'s. Let $\m$ be a character on $H_0$.

\begin{prop} \label{prop:smooth-generation}
If $\C$ is acted on by $H_0$, then
\begin{equation} \label{eqn:smooth-generation}
\C^{H_\infty,\m} \simeq \uscolim{i \in \I^\op, \oblv^{\rel}} \C^{H_i,\m}.
\end{equation}
\end{prop}

\begin{proof}

Without loss of generality, we may assume that $\mu =0$ and that $H_\infty = \{1\}$. We shall construct a pair of inverse functors $\C \rightleftarrows \colim_i \C^{H_i}$.
By Sect. \ref{sssec:fromLimToColim}, we have
$$
\uscolim{\I^\op, \oblv} \C^{H_i} \simeq \lim_{\I, \Av_*^{\rel}} \C^{H_i};
$$
we indicate by $\ins_i$ and $\ev_i$ the structure functors, as usual.
We define a functor $\alpha: \colim_i \C^{H_i} \to \C$ by imposing the equality $\alpha \circ \, \ins_i \simeq \oblv^{H_i}$. Then, its right adjoint $\beta:= \alpha^R$ satisfies the relation $\ev_i \circ \beta = \Av^{H_i}_*$. Equivalently, $\beta = \colim_i \ins_i \circ \Av_*^{H_i}$.

\medskip

Obviously, the composition $\beta \circ \alpha$ is the identity, by pro-unipotence of each $H_i$. It remains to show that $\alpha \circ \beta \simeq \id_\C$, which is equivalent to exhibiting a natural equivalence
$$
\delta_{1, H_0} \simeq \uscolim{i \in \I} \big(  (\iota_{H_i \to H_0})_* k_{H_i} \big).
$$
Note that $H_0 \simeq \lim_i H_0/H_i$ is a pro-scheme presentation of $H$, so that
$$
\delta_{1,H_0} \simeq \uscolim{i \in \I} (\pi_{\infty \to i})^* (\delta_{1, Q^i}),
$$
where $Q^i := H_0/H_i$.
We will show that $(\pi_{\infty \to i})^* (\delta_{1, Q^i}) \simeq (\iota_{H_i \to H_0})_* k_{H_i}$, for any $i \in \I$. By pseudo-contractibility,
$$
(\pi_{\infty \to i})^* (\delta_{1, Q^i}) \xrightarrow{ \Omega^{-1}} \big\{ (\pi_{j \to i})^* (\delta_{1, Q^i}) \big\}_{j \in \I_{/i}},
$$
and, by base-change,
\begin{equation} \label{eqn:aux}
(\pi_{\infty \to i})^* (\delta_{1, Q^i}) \simeq \{ k_{\on{fib}(Q^j \to Q^i)} \}_{j \in \I_{/i}}.
\end{equation}
Here, $\on{fib}(Q^j \to Q^i)$ is the fiber of the projection $Q^j \to Q^i$ over $1 \in Q^i$, which is of course $H^i /H^j$. Thus, (\ref{eqn:aux}) coincides with $k_{H_i}$, by the very construction of the latter.
\end{proof}

\ssec{Invariants and coinvariants with respect to $N \ppart$}

We focus now on categories with a (twisted) action of $\bN$ and study their (co)invariants in terms of the invariants for the sequence of $\bN_k$. 
We will define the Whittaker categories of an object $\C \in \bG \rrep$ and a natural functor between them.

\sssec{}

Recall that $\bN \simeq \colim_k \bN_k$; for each $k \in \NN$, denote by $i_k: \bN_k \to \bN_{k+1}$ the inclusion. 

\begin{prop} \label{prop:N-inv_vs_Nk-inv}
There are natural equivalences
$$
\C^{\bN} := \lim_{\oblv^{\rel}} \C^{\bN_k} 
\, \, \, \text{ and }\, \, \,
\C_{\bN} := \uscolim{\Av_*^{\rel}} \, \C^{\bN_k}.
$$
\end{prop}

\begin{proof}
Let us treat coinvariants first. As each functor $(i_k)_*: (\Dmod^*(\bN_k),\star) \to (\Dmod^*(\bN_{k+1}), \star)$ is monoidal, the equivalence $\Dmod^*(\bN) =\colim_{i_*} \Dmod^*(\bN_k)$ is an equivalence of \emph{monoidal} categories. Hence, we can commute the colimit under the tensor product:
$$
\C_{\bN} := \Vect \usotimes{\Dmod(\bN)} \C \simeq \Vect \usotimes{\uscolim{k,i_*} \Dmod(\bN_k)} \C
\simeq \uscolim{k,i_*} \Big( \Vect \usotimes{\Dmod(\bN_k)} \C \Big)  \simeq \uscolim{k,i_*} \, \C_{\bN_k}.
$$
Next, identifying $\C_{\bN_k}$ with $\C^{\bN_k}$ via Theorem \ref{thm:inv=coinv}, the map induced by $i_*$ goes over to $\Av_*^{\rel}: \C^{\bN_k} \to \C^{\bN_{k+1}}$, the right adjoint to the inclusion $\oblv^{\rel}: \C^{\bN_{k+1}} \to \C^{\bN_{k}}$.
Indeed, this follows from the commutativity of the diagram:
\begin{gather} \label{diag:1}
\xy
(0,0)*+{ \Dmod^*(\bN_{k+1}) \otimes \C \, \, }="00";
(32,0)*+{\, \,  \C \, \,  }="10";
(64,0)*+{\, \,  \C^{\bN_{k+1}}.}="20";
(00,15)*+{  \Dmod^*(\bN_{k}) \otimes \C \, \,  }="01";
(32,15)*+{  \, \,  \C \, \,  }="11";
(64,15)*+{\, \,  \C^{\bN_k}   }="21";
{\ar@{->}^{i_*} "01";"00"};
{\ar@<+0.3ex>^{} "00";"10"};
{\ar@<-0.5ex>^{} "00";"10"};
{\ar@<+0.3ex>^{} "01";"11"};
{\ar@<-0.5ex>^{} "01";"11"};
{\ar@{->}^{\id} "11";"10"};
%
%
{\ar@<-0.0ex>^{\Av_*^{\bN_{k+1}}} "10";"20"};
%
{\ar@<-0.0ex>^{\Av_*^{\bN_k}} "11";"21"};
{\ar@{->}^{\Av_*^{\rel}} "21";"20"};
\endxy
\end{gather}

\smallskip

The computation of $\bN$-invariants is easier: $\C^\bN$ is the limit of $\C^{\bN_k}$, along the transition maps $\C^{\bN_{k+1}} \to \C^{\bN_k}$ induced by $i^! : \Dmod^!(\bN_{k+1}) \to \Dmod^!(\bN_k)$. The relevant diagram
\begin{gather} 
\xy
(0,0)*+{ \Dmod^!(\bN_{k+1}) \otimes \C \, \, }="00";
(32,0)*+{\, \,  \C \, \,  }="10";
(64,0)*+{\, \,  \C^{\bN_{k+1}}}="20";
(00,15)*+{  \Dmod^!(\bN_{k}) \otimes \C \, \,  }="01";
(32,15)*+{  \, \,  \C \, \,  }="11";
(64,15)*+{\, \,  \C^{\bN_k}   }="21";
{\ar@{<-}^{i^!} "01";"00"};
{\ar@<+0.3ex>^{} "10";"00"};
{\ar@<-0.5ex>^{} "10";"00"};
{\ar@<+0.3ex>^{} "11";"01"};
{\ar@<-0.5ex>^{} "11";"01"};
{\ar@{->}^{\id} "10";"11"};
%
%
{\ar@{->}^{\oblv^{\bN_{k+1}}} "20";"10"};
%
{\ar@{<-}^{\oblv^{\bN_k}} "11";"21"};
{\ar@{<-}^{\oblv^{\rel}} "21";"20"};
\endxy
\end{gather}
is commutative (the assertion for the left square follows by duality from commutativity of the left square of (\ref{diag:1})). This identifies $\C^{\bN_{k+1}} \to \C^{\bN_k}$ as $\oblv^{\rel}$.
\end{proof}

\sssec{}

Let us now introduce the main objects of this paper. For $\C \in \bG \on{-}\mathbf{rep}$ and $\chi$ the character defined in (\ref{eqn:chi}), we define the \emph{Whittaker invariant} and \emph{Whittaker coinvariant} categories of $\C$ respectively as $\C^{\bN,\chi}$ and $\C_{\bN,\chi}$.

\medskip

In view of Proposition \ref{prop:N-inv_vs_Nk-inv}, we have:
$$
\C^{\bN,\chi} \simeq \lim_{\oblv^\rel} {\C}^{\bN_k,\chi} 
\, \, \, \text{ and }\, \, \,
\C_{\bN,\chi} \simeq \uscolim{\Av_*^{\rel}} \, {\C}^{\bN_k, \chi}.
$$

\sssec{}

For any choice of dimension theory on $\bN$, we shall construct a functor $\Theta: \C_{\bN, \chi} \to \C^{\bN,\chi}$ between the Whittaker categories, called the \emph{renormalized averaging} functor.

\medskip

Without loss of generality, we may assume that the self-duality $\Lambda: \Dmod^!(\bN) \to \Dmod^*(\bN)$ is determined by $\dim(\bN_0) = 0$ and that $\bN$ is presented as $\bN \simeq \colim_{n \geq 0} \bN_n$.
Consider the object $(-\chi)^\ren_\bN := \Lambda ((-\chi)^!exp)$ and the corresponding functor $(-\chi)^\ren_\bN \star - : \C \to \C$. Explicitly, this is the functor
\begin{equation} \label{theta}
(-\chi)^\ren_\bN \star - 
\simeq
 \uscolim{n \geq 0} \Big(  (-\chi)_{\bN_n} \star  - [2 \dim(\bN_n)] \Big).
\end{equation}
where $(-\chi)_{\bN_n}$ was defined in Remark \ref{rem:Av_*-twisted-in-prounip-case}. 

\begin{prop} \label{prop:theta-descends}
The functor (\ref{theta}) descends to a functor $\Theta: \C_{\bN, \chi} \to \C^{\bN,\chi}$.
\end{prop}

The proof fits the following general paradigm.

\sssec{} \label{paradigm-colim-lim}

Let $\I$ be a filtered indexing category, with an initial object $0 \in \I$.
Consider a diagram $\D_\bullet : \I \to \DGCat$, with transition functors $\ins_{i \to j}: \D_i \to \D_j$, and a diagram $\E^\bullet : \I^\op \to \DGCat$, with transition functors $\ev^{j \to i}: \E^j \to \E^i$. 
Let $\{\theta_i: \D_i \to \E^i \}_{i \in \I}$ be a collection of functors together with transitive systems of maps
$$
\theta_i \Longrightarrow \ev^{j\to i} \circ \theta_j \circ \ins_{i \to j},
$$
so that we can form the functor 
$$
\wt \Theta := \uscolim{i \in \I} \big( \ev^{i \to 0}  \circ \theta_i \circ \ins_{0 \to i} \big)
: \D_0 \to \E^0.
$$
Set $\D_\infty := \colim_{i \in \I} \D_i$ and $\E^\infty := \lim_{i \in \I^\op} \E^i$.

\begin{lem}
In the situation just described, assume that all $\ins_{i \to j}$ are \emph{essentially surjective} and that all $\ev^{j \to i}$ are \emph{fully faithful}. Then $\wt \Theta$ factors as $\ev^{\infty \to 0} \circ \Theta \circ \ins_{0 \to \infty}$ for some functor $\Theta: \D_\infty\to \E^{\infty}$.  
\end{lem}

\begin{proof}
By filteredness of $\I$, we can write
$$
\wt\Theta \simeq \ev^{\ell \to 0} \circ \Big( \uscolim{i \in \I_{\ell/}} (\ev^{i \to \ell} \circ \theta_i \circ \ins_{\ell \to i}) \Big) \circ \ins_{0 \to \ell}
$$  
for any $\ell \in \I$. The conclusion is manifest.
\end{proof}

\begin{proof}[Proof of Proposition \ref{prop:theta-descends}]
It suffices to realize that the functor (\ref{theta}) is the functor $\wt\Theta$ obtained via the above paradigm in the example where $\I = \NN$, $\D_ n := \C_{\bN_n, \chi}$, $\E^n := \C^{\bN_n, \chi}$ and $\theta_n := \theta_{\bN_n} [2 \dim(\bN_n)] :  \C_{\bN_n, \chi} \to \C^{\bN_n, \chi}$ is the functor of Theorem \ref{thm:theta_mu}.
\end{proof}

\sssec{}

The following conjecture has been proposed by Gaitsgory:

\begin{conj} \label{conj:Tequiv}
Let $\C$ a category equipped with a $\bG$-action and $\chi$ the character of (\ref{eqn:chi}). For any trivialization of the dimension torsor of $\bN$, the corresponding functor $\Theta: \C_{\bN,\chi} \to \C^{\bN,\chi}$ is an equivalence of categories.
\end{conj}

We prove a refinement of this conjecture for $G= GL_n$ in Section \ref{SEC:GL_n}, but first we need to study actions of loop vector spaces. This is the subject of the next section.

\section{Fourier transform and actions by loop vector groups} \label{SEC:loop_vectors}

If $G= GL_2$, then $N \simeq \AA^1$ is abelian, so that all the notions discussed above (group actions, invariants, coinvariants, averaging functors...) can be understood via Fourier transform. More generally, we consider the case of a vector group $\AA^n$ and its loop group $\bA := \AA^n \ppart$, which is of course the main example of a Tate vector space.

\ssec{Fourier transform for finite dimensional vector spaces}

We start by rendering the well-known theory of the Fourier-Deligne transform (see, e.g., \cite{Laumon}) to the DG setting. Namely, we show that the usual formula (\ref{funct:FT}) naturally upgrades to a \emph{symmetric monoidal} equivalence $(\Dmod(V), \star) \to (\Dmod(V^\vee), \otimes)$ of DG categories.

\sssec{}

Let $V$ be a finite dimensional vector space, thought of as a scheme, with dual $V^\vee$. We indicate by $m$ the addition in $V$, $V^\vee$ or $\GG_a$ (depending on the context) and by $Q: V \times V^\vee \to \GG_a$ the duality pairing. Let  $p_1$ and $p_2$ be the projections from $V \times V^\vee$ to $V$ and $V^\vee$, respectively.
Recall the $\fD$-module $exp$ on $\GG_a$, as in formula (\ref{def:exp}). The Fourier transform kernel is 
$$
exp^Q:=Q^!(exp) \in \Dmod(V \times V^\vee).
$$
A well-known key property of this ``integral kernel" is the equivalence
\begin{equation} \label{eqn:projection_exp}
(p_1)_*(exp^Q) \simeq \delta_{0,V}[2 d_V].
\end{equation}

\sssec{}

The \emph{Fourier transform} $\FT_V$ is the functor
\begin{equation} \label{funct:FT}
\FT_V: \Dmod(V) \rightarrow\Dmod(V^\vee), \, \, \, M \mapsto (p_2)_* (p_1^!(M) \otimes exp^{Q}).
\end{equation} 
Note that $(\FT_V)^\vee \simeq \FT_{V^\vee}$.

\medskip
We also define the \emph{inverse Fourier transform} $\IFT = \IFT_V$ as
$$
\IFT: \Dmod(V^\vee) \rightarrow\Dmod(V), \, \, \, M \mapsto (p_1)_* (p_2^!(M) \otimes exp^{-Q}) [-2 d_V].
$$
This name will be justified by Proposition \ref{prop:FT_equiv_fin-dim}. 
Emphasizing the dependence on $Q$, we record the formula
\begin{equation} \label{eqn:IFT-Q}
\FT_V^{Q} \simeq \IFT_{V^\vee}^{-Q}[2 d_V].
\end{equation}

\sssec{}

To upgrade the known results on $\FT_V$ to the setting of DG categories, we need to recall the formalism of correspondences and how it allows to handle base-change.
Let $\Sch^\fty$ the 1-category of schemes of finite type. We form $\Sch^{\fty}_{\corr}$, the 1-category whose objects are the same as $\Sch^{\fty}$ and whose morphisms are given by correspondences:
$$
\Hom_{\Sch^{\fty}_{\corr}}(S, T) = \Big\{ S \xleftarrow{ \, \alpha \,} H \xrightarrow{\, \beta \,} T \, : H \in \Sch^{\fty}  \Big\}.
$$
Such correspondences compose under fiber product. Moreover, $\Sch^{\fty}_{\corr}$ inherits a symmetric monoidal structure from $\Sch^{\fty}$.

\begin{thm}[\cite{GR_corresp}] \label{thm:GR-corresp}
The assignment $S \rightsquigarrow \Dmod(S)$ upgrades to a \emph{symmetric monoidal functor}
$$
\Dmod: \Sch^{\fty}_{\corr} \to \DGCat
$$
which sends $S \xleftarrow{ \, \alpha \,} H \xrightarrow{\, \beta \,} T$ to the functor $\Dmod(S) \xrightarrow{\beta_* \,\circ\, \alpha^!} \Dmod(T)$.
\end{thm}

In particular, by restricting the domain of the above functor to the 1-category of finite dimensional vector spaces (in schemes) and linear maps under correspondences, we obtain the theory of $\fD$-modules on vector spaces. To discuss further properties of the Fourier transform, we shall need a mild generalization of Theorem \ref{thm:GR-corresp}.

\sssec{}

\nc{\SchCorrGa}{\Sch^{\fty}_{\corr \to \GG_a}}
\nc{\KL}{\mathsf{KL}}

Denote by $\SchCorrGa$ the following symmetric monoidal 1-category. 
Its objects are schemes of finite type and the tensor structure is the ordinary product. Given two objects $V$ and $W$, the set of morphisms $V \dasharrow W$ consists of all diagrams of the form
\begin{gather} \label{diag:decorated_corresp}
\xy
(0,0)*+{ V }="s";
(26,0)*+{ W, }="t";
(13,10)*+{ H }="H";
(13,22)*+{ \GG_a }="Ga";
{\ar@{->}_{\alpha } "H";"s"};
{\ar@{->}^{\beta } "H";"t"};
{\ar@{->}^{f } "H";"Ga"};
\endxy
\end{gather}
where $H \in \Sch^{\fty}$.
For short, we write $(V \xleftarrow{ \, \alpha \,} H \xrightarrow{\, \beta \,} W; f)$ to indicate the morphism $V \dasharrow W$ in (\ref{diag:decorated_corresp}). The composition of $(V \xleftarrow{ \, \alpha \,} H \xrightarrow{\, \beta \,} W; f)$ with $(W \xleftarrow{ \, \gamma \,} K \xrightarrow{\, \delta \,} Z; g)$ is the correspondence
$$
\left( V \xleftarrow{ \, \, \,} H \underset{W}\times K \xrightarrow{\, \, \,} Z; f+g \right).
$$

\begin{thm}
There is a symmetric monoidal functor
$$
\Dmod^{enh}: \SchCorrGa \longrightarrow \DGCat
$$
which sends $V \rightsquigarrow \Dmod(V)$ and
$( V \xleftarrow{ \, \alpha \,} H \xrightarrow{\, \beta \,} W; f ) $ to the arrow
$
\Dmod(V) 
\xrightarrow{\beta_* ( f^!(exp) \otimes \alpha^!(-)) } 
\Dmod(W).
$
\end{thm}

\begin{proof}
By abstract nonsense, the functor $\Dmod$ of Theorem \ref{thm:GR-corresp} upgrades to a symmetric monoidal functor 
$$
\Dmod_{\GG_a}: \Big( \GG_a \mod (\Sch^{\fty}_\corr), \overset {\GG_a} \times \,\,\Big) \to \Big( \GG_a \rrep, \usotimes{(\Dmod(\GG_a), \star)} \,\, \Big).
$$ 
On the other hand, consider the functor of ordinary categories
$$
\Xi: \SchCorrGa 
\longrightarrow
\Big(\GG_a \mod (\Sch^{\fty}_\corr), \overset {\GG_a} \times \, \Big)
$$
defined as follows. At the level of objects, it sends $V \mapsto \GG_a \times V$, where $\GG_a$ acts freely on the first factor. At the level of morphisms, it sends
$$( V \xleftarrow{ \, \alpha \,} H \xrightarrow{\, \beta \,} W; f ) \mapsto 
( \GG_a \times V \xleftarrow{ \, \, \id \times \alpha \, \,} \GG_a \times H \xrightarrow{\, \sigma \,}  \GG_a \times W),
$$
where $\sigma (x, h) = (x + f(h), \beta(h))$. 
We leave it to the reader to check that such functor is symmetric monodal functor (and fully faithful).

\medskip

Next, let $\inv^{\GG_a, \exp} : \GG_a \on{-} \mathbf{rep} \to \DGCat$ be the functor of $(\GG_a, \exp)$-invariants. Recall the canonical equivalence 
$$
\Dmod(V) \simeq \inv^{\GG_a, \exp} \big( \Dmod( \GG_a \times V) \big) =: \Dmod(\GG_a \times V)^{\GG_a, \exp}
$$
under which $\oblv^{\GG_a,\exp}$ goes over to $ \exp \boxtimes -: \Dmod(V) \to \Dmod( \GG_a\times V)$.
We show that $\inv^{\GG_a, \exp}$ is symmetric monoidal, as a functor out of $(\GG_a \rrep, \otimes_{(\Dmod(\GG_a), \star)})$: apply $\inv^{\GG_a, \exp}$ to the natural equivalence (of left $\GG_a \rrep$)
$$
\C \usotimes{\Dmod(\GG_a)} \E \simeq \Dmod(\GG_a) \usotimes{\Dmod(\GG_a) \otimes \Dmod(\GG_a)} (\C \otimes \E),
$$
and note that the equivalence $\Vect_\exp \simeq \Vect_\exp \otimes \Vect_\exp$ is $(\GG_a, \GG_a)$-equivariant.

\medskip

Finally, let us set
$$
 \Dmod^{enh} := \inv^{\GG_a, \exp} \circ \Dmod_{\GG_a} \circ \Xi.
$$
This is symmetric monoidal by construction and behaves as claimed on objects and $1$-morphisms. To check the latter, write $\sigma$ as the composition $(m_{\GG_a} \times \id) \circ (\id \times f \times \beta ) \circ (\id \times \Delta)$.
\end{proof}

\sssec{}
Let us exploit this formalism.
By construction, $\FT_V$ is the value of $\Dmod^{enh}$ on the following ``arrow" (which we also call $\FT_V$):
\begin{gather} \label{diag:decorated_corresp_FT}
\xy
(-20,10)*+{ V \overset{\FT_V}\dasharrow V^\vee \, \, \, \, := }="VVvee";
(0,0)*+{ V }="s";
(28,0)*+{ V^\vee. }="t";
(14,10)*+{ V \times V^\vee }="H";
(14,22)*+{ \GG_a }="Ga";
{\ar@{->}_{ p_1 } "H";"s"};
{\ar@{->}^{ p_2 } "H";"t"};
{\ar@{->}^{ Q} "H";"Ga"};
\endxy
\end{gather}

\begin{prop} \label{prop:FT_equiv_fin-dim}
For any finite dimensional vector space $V$, the functors $\FT_V$ and $\IFT_V$ are mutually inverse equivalences of categories.
\end{prop}

\begin{proof}
It suffices to exhibit a natural isomorphism $\IFT \circ \FT \simeq \id_V$. Note that the functor $\IFT[2d_V]$ is the value of $\Dmod^{enh}$ on the arrow $\IFT' := (V^\vee \xleftarrow{\, p_2 \,} V^\vee \times V \xrightarrow{\, p_1 \,} V, -Q)$.

At the level of correspondences, the composition $\IFT'\circ \FT: V \dasharrow V^\vee \dasharrow V$ is given by
$$
(V \xleftarrow{\, \wh p_1 \,} V \times V \times V^\vee \xrightarrow{\, \wh p_2 \,} V, \xi ),
$$
where $\wh p_i : V\ \times V \times V^\vee \to V$ are the projections and $\xi : V \times V \times V^\vee \to \GG_a$ sends $(v,v',\phi) \mapsto \phi(v-v')$.
Let $p_{12}: V \times V \times V^\vee \to V \times V$ the projection to the first two coordinates.
From (\ref{eqn:projection_exp}) and base change, one checks that 
$$
(p_{12})_*\xi^! (\exp) \simeq (\Delta_V)_*(\omega_V)[2d_V].
$$
Then the proof reduces to a simple diagram chase.
\end{proof}
%
%
%

\begin{lem} \label{lem:FT-behavior-linear-maps}
Given a linear map of finite dimensional vector spaces $f: W \to V$ and its dual $\phi: V^\vee \to W^\vee$, the following diagrams are commutative:
\begin{gather} \label{eqn:i_*}
\xy
(00,00)*+{ \Dmod(V) }="00";
(0,15)*+{ \Dmod(W) }="10";
(35,00)*+{ \Dmod(V^\vee) }="01";
(35,15)*+{ \Dmod(W^\vee) }="11";
{\ar@{->}^{ \FT_V } "00";"01"};
{\ar@{<-}_{ f_*  } "00";"10"};
{\ar@{<-}_{\phi^! } "01";"11"};
{\ar@{->}^{ \FT_W} "10";"11"};
(60,0)*+{ \Dmod(V) }="20";
(60,15)*+{ \Dmod(W) }="30";
(97,0)*+{ \Dmod(V^\vee). }="21";
(97,15)*+{ \Dmod(W^\vee) }="31";
{\ar@{<-}^{ \FT_{V^\vee} } "20";"21"};
{\ar@{<-}_{ f_* [2d_V-2 d_W] } "20";"30"};
{\ar@{<-}_{(-\phi)^! } "21";"31"};
{\ar@{<-}^{ \FT_{W^\vee}} "30";"31"};
\endxy
\end{gather} 
\end{lem}

\begin{proof}
We only prove the first commutativity, the second follows by applying the inverse Fourier trasform.
Consider the arrows 
$$
W \overset {(id, f)} { \dasharrow}  V :=\left( W  \xleftarrow{ \, \id \,} W  \xrightarrow{\, f \,} V; 0 \right)
$$
and
$$
W^\vee \overset {(\phi, \id)} { \dasharrow}  V^\vee :=\left( W^\vee  \xleftarrow{ \, \phi \,} V^\vee  \xrightarrow{\, id \,} V^\vee; 0 \right).
$$
It suffices to prove that the following diagram is commutative:
\begin{gather} 
\xy
(0,0)*+{ V \, }="00";
(35,0)*+{ \, V^\vee. }="10";
(00,15)*+{  W  }="01";
(35,15)*+{   W^\vee }="11";
{\ar@{<--}^{ (\id, f) } "00";"01"};
{\ar@{-->}^{ \FT_V } "00";"10"}; 
{\ar@{-->}^{ \FT_W  } "01";"11"};
{\ar@{<--}_{ (\phi, \id) } "10";"11"};
\endxy
\end{gather}
We leave it to the reader to verify that both paths coincide with
$$
\left( W  \xleftarrow{ \, p_1 \,} W \times V^\vee  \xrightarrow{\, p_2 \,} V^\vee; h\right),
$$
where
$h = Q \circ (f \times \id_{V^\vee}) = Q \circ (\id_W \times \phi)$.
\end{proof}

\begin{cor}
$\FT_V$ is a monoidal equivalence between $(\Dmod(V), \star)$ and $(\Dmod(V^\vee), \otimes)$.
Analogously, $\FT_V$ is a comonoidal equivalence between $(\Dmod(V), \Delta_*)$ and $(\Dmod(V^\vee), m^!)$.
\end{cor}

Both statements follow at once from Lemma \ref{lem:FT-behavior-linear-maps}. Nevertheless, here is a direct proof.

\begin{proof}
We only prove the first claim, the second one is obtained by duality (or by a direct argument).
The convolution monoidal structure on $\Dmod(V)$ arises from the algebra structure on $V$ given by
$$
V \times V \overset m { \dasharrow}  V :=\left( V \times V \xleftarrow{ \, \id \,} V \times V  \xrightarrow{\, m \,} V; 0 \right);
$$
the pointwise monoidal structure on $\Dmod(V^\vee)$ from the algebra structure
$$
V^\vee \times V^\vee \overset \Delta { \dasharrow}  V^\vee := \left( V^\vee \times V^\vee \xleftarrow{ \, \Delta \,} V^\vee  \xrightarrow{\, \id \,} V^\vee; 0 \right).
$$
We just need to prove that $\FT: V \dasharrow V^\vee$  intertwines the two algebra structures, or equivalently that the following diagram in $\KL_{\GG_a}$ commutes:
\begin{gather} 
\xy
(0,0)*+{ V \, }="00";
(35,0)*+{ \, V^\vee. }="10";
(00,15)*+{  V \times V  }="01";
(35,15)*+{   V^\vee \times V^\vee }="11";
{\ar@{<--}^{ m } "00";"01"};
{\ar@{-->}^{ \FT } "00";"10"}; 
{\ar@{-->}^{ \FT \times \FT  } "01";"11"};
{\ar@{<--}_{ \Delta } "10";"11"};
\endxy
\end{gather}
Again, it is routine to check that both paths coincide with
$$
\left( V \times V \xleftarrow{ \, p_{12} \,} V \times V \times V^\vee  \xrightarrow{\, m \times id_{V^\vee} \,} V^\vee ; f \right),
$$
where $f: V \times V \times V^\vee \to \GG_a$ sends $(v,w,\phi) \mapsto Q(\phi, w-v)$. 
\end{proof}

Hence, we have
$$
\FT_V : V\on{-}\mathbf{rep} := (\Dmod(V), \star) \mmod \xrightarrow{\, \, \, \simeq \, \, \, } (\Dmod(V^\vee), \otimes) \mmod;
$$
in other words, Fourier transform indentifies categorical representations of $V$ and \emph{crystals of categories} over $V^\vee$, that is, categories with an action of $(\Dmod(V^\vee), \otimes)$.

\medskip

\ssec{Fourier transform for Tate vector spaces}

\renewcommand{\bV}{\mathcal{V}}
\renewcommand{\bW}{\mathcal{W}}

We shall need the notion of Fourier transform for the loop group $\bA := \AA^n \ppart$ of the affine scheme $\AA^n$, thought of as a vector group. More generally, we define the Fourier transform functor for Tate vector spaces and establish its properties, parallel to the ones of the previous section.

\sssec{}
\nc{\PRO}{\mathsf{Pro}}
\nc{\TVect}{\mathsf{TopVect}}

In the sequel, we recall the notion of \emph{Tate vector space}. Our main references are \citep{Beil}, \citep{Dr}, \citep{FG2}. To fix the notation, let $\Vect^\heartsuit$ be the $1$-category of vector spaces, $\PRO(\Vect^\heartsuit)$ its pro-completion, $\Vect^{\ms{f.d.}}$ the $1$-category of finite dimensional vector spaces. 

\medskip

Consider the additive $1$-category $\TVect$ of \emph{topological vector spaces} over $\kk$, where $\kk$ is given the discrete topology. The topology of each object $V \in \TVect$ is supposed to be linear, complete and separated. Morphims in $\TVect$ are simply continuous linear maps.
There is an obvious adjunction
$$
\ms{\delta} : \Vect^\heartsuit \rightleftarrows \TVect : \oblv^{\ms{top}},
$$
where $\delta$ endows a vector space with the discrete topology and $\oblv^{\ms{top}}$ forgets the topology. Note that $\delta$ is fully faithful; its image is referred to as the subcategory of \emph{discrete vector spaces}. 

\sssec{}

According to \cite[Section 19.1]{FG2}, $\TVect$ is the full subcategory of $\PRO(\Vect^\heartsuit)$ consisting of objects that can be realized as projective limits of discrete vector spaces along surjective (linear) maps.
Since $\TVect$ admits projective limits and filtered colimits, each set $\Hom_{\TVect}(V,W)$ is naturally a topological vector space that we shall denote by $\H om(V,W)$.
In particular, the \emph{topological dual} (henceforth simply called dual) of $V$ is the object $V^\vee := \H om(V, \kk) \in \TVect$.
A pro-finite dimensional vector space can be written as $\H om(V,\kk)$, where $V$ is discrete. 

\sssec{}

Let $\Tate$ be the full subcategory of $\TVect$ whose objects are the direct sums $P \oplus Q$, where $P$ is pro-finite dimensional and $Q$ is discrete. Obviously, the operation $\H om(-, \kk)$ transforms discrete into pro-finite and viceversa, hence $\Tate$ admits duals.

\medskip

Write $P = \underset{r, \p}\lim  P^r$ and $Q = \uscolim {n, \i} Q_n$ with finite dimensional $P^r$ and $Q_n$. Setting $V_n^r := P^r \oplus Q_n$, we have
\begin{equation} \label{def:tate}
P \oplus Q \simeq  \underset{r, \p}\lim \, \uscolim {n, \i} V_n^r  \simeq \uscolim {n, \i} \underset{r, \p}\lim V_n^r
\end{equation}
and
\begin{equation} \label{def:tate-dual}
(P \oplus Q)^\vee \simeq  \underset{r, \p^\vee}\colim \, \lim_{n, \i^\vee } (V_n^r)^\vee  \simeq \lim_ {n, \i^\vee} \underset{r, \p^\vee}\colim (V_n^r)^\vee.
\end{equation}
Thus, objects of $\Tate$ are filtered limits of discrete vector spaces along surjections of finite dimension, as well as filtered colimits of pro-finite dimensional vector spaces along injections of finite codimension. The latter characterization makes it clear that Tate vector spaces are the analogue of ind-pro-schemes in the linear algebra setting.

\sssec{}

More precisely, the extension of the natural functor $\Vect^{\ms{f.d.}} \to \Sch^{\fty}$ to $\PRO(\Vect^\heartsuit) \to \PreStk$ maps $\Tate$ to  $\IndSch^{\pro}$. We will henceforth view Tate vector spaces as ind-pro-schemes; in particular, for $\bV = \colim_n \lim_r  V_n^r$ as in (\ref{def:tate}), we have well-defined categories of $\fD$-modules:
$$
\Dmod^*(\bV) \simeq \uscolim{n, \i_*} \lim_{r, \p_*} \Dmod(V_n^r),
\hspace{.5 cm}
\Dmod^!(\bV^\vee) \simeq \lim_{r, (\p^\vee)^!} \uscolim{n, (\i^\vee)^!} \Dmod((V_n^r)^\vee).
$$
We wish to extend the Fourier transform equivalence to Tate vector spaces. This amounts to a combination of a left and a right Kan extension of the usual Fourier transform.

\sssec{}

Consider the usual finite dimensional Fourier transform $\FT_V$ as a functor of $V$:
\begin{equation} \label{eqn:FT-functor of V}
\FT : \Vect^{\mathsf{f. d.}}
\longrightarrow
\Fun(\Delta^1, \DGCat^{\on{SymmMon}}).
\footnote{The $\infty$-category on the right is the one of \emph{arrows in $\DGCat^{\on{SymmMon}}$},
described informally as follows: objects are symmetric monoidal functors of DG categories, and $1$-morphisms are commutative squares of such}.
\end{equation}
At the level of objects, $\FT$ sends a finite dimensional vector space to the equivalence of symmetric monoidal categories categories $(\Dmod(V),\star) \to (\Dmod(V^\vee), \otimes)$. At the level of morphisms, $\FT$ sends the linear map $f: W \to V$ to the natural transformation
\begin{gather}
\xy
(0,0)*+{ \big(\Dmod(V), \star \big) }="00";
(35,0)*+{ \big(\Dmod(V^\vee), \otimes \big),  }="10";
(00,15)*+{ \big(\Dmod(W), \star \big)   }="01";
(35,15)*+{ \big(\Dmod(W^\vee), \otimes \big)   }="11";
{\ar@{<-}^{ f_*} "00";"01"};
{\ar@{->}^{ \FT_V } "00";"10"};
{\ar@{->}^{ \FT_W } "01";"11"};
{\ar@{<-}_{ \phi^! } "10";"11"};
\endxy
\end{gather}
where $\phi = f^\vee : V^\vee \to W^\vee$. That this defines a functor is precisely the content of Lemma \ref{lem:FT-behavior-linear-maps}.
Note that $f_*$ and $\phi^!$ are compatible with the convolution and pointwise symmetric monoidal structures, respectively. 

\sssec{}

\renewcommand{\fd}{\ms{f.d.}}

Let $(\Vect^\fd)^\pro$ be the $1$-category of pro-finite dimensional vector spaces (thought of as pro-schemes). We define the functor
\begin{equation} \label{eqn:FT-functor of V-pro}
\FT : (\Vect^\fd)^\pro
\longrightarrow
\Fun(\Delta^1, \DGCat^{\on{SymmMon}})
\end{equation}
by right Kan extension of (\ref{eqn:FT-functor of V}) along the inclusion $\Vect^{\mathsf{f. d.}} \hookrightarrow (\Vect^\fd)^\pro$.

\medskip
\nc{\indpro}{\ind \text{-} \pro}

To define $\FT$ at the level of Tate vector spaces, we left Kan extend (\ref{eqn:FT-functor of V-pro}) along $(\Vect^\fd)^\pro \hookrightarrow \Tate$:
\begin{equation} \label{eqn:FT-functor of V-indpro}
\FT :  \Tate 
\longrightarrow
\Fun(\Delta^1, \DGCat^{\on{SymmMon}}).
\end{equation}

\sssec{}

Let us unravel the above constructions. If $V = \lim_{r, \pi} V^r \in (\Vect^{\ms{f.d.}})^\pro$ is a pro-finite dimensional vector space, then its dual $V^\vee$ is an ind-vector space of finite type. Denote by $\iota$ the injections dual to $\pi$ and by $(\iota_{r \to \infty})^! : \Dmod(V^\vee) \to \Dmod((V^r)^\vee)$ the tautological evaluation functors. Then, $\FT_V$ is defined by the requirement
\begin{equation} \label{eqn:FT-pro}
(\iota_{r \to \infty})^! \circ \FT_V  =  \FT_{V^r}  \circ (\pi_{\infty \to r})_*.
\end{equation}
Next, let $\bV$ be a Tate vector space, with presentation $\bV = \colim_{n, \i} V_n$ as an ind-scheme, where each $V_n$ is of pro-type. Its dual, $\bV^\vee$ is then written as $\lim_{n, \p} V_n^\vee$, where $\p$ are the projections dual to $\i$. 

\begin{lem} \label{lem:bravo}
Let $(\Vect^\fd)^{\ind }$ denote the $1$-category of vector spaces, considered as ind-schemes of finite type.
The functor $\Dmod^!: \Tate^\op \to \DGCat$ is the left Kan extension of $\Dmod : \big( (\Vect^\fd)^{\ind }\big)^\op \to \DGCat$ along the inclusion $(\Vect^\fd)^{\ind } \hookrightarrow \mathsf{Tate}$.
\end{lem}

\nc{\W}{\mathcal{W}}

\begin{proof}
Recall the presentation of $\W \in \Tate$ as in (\ref{def:tate}). We need to prove that the colimit and the limit in the formula $\Dmod^!(\W) \simeq \lim_{n, \i^!} \colim_{r, \p^!} \Dmod(W_n^r)$ can be switched:
\begin{equation} \label{tate-switch}
\lim_{n, \i^!} \uscolim{r, \p^!} \Dmod(W_n^r) \simeq  \uscolim{r, \p^!} \lim_{n, \i^!} \Dmod(W_n^r).
\end{equation}
Observe that the left adjoint of $\i^!$ is $\i_+$. Since the diagram $(n,r) \rightsquigarrow W_n^r$ obviously has Cartesian squares, the assertion follows from the $(\i_+, \p^!)$ base-change.
\end{proof}

This allows to write both $\Dmod^*(\bV)$ and $\Dmod^!(\bV^\vee)$ as colimits:
$$
\Dmod^*(\bV) \simeq \uscolim{n, \i_*} \Dmod^*(V_n),
\hspace*{.6cm}
\Dmod^!(\bV^\vee) \simeq \uscolim{n,\p^!} \Dmod(V_n^\vee).
$$
By construction, $\FT_{\bV}$ and $\FT_{V_n}$ are intertwined by the insertion functors:
\begin{equation} \label{eqn:FT-indpro}
\FT_\bV \circ (\iota_{n \to \infty})_*   =  (\p_{\infty \to n})^! \circ \FT_{V_n}.
\end{equation}

\begin{lem} \label{lem:importantFTtate}
Let $\bV$ and $\bW$ two Tate vector spaces. Given a linear map $f: \bW \to \bV$ and its dual $\phi : \bV^\vee \to \bW^\vee$, the following diagram is commutative:
\begin{gather}
\xy
(00,00)*+{ \Dmod^*(\bV) }="00";
(0,15)*+{ \Dmod^*(\bW) }="10";
(35,00)*+{ \Dmod^!(\bV^\vee). }="01";
(35,15)*+{ \Dmod^!(\bW^\vee) }="11";
{\ar@{->}^{ \FT_\bV } "00";"01"};
{\ar@{<-}_{ f_*  } "00";"10"};
{\ar@{<-}_{\phi^! } "01";"11"};
{\ar@{->}^{ \FT_\bW} "10";"11"};
\endxy
\end{gather} 
Moreover, $(\FT_\V)^\vee \simeq \FT_{\V^\vee}$.
\end{lem}

\begin{proof}
Both statements are obtained by left-right Kan extending the corresponding statements in the finite dimensional case.
\end{proof}

\begin{rem} \label{rem:FT-ren-pro}
By left Kan extending (\ref{eqn:IFT-Q}), we deduce that, for a pro-scheme $V$ with dual $V^\vee$, we have $\FT^Q_V \circ \lambda_V \simeq \IFT^{-Q}_{V^\vee}$.
\end{rem}

As a corollary of Lemma \ref{lem:importantFTtate}, we obtain:

\begin{thm} \label{thm:mon-comon-Tate}
For any Tate vector space $\bV$, Fourier transform yields the monoidal equivalence
$$
\FT_{\bV}: (\Dmod^*(\bV), \star) \xrightarrow{\,\, \simeq \,\,} (\Dmod^!(\bV^\vee), \otimes),
$$
as well as the comonoidal equivalence
$$
\FT_{\bV}: (\Dmod^*(\bV), \Delta_*) \xrightarrow{\,\, \simeq \,\,} (\Dmod^!(\bV^\vee), m^!).
$$
\end{thm}

\ssec{Categories tensored over ind-pro-schemes}

The above theory motivates the study of \emph{categories tensored over} a Tate vector space $\bW$, i.e. objects of $(\Dmod^!(\bW), \otimes) \mmod$. In this section, we look at the $\infty$-category $(\Dmod^!(\X), \otimes)$ where $\X$ is an arbitrary ind-pro-scheme, and record some of its functoriality to be used in Section \ref{SEC:trans_act}.

For clarity, given $\C$ tensored over $\X$, we indicate by $\diam$ the action $\Dmod^!(\X) \otimes \C \to \C$.

\sssec{}

Recall that, for any map $f: \Y \to \X$ of ind-pro-schemes, $f^!: \Dmod^!(\X) \to \Dmod^!(\Y)$ is $\Dmod^!(\X)$-linear. Thus, for any $\C \in \Dmod^!(\X) \mmod$, we define the \emph{corestriction} and the \emph{restriction} of $\C$ along $f$:
$$
\restr \C \Y := \Dmod^!(\Y) \usotimes{\Dmod^!(\X)} \C
\hspace{.6cm}
\uprestr \C \Y := \Hom_{\Dmod^!(\X)} (\Dmod^!(\Y), \C).
$$
(When $f: \pt \to \X$, we call the above categories \emph{cofiber} and \emph{fiber}, respectively.)
Applying the paradigm of Sect. \ref{paradigm}, we obtain the tautological functors 
$$
(f^!)_\C : \C \to \restr \C \Y
\text{\, \, \, and \, \, }
(f^!)^\C : \uprestr \C \Y \to \C.
$$
\begin{lem} \label{lem:useful}
Let $\iota: Y \hto X$ be a finitely presented closed embedding of pro-schemes and $\C \in \Dmod^!(X)$. Then: 
\begin{itemize}

\item
there are adjuctions
$$
(\iota_+)_\C: \restr \C Y \rightleftarrows \C: (\iota^!)_\C
\hspace{.6cm}
(\iota^!)^\C : \uprestr \C Y \rightleftarrows \C : (\iota_+)^\C,
$$
with both left adjoints fully faithful;

\medskip

\item

$\iota_+(\omega_Y) \diam -: \C \to \C$ induces a functor that we name $\theta_{Y \hto X}: \restr \C Y \to \uprestr \C Y$;

\medskip

\item

$\theta_{Y \hto X} \simeq (\iota_+)^\C \circ (\iota_+)_\C$ is an equivalence with inverse $(\iota^!)_\C \circ (\iota^!)^\C: \uprestr \C Y \to \restr \C Y$.

\end{itemize}

\end{lem}

\begin{proof}
For the first claim, it suffices to note that $\iota_+$ is $\Dmod^!(X)$-linear (by base-change) and fully faithful. 
Secondly, the canonical equivalence of functors $\iota_+(\omega_Y) \otimes - \simeq \iota_+ \circ \iota^! : \Dmod^!(X) \to \Dmod^!(X)$ shows that $\iota_+(\omega_Y) \diam -$ factors as 
$$
\C \twoheadrightarrow \restr \C Y \xto{\, \, \theta_{Y \hto X} \, \,} \uprestr \C Y \hto \C.
$$
It remains to check that $\theta := \theta_{Y \hto X}$ is an equivalence. This can be done directly, by testing that the two compositions are the identity functors. 
\end{proof}

\sssec{}

Let now $i_x: \{x\} \hto \X$ be the inclusion of a point into an ind-pro-scheme $\X$. Assume that $\X$ has been given a dimension theory. 
This allows to define the functor $(i_x)_*^\ren := \Lambda^{-1} \circ (i_x)_*: \Vect \to \Dmod^!(\X)$ and the object $(\delta_x)^\ren := \Lambda^{-1}(\delta_x) \in \Dmod^!(\X)$. By the projection formula (Proposition \ref{prop:proj-formula-indschemes}), we deduce the following fact.

\begin{lem}
The functor $(i_x)_*^\ren$ is $\Dmod^!(\X)$-linear. Informally, this means that tensoring $(\delta_x)^\ren$ with an object of $\Dmod^!(\X)$ is the same as tensoring $(\delta_x)^\ren$ with the $!$-fiber at $x$ of that object.
\end{lem}

Hence, $(\delta_x)^\ren \diam - :\C \to \C$ yields a functor
\begin{equation}
\label{eqn:Theta_x}
\Theta_{x \hto X} : \restr \C x \to \uprestr \C x.
\end{equation}

\begin{lem} \label{lem:FT-of-omega-ren}
Let $\bV$ be Tate vector space and $\m$ a point in $\bV^\vee$. For any pair of companion dimension theories\footcite{I.e., if $\Lambda_\bV$ corresponds $\dim (V_n^r) =0$ for some indices $n$ and $r$, then $\Lambda_{\bV^\vee}$ corresponds to $\dim ((V_n^r)^\vee) = 0$} on $\bV$ and $\bV^\vee$, we have
\begin{equation}
\FT_\bV \big(\Lambda((-\m)^! exp ) \big) \simeq (\delta_{\m, \bV^\vee})^\ren \in \Dmod^!(\bV^\vee).
\end{equation}
\end{lem}

\begin{proof}
Let $\bV$ be presented as in (\ref{def:tate}).
By definition,
$$
\FT_\bV \big(\Lambda((-\m)^! exp ) \big)
\simeq
\colim_n \FT_{V_n} \circ \lambda_{V_n} \big( (-\mu)^! \exp \big) [2 \dim(V_n)].
$$
Thanks to Remark \ref{rem:FT-ren-pro} and $\FT (\delta_\m) \simeq \m^!(\exp)$, we know that 
$$
\FT_{V_n} \circ \lambda_{V_n} \big( (-\mu)^! \exp \big) \simeq \delta_{\chi, V_n^\vee}.
$$
The assertion follows, after tracing through the equivalence (\ref{tate-switch}) for $\W = \V^\vee$.
\end{proof}

\ssec{Invariants and coinvariants via Fourier transform}

Let $\bV \in \Tate$. Suppose that $\bV$ acts on $\C$; as usual, we indicate by $\star$ the action of $\Dmod^*(\bV)$ on $\C$.
We wish to express the invariant and coinvariant categories $\C^\bV$ and $\C_\bV$ in terms of the action of $(\Dmod^!(\bV^\vee),\otimes)$ on $\C$, which is given tautologically by
\begin{equation} \label{eqn:action_IFT}
\Dmod^!(\bV^\vee) \otimes \C \to \C, \, \, \, \, P \otimes c \mapsto \IFT_{\bV}(P) \star c.
\end{equation}

\begin{lem} \label{lem:coinv_FT-chi}
Let $\C$ be endowed with an action of $\Dmod^*(\V)$ and $\m: \V \to \GG_a$ be a character.
Under Fourier transform,
$$
\C_{\V, \chi} \simeq \restr \C {\chi}
\hspace*{0.6cm}
\C^{\V,\chi} \simeq \uprestr {\C}{\chi},
$$
where $\restr \C {\chi}$ and $\uprestr \C {\chi}$ denote the cofiber and fiber of $\C$ along $(i_\chi)^!: \Dmod^!(\V^\vee) \to \Dmod(\chi) \simeq \Vect$.
\end{lem}

\begin{proof}
It suffices to prove that the action of $\Dmod^*(\V)$ on $\Vect$ determined by the comonoidal functor
\begin{equation} \label{eqn:comocomo}
\Vect \to (\Dmod^!(\V), m^!), \hspace*{0.3cm} \kk \mapsto \chi^!(\exp)
\end{equation}
goes over to the action of $(\Dmod^!(\V^\vee), \Delta^!)$ on $\Vect$ given by $(i_\chi)^!$. 

By Fourier transform, (\ref{eqn:comocomo}) goes over to the comonoidal functor
\begin{equation} \label{eqn:comocomo1}
\Vect \to (\Dmod^*(\V^\vee), \Delta_*), \hspace*{0.3cm} \kk \mapsto \delta_{\chi, \V^\vee}.
\end{equation}
This is a quick consequence of Lemma \ref{lem:importantFTtate} and Theorem \ref{thm:mon-comon-Tate}, after recalling that $\FT_{\GG_a}(\delta_1) \simeq \exp$. To conclude, observe that the duality (\ref{eqn:eval-compact}) between $\Dmod^*(\V^\vee)$ and $\Dmod^!(\V^\vee)$ transforms (\ref{eqn:comocomo1}) into the desired monoidal functor $(i_\chi)^!$.
\end{proof}

\sssec{}

\renewcommand{\bV}{\mathcal{V}}
\renewcommand{\bW}{\mathcal{W}}

Consider a category $\C$ equipped with an action of a Tate vector space $\bV$. Let $\bW \hookrightarrow \bV$ be a Tate vector subspace. We wish to describe the procedure of taking twisted (co)invariants of $\C$ with respect to $\bW$ via Fourier transform. 
Let $\p: \bV^\vee \to \bW^\vee$ be the projection dual to $\bW \hookrightarrow \bV$ and $\bW^\perp \subset \bV^\vee$ the annihilator of $\bW$ in $\bV^\vee$. 

\begin{prop} \label{prop:FT_char_subspace_indpro}
Let $\chi \in V^\vee$ be a character. Under Fourier transform,
$$
\C_{\bW,\chi} \simeq \Dmod^!(\bW^\perp + \{ \chi \}) \underset{\Dmod^!(\V^\vee)}\otimes \C,
\, \, \, \, \, \, \,\,\,
\C^{\bW,\chi} \simeq \Hom_{\Dmod^!(V^\vee)} \big( \Dmod^!(\bW^\perp + \{ \chi \}), \C \big),
$$
where $\Dmod^!(V^\vee)$ acts on $\Dmod^!(\bW^\perp + \{\chi \})$ via $!$-pullback along the inclusion $\bW^\perp + \{\chi \} \subseteq V^\vee$. 
\end{prop}

\begin{proof}
Using the standard tensor-hom adjunction, we see that it suffices to prove the equivalence
$$
\Dmod^*(\bV)_{\bW, \chi} \simeq \Dmod^!(\{\chi \} +\bW ^\perp).
$$ 
By the lemma above, we have
$$
\Dmod^*(\bV)_{\bW, \chi} \simeq \Dmod(\chi) \underset{\Dmod^!(\bW^\vee)}\otimes \Dmod^!(\bV^\vee)
$$
and we conclude by Lemma \ref{lem:BN-ind}, since $\chi \times_{\bW^\vee} \bV^\vee \simeq \chi + \bW^\perp$.
\end{proof}

\renewcommand{\bV}{\mathbf{V}}
\renewcommand{\bW}{\mathbf{W}}

\section{Categories over quotient stacks} \label{SEC:trans_act}

Let $X/G$ be a quotient stack, where $X$ and $G$ are both of finite type. We have seen (Proposition \ref{prop:ShvCat(X/G)}) that a crystal of categories over $X/G$ is the same as a module category for $\Dmod(G) \ltimes \Dmod(X)$. The latter espression admits an obvious generalization when $G$ and $X$ are of ind-pro-finite type, which we introduce in Sect. \ref{ssec:6.1}. 

\medskip

A category over $X/G$ is a category tensored over $X$ in a $G$-equivariant way. For such $\C$, we analyze the relations between (co)invariants with respect to subgroups of $G$ and (co)restrictions along maps $Y \to X$.
In Sect. \ref{ssec:transitive-group-acts}, we specialize to the case where $G$ acts transitively on $X$. Then the invariant category $\C^G$ is equivalent to the $\Stab_G(x)$-invariant category of the restriction of $\C$ to a point $x \in X$.

\ssec{The main definitions} \label{ssec:6.1}

Let $\G$ be an ind-pro-group and $\X$ an ind-pro-scheme endowed with an action of $\G$. Then $\Dmod^!(\X)$ is an algebra object in the monoidal $\infty$-category $\Dmod^!(\G) \ccomod$, whence we can form the crossed product monoidal category $\Dmod^*(\G) \ltimes \Dmod^!(\X)$.

\medskip

We say that $\C \in \DGCat$ is \emph{a category tensored over $\X/\G$} if it is endowed with the structure of a module category for $\Dmod^*(\G) \ltimes \Dmod^!(\X)$,
where we recall that
$$
\Dmod^*(\G) \ltimes \Dmod^!(\X) \mmod \simeq \Dmod^!(\X) \mod (\Dmod^!(\G) \ccomod).
$$

\sssec{}

Recall that we indicate by $M \diamond c$ the action of $M \in \Dmod^!(\X)$ on $c \in \C$, while the action of $\Dmod^*(\G)$ on $\C$ is the usual $\star$.
By construction, the actions of $\G$ and $\Dmod^!(\X)$ on $\C$ are compatible in the following way: for each $M$ and $c$ as above, there is a canonical isomorphism
\begin{equation} \label{eqn:coact_compatib}
\coact_\G ( M \diamond c) \simeq \act_{\G,\X}^!(M) \diamond \coact_\G (c).
\end{equation}
To be precise, the symbol $\diamond$ in the RHS means $\otimes$ on the $\Dmod^!(\G)$-factor and action of $\Dmod^!(\X)$ on $\C$.

\begin{example} \label{example:theta-G-inv}
Let  $f:\Y \to \X$ be $\G$-equivariant. Then, $\Dmod^!(\Y)$ and $\Dmod^!(\X)$ are tensored over $\X/\G$ in the natural way. The functor $f^!: \Dmod^!(\X) \to \Dmod^!(\Y)$ is $\Dmod^*(\G) \ltimes \Dmod^!(\X)$-linear. 
\end{example}

\begin{lem} \label{lem:pullback-equivariant}
In the situation of the example above, let $\C$ be tensored over $\X/\G$. Then the arrows $(f^!)_\C: \restr \C \X \to \restr \C \Y$ and $(f^!)^\C: \uprestr \C \Y \to \uprestr \C \X$ are $\G$-equivariant.
\end{lem}

\begin{proof}
For any $\M \in \Dmod^!(\X) \mod (\Dmod^!(\G) \ccomod)$, the functor $ \Hom_{\Dmod^!(\X)}(\M, -): \Dmod^!(\X) \mmod \to \DGCat$ upgrades naturally to a functor 
$$
\Dmod^!(\X) \mod (\Dmod^!(\G) \ccomod) \to  \Dmod^!(\G) \ccomod,
$$
and similarly for the relative tensor product. Then, the assertions follow from the example above applied to $\M = \C$.
\end{proof}

\begin{cor} \label{cor:theta_inv}
Let $\iota: Y \hto X$ be a $\G$-equivariant finitely presented closed embedding of pro-schemes. For any $\C$ tensored over $X/\G$, the functor $\theta_{Y \to X} : \restr \C Y \to \uprestr \C Y$ is $\G$-equivariant.
\end{cor}

\begin{proof}
Lemma \ref{lem:useful} guarantees that the inverse of $\theta_{Y \to X}$ is the composition $(\iota^!)_\C \circ (\iota^!)^\C $, which is $\G$-equivariant by Lemma \ref{lem:pullback-equivariant}.
\end{proof}

\begin{lem} \label{lem:stupid_cat_over_X/G}
Let $\G$ be a pro-group and $\X$ an ind-pro-scheme with $\G$-action. For $\C$ a category tensored over $\X/\G$, we have:
\begin{itemize}

\item[(a)]
Let $S \subset \G$ a closed subgroup and $f: \Y \to \X$ an $S$-equivariant map. (Note that we do not demand that $\G$ acts on $\Y$.) Then $\restr \C \Y$ is tensored over $\Y/S$  and the following diagram is commutative:
\begin{gather} \label{diag:X/G<---Y/K}
\xy
(0,18)*+{ \C^\G }="00";
(00,00)*+{  \C }="10";
(33,18)*+{ (\restr \C \Y)^S}="01";
(33,00)*+{ \restr \C \Y.  }="11";
{\ar@{->}^{(f^!)_\C \circ \oblv^\rel  } "00";"01"};
{\ar@{->}^{\oblv^\G} "00";"10"};
{\ar@{->}^{ (f^!)_\C} "10";"11"};
{\ar@{->}^{\oblv^S} "01";"11"};
\endxy
\end{gather}

\smallskip 

\item[(b)]
If $\G$ acts trivially on $\X$, then $\C^\G$ and $\C_\G$ are tensored over $\X$ compatibly with $\oblv^{\G}: \C^\G \to \C$ and $\pr_{\G}: \C \to \C_\G$.

\smallskip

\item[(c)]
Let $f: \Y \to \X$ be a $\G$-equivariant map with $\G$ acting trivially on both spaces. Then taking $\G$-invariants commutes with restriction to $\Y$: that is, there is a canonical equivalence $(\restr \C \Y)^\G \simeq (\restr{\C^\G)} \Y$.

\end{itemize}
\end{lem}

\begin{proof}
The first two assertions are obvious. The third one is an immediate consequence of the fact that $\inv^\G$ commutes with colimits when $\G$ is a pro-group.
\end{proof}

\sssec{}

We pause to give one important example of category over a quotient stack.
Let $\G$ is an ind-pro-group acting on a Tate vector space $\V$. 
If $\G \ltimes \V$ acts on a category $\C$, then 

\begin{itemize}
\item
$\C$ fibers over $ \V^\vee$, by Fourier transform;
\item
$\G$ acts on $\V^\vee$, via the dual action;
\item
$\G$ acts on $\C$ via the embedding $\G \to \G \ltimes \V$.
\end{itemize}
The proposition below makes it precise that these three pieces of data are compatible:

\begin{prop} \label{prop:KV^vee-module}
With the above notation, the Fourier transform induces a monoidal equivalence 
$$
\on{id}_{\Dmod^*(\G)} \otimes \FT_{\V} : \Dmod^*(\G \ltimes \V) \xrightarrow{ \, \, \simeq \, \,} \Dmod^*(\G) \ltimes \Dmod^!(\V^\vee).
$$
In particular, the $\infty$-category $\G \ltimes \V \rrep$ is equivalent to the $\infty$-category of categories tensored over $\V^\vee/\G$.
\end{prop}

\begin{proof}
It suffices to notice that $\Dmod^*$ preserves $\ltimes$, so that $\Dmod^*(\G \ltimes \V) \simeq \Dmod^*(\G) \ltimes \Dmod^*(\V)$.
\end{proof}

\begin{example}
Let $P \simeq GL_{n-1} \ltimes \AA^{n-1} \subset GL_n$ be the mirabolic group of $GL_n$. We obtain an equivalence between $P \ppart \mmod$ and the $\infty$-category of categories tensored over $\AA^{n-1} \ppart / GL_{n-1} \ppart$, where $GL_{n-1} \ppart$ acts on $\AA^{n-1} \ppart$ via the dual action.
\end{example}

\ssec{Interactions between invariants and corestrictions, I}

Let us further analyze the situation of Lemma \ref{lem:stupid_cat_over_X/G}, item (a).
In this section, we discuss the finite dimensional situation. Let $G$ be a group of finite type, $X$ be $G$-scheme of finite type and $\iota: Y \hookrightarrow X$ a closed embedding. Let $S \subset G$ be a closed subgroup that preserves $Y$.

Our goal is to establish conditions under which the top horizontal arrow $\rho: \C^G \to (\restr \C Y)^S$ of diagram (\ref{diag:X/G<---Y/K}) is an equivalence.

\sssec{}

By construction, we have the commutative diagram
\begin{gather}  \label{diag:amico}
\xy
(0,18)*+{ X/G }="00";
(00,00)*+{  X }="10";
(33,18)*+{ Y/S }="01";
(33,00)*+{ Y.  }="11";
{\ar@{<-}^{\wt\iota} "00";"01"};
{\ar@{<-}^{q^X} "00";"10"};
{\ar@{<-}_{ \iota } "10";"11"};
{\ar@{<-}^{q^Y} "01";"11"};
\endxy
\end{gather}

\begin{prop} \label{prop:X/G=Y/S}
Let $\C$ be a category tensored over $X/G$. Suppose that the map $\wt\iota$ above is an isomorphism of stacks $Y/S \xrightarrow{\, \simeq \, } X/G$. 
Then the functor $(\iota^!)_\C \circ \oblv^\rel: \C^{G} \to \C^S \to  (\restr \C Y)^{S}$ is an equivalence making the following diagram commutative:
\begin{gather} \label{diag:bello}
\xy
(0,18)*+{ \C^{G} }="00";
(00,00)*+{  \C }="10";
(33,18)*+{ (\restr \C Y)^{S}}="01";
(33,00)*+{ \restr \C Y.  }="11";
{\ar@{->}^{\simeq}_{\rho} "00";"01"};
{\ar@{->}^{\oblv^G} "00";"10"};
{\ar@{->}_{(\iota^!)_\C} "10";"11"};
{\ar@{->}^{\oblv^S} "01";"11"};
\endxy
\end{gather}
\end{prop}

\begin{proof}
By Proposition \ref{prop:ShvCat(X/G)}, there exists $\wt\C \in \ShvCat(X_\dR/G_\dR)$ such that $\C := \bGamma(X_\dR, \wt\C)$.
We apply the contravariant functor $\bGamma((-)_\dR, \wt\C)$ to (\ref{diag:amico}). Tautologically, $\bGamma((X/G)_\dR, \wt\C) \simeq \C^G$ and pull-back along $q^X$ (resp., $q^Y$) is $\oblv^G$ (resp., $\oblv^S$) by construction. On the other hand, pull-back along $\iota$ is the functor $(\iota^!)_\C : \C \to \Dmod(Y) \otimes_{\Dmod(X)} \C$, as claimed.
\end{proof}

\begin{rem} \label{rem:char-unnecessary}
Let $\mu: G \to \GG_a$ be a character. Then, $\C \otimes \Vect_{-\m}$ also fibers on $X/G$. The above proposition admits an obvious variant involving $\mu$-twisted invariants (Corollary \ref{cor:inverse_with_Av_!}). In every proof of Section \ref{SEC:trans_act}, it suffices to treat $\m = 0$ (the general case follows from the change $\C \rightsquigarrow \C \otimes \Vect_{-\m}$). We will tacitly assume that this step has been performed.
\end{rem}

To summarize the above discussion, we record:

\begin{cor} \label{cor:inverse_with_Av_!}
In the situation of Proposition \ref{prop:X/G=Y/S}, let $\mu: G \to \GG_a$ be a character.
The functor
$$
\Av_!^{G/S} \circ (\iota_*)_\C : (\restr \C Y)^{S, \mu} \to \C^{G, \mu} : \rho= (\iota^!)_\C \circ \oblv^\rel
$$
are mutually inverse equivalences of categories.
\end{cor}

Here and later in this section, $\Av_\heartsuit^{G/S} : \C^{S, \m} \to \C^{G, \m}$ is the relative $\heartsuit$-averaging functor for $\heartsuit=!$ or $\heartsuit=*$, whereas $\oblv^\rel: \C^{G,\m} \to \C^{S,\m}$ is the relative forgetful functor.

\begin{proof}
The functor from left to right is left adjoint (hence inverse) to the equivalence $\rho$.
\end{proof}

In particular, setting $Y = \{x\}$, we obtain:

\begin{cor} \label{cor:transitive_action_fd}
Let $\C$ be a category tensored over $X/G$, with $G$ acting transitively on $X$. Let $x \in X$ be a point and $S = \Stab(x \in X)$. Then the pull-back functor along $\iota: \{x\} \hookrightarrow X$ yields an equivalence 
$$
\rho: \C^{G, \mu}  \xto{\oblv^\rel} \C^{S, \m} \xto{(\iota^!)_\C} (\restr \C x)^{S, \mu}.
$$
\end{cor}

The following result will be important in the sequel:

\begin{prop} \label{prop:inverse_with_Av_*-shift}
In the situation of Corollary \ref{cor:inverse_with_Av_!}, the functor
$$
\Av_*^{G/S} \circ (\iota_*)_\C [2 \codim(G:S)]: (\restr \C Y)^{S, \mu} \to \C^{G, \mu}
$$
is inverse to $\rho = (\iota^!)_\C \circ \oblv^\rel$. 
In particular, we obtain a functorial identification 
\begin{equation} \label{eq:ident-shift}
\Av_!^{G/S} \circ (\iota_*)_\C \simeq \Av_*^{G/S} \circ (\iota_*)_\C [2 \codim(G:S)].
\end{equation}
\end{prop}

\begin{proof}
Let us name $\tau$ the displayed functor and set $D := 2 \codim(G:S)$. It suffices to prove that $\tau \circ \rho$ is the identity of $\C^G$. By the compatibility between the $\Dmod(X)$-action and the $\Dmod(G)$-coaction, we have
$$
\tau \circ \rho 
\simeq
\Av_*^G  \circ (\iota_*)_\C \circ (\iota^!)_\C \circ \oblv^G [D]
\simeq
\Av_*^G \big(\iota_*(\omega_Y) \diam \oblv^G \big)[D]
\simeq
\Av_*^G (\iota_*(\omega_Y)) \diam \oblv^G [D].
$$
Hence, it suffices to prove the proposition for $\C = \Dmod(X)$, with its natural structure of category tensored over $X/G$.

\medskip

More generally, we will prove it for $\C = \Dmod(E)$, where $E$ is a scheme mapping $G$-equivariantly to $X$. Let $F:= Y \times_X E$ and (abusing notation) $\iota: F \hto E$ the closed embedding induced by $Y \hto X$. Consider the diagram
\begin{gather}
\xy
(00,0)*+{ E/G }="00";
(30,0)*+{ F/S. }="10";
(0,15)*+{ E }="01";
(30,15)*+{ F }="11";
{\ar@{<-}_{ q } "00";"01"};
{\ar@{<-}^{ \wt \iota  }_{\simeq} "00";"10"};
{\ar@{<-}^{ \iota } "01";"11"};
{\ar@{<-}_{  q' } "10";"11"};
\endxy
\end{gather}
Under the equivalence $\Dmod(E)^G \simeq \Dmod(E/G)$, the adjuncton $(\oblv^G, \Av_*^G)$ becomes $q^! : \Dmod(E/G) \rightleftarrows \Dmod(E) : q_*[-2d_G]$, while $\Av_!^G \simeq q_!$. Also, by smoothness,
$$
\iota^! \circ q^! 
\simeq
(q')^! \circ (\wt\iota)^!
 \simeq
(q')^* \circ (\wt\iota)^* [2 d_S]
 \simeq
 \iota^* \circ q^* [2 d_S]. 
$$
Hence, the original equivalence $\rho$ becomes
$$
\rho \simeq (q')_! \circ  \iota^* \circ q^* [2 d_S],
$$
whose right adjoint is evidently $\tau$.
\end{proof}

\ssec{Interactions between invariants and corestrictions, II}

Our goal now is to extend Proposition \ref{prop:X/G=Y/S} and its corollaries to the pro-finite dimensional setting.

\sssec{} \label{sssec:quotient-wise}

For $H$ a pro-group, denote by $\Sch_H^\pro$ the $1$-category whose objects are pro-schemes equipped with a $H$-action and whose morphisms are $H$-equivariant maps.

It is clear that any $X \in \Sch_H^\pro$ admits a presentation $X = \lim_r X^r$ where each $X^r$ is acted on by $H$ and the transition maps $\pi_{s \to r}$ are $H$-equivariant. Such a presentation will be called \emph{H-compatible}.

\sssec{} \label{sssec:setup}

Let us fix the following set-up:
\begin{itemize}
\item[(i)]
$H$ is a pro-group endowed with a character $\mu: H \to \GG_a$;
\item[(ii)]
$X$ is an element of $\Sch_H^\pro$ equipped with an $H$-compatible presentation $X= \lim_r X^r$;
\item[(iii)]
$\iota: Y \hookrightarrow X$ is a closed embedding of finite presentation;
\item[(iv)]
$S \subseteq H$ is the subgroup preserving $Y$;
\item[(v)]
we assume that the natural map $Y/S \to X/H$ is an \emph{isomorphism}.
\end{itemize}

\medskip

Given the first four items, we may form the functor
$$
\rho_Y: \C^{H,\m} \xto{\, \oblv^{\rel} \,} \C^{S,\m} \xto { \, (\iota^!)_\C \,} (\restr \C Y)^{S,\m},
$$
where we have used that $(\iota^!)_\C$ is $S$-equivariant (Lemma \ref{lem:pullback-equivariant}).

\medskip

Item (v) implies that $S \hto H$ is a \emph{finite presented} closed embedding of pro-groups, whence $\codim(H:S)$ is well-defined. Consider the functor
$$
 \sigma_Y : 
(\restr \C Y)^{S,\m} \xto { \, (\iota_+)_\C \,} \C^{S,\m} \xto{\, \,\Av_*^{H/S} \, \,} \C^{H,\m}.
$$

\begin{prop} \label{prop:X/G=Y/S_pro}
In the set-up on Sect. \ref{sssec:setup}, $\rho_Y$ and $\sigma_Y[2 \codim(H:S)]$ are mutually inverse equivalences of categories.
\end{prop}

\begin{proof}
By definition, $\iota$ is the pull-back of a closed embedding $\iota^r: Y^r \hookrightarrow X^r$. Thanks to Lemma \ref{lem:BN}, we have
$$
\restr \C Y \simeq \Dmod(Y^r) \usotimes{\Dmod(X^r)} \C,
$$
where $\Dmod(X^r)$ acts on $\C$ via the monoidal functor $(\pi_{\infty \to r})^!: \Dmod(X^r) \to \Dmod^!(X)$.
For any $r \in \R$, pick a normal subgroup $H^r \subseteq H$ such that the action of $H$ on $X^r$ factors through the finite dimensional quotient $H \twoheadrightarrow H/H^r$. Obviously, $H^r$ is normal in $S$ and $S/H^r$ acts on $Y^r$.

By Lemma \ref{lem:auxiliary} and Lemma \ref{lem:stupid_cat_over_X/G}c, we obtain
$$
(\restr \C Y)^S 
\simeq
\Big( \Dmod(Y^r) \usotimes{\Dmod(X^r)} \C \Big)^S
 \simeq
\Big(  \Dmod(Y^r) \usotimes{\Dmod(X^r)} \C^{H^r} \Big)^{S/H^r}.
$$
The hypothesis implies that $X^r / (H/H^r) \simeq Y^r/(S/H^r)$ and $\C^{H^r}$ fibers over $X^r / (S/H^r)$ (see Lemma \ref{lem:stupid_cat_over_X/G}a), whence we obtain the equivalence
$$
\rho^{\fty} : (\C^{H^r})^{H/H^r} \rightleftarrows \Big(  \Dmod(Y^r) \usotimes{\Dmod(X^r)} \C^{H^r} \Big)^{S/H^r}: \sigma^{\fty}[ 2\, \codim(H:S)],
$$
from Lemma \ref{prop:X/G=Y/S} and Proposition \ref{prop:inverse_with_Av_*-shift}.
Consequently, 
$$
\C^H \simeq  (\C^{H^r})^{H/H^r} 
\overset { \rho^{\fty} \, \, } {\underset{\sigma^{\fty}[ 2\, \codim(H:S)]} {\rightleftarrows} }
 \Big(  \Dmod(Y^r) \usotimes{\Dmod(X^r)} \C^{H^r} \Big)^{S/H^r}
\simeq
(\restr \C Y)^S.
$$
That these two functors are exactly $\rho$ and $\sigma[ 2\, \codim(H:S)]$ can be traced immediately.
\end{proof}

From the isomorphism $(\iota^!)^\C \circ \theta_{Y \hto X} \simeq (\iota_+)_\C$, we deduce:

\begin{cor}
In the situation described in Sect. \ref{sssec:setup}, the three equivalences
\begin{gather} \label{triangle}
\xy
(25,25)*+{ \C^{H,\m} }="A";
(00,00)*+{ (\restr \C Y)^{S,\m} }="B";
(50,00)*+{ \big(\!\uprestr \C Y \! \big)^{S,\m} }="C";
{\ar@{->}_{ \rho_Y = (\iota^!)_\C \circ \, \oblv^\rel } "A";"B"};
{\ar@{->}^{ \theta_{Y \hto X}[2\codim(H:S)] } "B";"C"};
{\ar@{->}_{ \sigma_Y = \Av_*^{H/S} \circ \, (\iota^!)^\C } "C";"A"};
\endxy
\end{gather} 
form a commutative diagram.
\end{cor}

\ssec{Transitive groups actions} \label{ssec:transitive-group-acts}

We wish to generalize Corollary \ref{cor:transitive_action_fd} to the pro-finite dimensional setting. This is not implied by the above section as $x \hto X$ is not of finite presentation whenever $X$ is infinite dimensional. 
Instead, we will consider a family of triangles as above with $Y$ shrinking to $x$ and deal with a convergence problem.

\sssec{} \label{sssec:setup-transitive}

The set-up for the present subsection in the following:
\begin{itemize}
\item[(i)]
$H$ is a pro-group endowed with a character $\mu: H \to \GG_a$;
\item[(ii)]
$X \in \Sch_H^\pro$, equipped with an $H$-compatible presentation $X= \lim_r X^r$;
\item[(iii)]
$\iota: \{x\} \hookrightarrow X$ be the inclusion of a point into $X$;
\item[(iv)]
$S \subseteq H$ is the stabilizer of $x$;
\item[(v)]
the $H$-action on $X$ is transitive, i.e., that the natural map $\{x\}/S \to X/H$ is an isomorphism. 
\end{itemize}
Our goal is to prove the following theorem:

\begin{thm}
In the situation described in Sect. \ref{sssec:setup-transitive}, the three functors
\begin{gather} \label{triangle-transitive}
\xy
(25,25)*+{ \C^{H,\m} }="A";
(00,00)*+{ (\restr \C x)^{S,\m} }="B";
(50,00)*+{ \big(\!\uprestr \C x \! \big)^{S,\m} }="C";
{\ar@{->}_{ \rho_x := (\iota^!)_\C \circ \, \oblv^\rel } "A";"B"};
{\ar@{->}^{ \Theta_{x \hto X} } "B";"C"};
{\ar@{->}_{ \sigma_x := \Av_*^{H/S} \circ \, (\iota^!)^\C } "C";"A"};
\endxy
\end{gather} 
form a commutative diagram of \emph{equivalences}. Here, $\Theta_{x \hto X}$ is the functor (\ref{eqn:Theta_x}), associated to the dimension theory of $X$ specified in Remark \ref{rem:dim_thry-chosen}.
\end{thm}

We shall proceed in several steps. The proof will be given in Propositions \ref{prop:X/H---x/S}, \ref{prop:sigma_equiv} and \ref{ZC}.

\sssec{}

We start with a general result concerning the inclusion of a point $x$ into a pro-scheme $X = \lim_r X^r$.

\medskip

For each $r \in \R$, define the pro-scheme $Y^r := \{x^r\} \times_{X^r} X$, where $x^r = \pi_{\infty \to r}(x)$. The induced closed embedding $\iota_r : Y^r \hto X$ is finitely presented, as it is pulled back from $x^r \hto X^r$. 
For any arrow $r \to s$ in $\R$, there is a closed embedding $\iota_{r,s} : Y^r \hto Y^s$; the limit (i.e., intersection) of the $Y^r$'s along these maps is the singleton $\{x\}$.

\begin{lem} \label{lem:approx_x}
The inverse family of $\Dmod^!(X)$-linear functors $(\iota_{x \hto Y^r })^! : \Dmod^!(Y^r) \to \Dmod(x)$ gives rise to a $\Dmod^!(X)$-linear equivalence
$$
\uscolim{\R^{op}, (\iota_{r,s})^!} \Dmod^!(Y^r) \to \Dmod(x).
$$
\end{lem}

\begin{proof}
The identifications of Lemma \ref{lem:BN} give assemble into a $\Dmod^!(X)$-linear equivalence
$$
\uscolim{\R^{op}, (\iota_{r,s})^!} \Dmod^!(Y^r) 
\simeq
\uscolim{\R^{op}, \mathsf{proj}} \Dmod(x^r) \usotimes{\Dmod(X^r)} \Dmod^!(X).
$$
The RHS is itself equivalent to $\Dmod(x^r) \usotimes{\Dmod^!(X)} \Dmod^!(X) \simeq \Dmod(x)$.
The composition is the compatible family of pull-backs, as claimed.
\end{proof}
 
\begin{cor}
For any $\C \in \Dmod^!(X) \mmod$, the functor of the lemma yields equivalences
$$
\uscolim{\R^{op},( \iota^!)_\C} \! \restr  \C {Y^r} \xto{ \,\, \simeq \, \,} \restr \C x 
\hspace{.6cm}
\uprestr \C x \xto{ \,\, \simeq \, \,} \lim_{\R,( \iota^!)^\C} \uprestr \C {Y^r}.
$$
\end{cor}

\sssec{}

From now on, we reinstate the set-up of Sect. \ref{sssec:setup-transitive}.

\medskip

By hypothesis, $H$ acts on each $X^r$; we let $S^r$ be the stabilizer of $x^r \in X^r$. 
Equivalently, $S^r$ is the subgroup of $H$ preserving $Y^r$. 
The embeddings $\iota_{r,s}$ yield isomorphisms $Y^r/S^r \to Y^s/S^s$, which interpolate the given $\{x\}/S
\xto{\simeq} X/H$.

\begin{rem} \label{rem:dim_thry-chosen}
We equip $X$ with the dimension theory defined by $\dim(X^r) = \codim (H : S^r)$. Clearly, the latter integer equals $\codim(X : Y^r)$.
\end{rem}

\begin{lem} \label{lem:smooth-gen-in-action}
Consider the functor $\sQ : \R^\op \to \DGCat$ that sends $r$ to $(\restr \C {Y^r})^{S^r}$ and $r \to s$ to
$$
\rho_{r,s} := (\iota_{r,s}^!)_\C \circ \oblv^{S^s \to S^r}: (\restr \C {Y^s})^{S^s} \to (\restr \C {Y^r})^{S^r}.
$$
The compatible family of restrictions $(\restr \C {Y^r})^{S^r} \to (\restr \C {x})^{S}$ yields the equivalence
\begin{equation}
\label{eqn:zaaa}
\uscolim{\R^{\op}, \sQ } (\restr \C {Y^r})^{S^r} \simeq (\restr \C {x})^{S}.
\end{equation}
\end{lem}

\begin{proof}
Indeed,
\begin{eqnarray} \label{eqn:zbbb}
\uscolim{r \in \R^\op} (\restr \C{Y^r})^{S^r}
\simeq
\uscolim{r \in \R^\op} \,  \uscolim{s \in \R^\op_{r/}}  (\restr \C{Y^r})^{S^s}
\simeq
\Big(\,\uscolim{r \in \R^\op} (\restr \C{Y^r}) \Big) ^S
\simeq
(\restr \C{x}) ^S.
\end{eqnarray}
The first equivalence is a standard cofinality argument, the second one is ``smooth generation" for $S= \bigcap S^r$ (see Proposition \ref{prop:smooth-generation}), the third follows from Lemma \ref{lem:approx_x}.
\end{proof}

\begin{prop} \label{prop:X/H---x/S} 
The functor
$$
\rho_x :\C^{H, \mu} \to (\restr \C x)^{S, \mu}
$$
is an equivalence of categories.
\end{prop}

\begin{proof}
Our functor can be factored as
$$
\rho_x : \C^H \longrightarrow \uscolim{\R^{\op}, \sQ } (\restr \C {Y^r})^{S^r} \xto{\, \,\,\ref{eqn:zaaa} \,\, \,} (\restr \C {x})^{S},
$$
where the first arrow is the direct family of equivalences $\rho_{Y^r}$ of Proposition \ref{prop:X/G=Y/S_pro}.
\end{proof}

\begin{prop} \label{prop:sigma_equiv} 
The functor
$$
\sigma_x = \Av_*^{H/S} \circ (\iota^!)^\C :(\uprestr \C x)^{S, \mu} \to \C^{H, \mu} 
$$
is an equivalence.
\end{prop}

\begin{proof}
The proof is similar to the one above, using the equivalences $\sigma_Y$ and the obvious variant of Lemma \ref{lem:smooth-gen-in-action}.
\end{proof}

\begin{prop} \label{ZC}
The composition
$$
(\sigma_x)^{-1} \circ (\rho_x)^{-1}: (\restr \C x)_S \to (\uprestr \C x)^S
$$
coincides with the map $\Theta_{x \to X}$, which is, therefore, an equivalence.
\end{prop}

\begin{proof}
By Corollary \ref{triangle}, the composition in question comes from the colimit of the functors
$\theta_{Y^r \hto X} [2 \dim(H:S^r)]$. By the general paradigm of Section \ref{paradigm-colim-lim}, this functor is induced from the functor of action with
$$ 
\uscolim{r \in \R^\op} \Big( (\iota_{Y^r \hto X})_+ (\omega_{Y^r}) [2 \codim(H:S^r)] \Big) \in \Dmod^!(X).
$$
Thus, we need to show that the latter is isomorphic to $\delta_x^\ren$, for the dimension theory of Remark \ref{rem:dim_thry-chosen}.
It suffices to prove that $(\iota_{Y^r \hto X})_+ (\omega_{Y^r}) \simeq (\pi_{X \to X^r})^! (\delta_{x^r})$. Thanks to the $(\iota_+, p^!)$ base-change along
\begin{gather}
\xy
(00,0)*+{ \pt }="00";
(30,0)*+{ X^r, }="10";
(0,15)*+{ Y^r }="01";
(30,15)*+{ X  }="11";
{\ar@{<-}_{ p  } "00";"01"};
{\ar@{->}_{ {x^r} } "00";"10"};
{\ar@{->}^{ \iota } "01";"11"};
{\ar@{<-}_{  \pi_{\infty \to r}} "10";"11"};
\endxy
\end{gather} 
this is immediate.
\end{proof}

\ssec{Actions by semi-direct products} \label{ssec:semi-direct}

We conclude with a short discussion of action by semi-direct products.

\sssec{}

Let $H \ltimes S$ be a semidirect product of pro-groups and denote by $\q: H \ltimes S \to S$ the projection.
Since $\Dmod^*(H \ltimes S) \mmod \simeq \Dmod^*(S) \mod (\Dmod^*(H) \mmod)$, we obtain

\begin{lem} \label{lem:AvAv_commute}
For any $\C \in \Dmod^*(H \ltimes S) \mmod$, the functors $\oblv^S : \C^S \rightleftarrows \C: \Av_*^S$ are $\Dmod^*(H)$-linear.
\end{lem}

In particular, if $\m: H \ltimes S \to \GG_a$ is a character, we obtain a functorial identification 
$$
(-\m)_H \star (-\m)_S \simeq (-\m)_S \star (-\m)_H;
$$
in other words, $\Av_*^{H,\m}$ and $\Av_*^{S,\m}$ commute (when viewed as endofunctors of $\C$).
 
\sssec{} 

The action of $\Dmod^*(H)$ on $\C^S$ and $\C_S$ coincides with the action of $\Dmod^*((H \ltimes S)/S)$ guaranteed by Lemma \ref{lem:auxiliary}. This is obvious as $(H \ltimes S)/S$ and $H$ are isomorphic as groups. Hence:

\begin{lem} \label{lem:semidirect-pro}
There is an equivalence $(\C^S)^H \xto{\simeq} \C^{H \ltimes S}$ compatible with $\oblv^H$ and $\oblv^{H \ltimes S}$. Likewise, there is an equivalence $(\C_S)_H \xleftarrow{\simeq} \C_{{H \ltimes S}}$ compatible with $\pr_S$ and $\pr_{H \ltimes S}$. In other words, the following diagrams are commutative:
\begin{gather}
\xy
(00,00)*+{ \C^S }="00";
(35,00)*+{ \C }="10";
(00,15)*+{ (\C^S)^H }="01";
(35,15)*+{ \C^{H \ltimes S} }="11";
{\ar@{<-}^{ \oblv^H } "00";"01"};
{\ar@{->}_{ \oblv^S  } "00";"10"};
{\ar@{->}_{\simeq } "01";"11"};
{\ar@{<-}^{ \oblv^{H \ltimes S}} "10";"11"};
(60,0)*+{ \C_{H \ltimes S} }="20";
(60,15)*+{ \C }="30";
(97,0)*+{ (\C_S)_H. }="21";
(97,15)*+{ \C_S }="31";
{\ar@{->}^{ \simeq } "20";"21"};
{\ar@{<-}_{ \pr_{H \ltimes S} } "20";"30"};
{\ar@{<-}_{ \pr_{H} } "21";"31"};
{\ar@{->}^{ \pr_S} "30";"31"};
\endxy
\end{gather} 
\end{lem}

\begin{rem}
By adjunction, the above equivalence $(\C^S)^H \xto{\simeq} \C^{H \ltimes S}$ is also compatible with $\Av_*^H$ and $\Av_*^{H \ltimes S}$.
\end{rem}

\begin{rem}
We will also need a version of the above lemma for ind-pro-groups. Let $\H = \colim_\I H_n$, $\S = \colim_\I S_n$ be filtered colimits of pro-groups under closed embeddings and assume that $H_n$ acts on $S_n$ for each $n \in \I$. For instance, $\H = \bN'$ and $\S = \bA$ (see Sect. \ref{combinat}).
The statement of Lemma \ref{lem:semidirect-pro} applies verbatim to $\H \ltimes \S$. The proof is straightforward.
\end{rem}

\section{Actions by the loop group of $GL_n$} \label{SEC:GL_n}

We assume from now on that $G = GL_n$. Let $B$ be the Borel subgroup of upper triangular matrices and $N \subset B$ its unipotent radical: upper triangular matrices with $1$'s on the diagonal. Further, let $P \subseteq G$ be the mirabolic subgroup, $T \subseteq B$ be the maximal torus of diagonal matrices, $B_-$, $N_-$ be the opposite Borel and its unipotent radical.
The character $\chi$ on $\bN$ simply computes the sum of the residues of the entries in the diagonal $(i,i+1)$.
To approximate $\bN$, we shall use an explicit sequence of group schemes $\bN_k$: consider the diagonal element $\gamma:= \on{diag}\left( t^{nk}, t^{(n-1)k}, \ldots, t^{k} \right) \in \bT$. We let $\bN_k := \gamma^{-1} \cdot N[[t]] \cdot \gamma$. 

We will also need analogous notations for subgroups of $G' := GL_{n-1}$, its loop group and so on. Thus $B', N', \bG', \bN', \chi'$ have their obvious meanings.

\ssec{Statement of the main theorem}

Let $\C$ be a category acted on by $\bP := P \ppart$. 
Let $k \geq 1$ and fix once and for all the trivialization of the dimension torsor of $\bN$ defined by $\dim(\bN_k) = 0$; we indicate by $\Lambda$ the corresponding equivalence $\Lambda: \Dmod^!(\bN) \to \Dmod^*(\bN)$. 
Recall the functor $\Theta := \Theta_{\Lambda} : \C_{\bN, \chi} \to \C^{\bN, \chi}$, defined in Conjecture \ref{conj:Tequiv}:
$$
\Theta \simeq \uscolim{\ell \geq k} \, \Av_*^{\bN_\ell, \chi} [2 \codim(\bN_\ell : \bN_k)].
$$
Our goal is to prove:

\begin{thm} \label{thm:main}
For any $\C \in \bP \rrep$, the above functor $\Theta: \C_{\bN,\chi} \to \C^{\bN,\chi}$ is an equivalence of categories.
\end{thm}

In the remainder of the text, we shall define a pro-unipotent group scheme $\bH_k \subset GL_n \ppart$ and establish the following two results, whose combination implies the above theorem.

\begin{thm} \label{thm:refinement}
For any $k \geq 1$, the functors
$$
\Phi:= \Av_*^{\bH_k, \chi} \circ \, \oblv^{\bN,\chi}: \C^{\bN,\chi} \longrightarrow  \C^{\bH_k, \chi}
$$
$$
\Psi:= \pr_{\bN, \chi} \circ \, \oblv^{\bH_k,\chi}: \C^{\bH_k, \chi} \longrightarrow \C_{\bN,\chi}
$$
are equivalences of categories.
\end{thm}

\begin{prop} \label{prop:comp=T}
For $k$ and $\Lambda$ chosen as above, we have a commutative triangle
\begin{gather} \label{trianleZ}
\xy
(20,20)*+{ \C^{\bH_k,\chi} }="A";
(00,00)*+{ \C_{\bN,\chi} }="B";
(40,00)*+{ \C^{\bN,\chi}. }="C";
{\ar@{->}_{ \Psi } "A";"B"};
{\ar@{->}^{ \Theta } "B";"C"};
{\ar@{->}_{ \Phi } "C";"A"};
\endxy
\end{gather} 
\end{prop}
\smallskip

\ssec{Some combinatorics of $GL_n$} \label{combinat}

Let us introduce the long-awaited group ${\bH}_k$.
For $k \geq 1$, consider the product $\bB_-^k \cdot N(\O) \subset \GO$. To show it is a group, notice that it is the preimage of $N[t]/t^n$ under the group epimorphism $\p: \GO \to G[t]/t^k$.

\sssec{}

Let $\gamma':= \on{diag}\left( t^{(n-1)k}, \ldots, t^{k} \right) \in \bT'$. We first define the group
$$
\bG_k := \Ad_{\gamma'^{-1}} \left( (\bB_-')^k \cdot N'(\O) \right) = \Ad_{\gamma'^{-1}} \left( (\bN_-')^k \right)  \cdot (\bT')^{k} \cdot \bN'_k.
$$
For example, when $n=2$ and $n=3$ and $n=4$, we have
$$
\bG^{(2)}_k = (1+ t^k \O ),
\, \, \, 
\bG^{(3)}_k = \left( \begin{array}{cc}
1+t^k \O & t^{-k}\O  \\
t^{2k}\O & 1+t^k\O \end{array} \right),
\, \, \,
\bG_k^{(4)} = \left( \begin{array}{ccc}
1+t^k\O & t^{-k}\O & t^{-2k}\O \\
t^{2k}\O & 1+t^k\O & t^{-k}\O \\
t^{3k}\O & t^{2k}\O & 1+t^k\O \end{array} \right).
$$
For higher $n$, the structure of $\bG_k$ follows the evident pattern. The first important feature of $\bG_k$ is the following:

\begin{lem}
The group $\bG_k$ is endowed with a character $\chi_g$ that extends the character $\chi'$ on $\Ad_{\gamma'^{-1}} \left( N'(\O) \right) = \bN'_k$ and that is trivial on $ \Ad_{\gamma'^{-1}} \left( (\bB_-')^k \right)$. In other words, $\chi_g$ computes the sums of the residues of the entries $(i,i+1)$.
\end{lem}

\begin{proof}
Each element of $\bG_k$ can be written uniquely as $\gamma'^{-1} \cdot y \cdot \gamma'$. We set
$$
\chi_g( \gamma'^{-1} \cdot y \cdot \gamma' ) := \chi'(\gamma'^{-1} \cdot \wh\p(y) \cdot \gamma'),
$$
where $\wh\p: (\bB_-')^k \cdot N'(\O) \to N'[t]/t^k$ is the projection. Since $\wh\p$ is multiplicative at sight, $\chi_g$ is a character. The other required properties are obvious.
\end{proof}

In view of this, we henceforth denote the character $\chi_g$ simply by $\chi'$.

\sssec{}

Consider now, for any $\ell \in \mathbb N$ the vector group
$$
\bA_\ell := \bA^{n-1}_\ell := \left(  \begin{array}{l}
t^{-(n-1)\ell}\O \\
t^{-(n-2)\ell}\O \\
\vdots \\
t^{-\ell }\O \end{array} \right)
$$
Let $\chi_a$ be the character on $\bA_\ell$ that computes the residue of the \emph{last entry}.
The second important feature of $\bG_k$ is that it acts on $\bA_k$. We thus form the semidirect product
$$
\bH_k :=\bG_k \ltimes \bA_k =  
\left( \begin{array}{cc}
\bG_k &  \bA_k \\
0 & 1 \end{array} \right).
$$
It is evident that $\chi_a$ in $\bG_k$-invariant; hence the sum $\chi' + \chi_a$ defines a character on $\bH_k$, also indicated by $\chi$.

\sssec{}

Let $\bA := \mathbb A^{n-1} \ppart$, identified with its dual via the residue pairing. The annihilator of $\bA_k$ in $\bA ^\vee \simeq \bA$ is obviously
$$
\bA_k^\perp =
\left(  \begin{array}{l}
t^{(n-1)k}\O \\
t^{(n-2)k}\O \\
\vdots \\
 t^{k}\O \end{array} \right).
$$
Let $\bL_k := \{e_{n-1}\} + \bA_k^\perp \subset \bA^\vee$.
Notice the third important feature of $\bG_k$, which is actually the main motivation for the theory of Section \ref{SEC:trans_act}.

\begin{lem} \label{lem:last_row_e_n-1}
The dual $\bG_k$-action on $\bA^\vee$ makes $\bL_k$ into a pro-scheme acted on transitively by $\bG_k$. The stabilizer of the point $e_{n-1} \in \bL_k$ is exactly $\bH'_k$.
\end{lem}

\begin{proof}
The action in question preserves $\bL_k$ and it is transitive at sight. To show the last claim, it suffices to notice that $\bH'_k$ is obtained from $\bG_k$ by setting the last row of the latter to be $(0, \ldots,0,1)$.
\end{proof}

\ssec{Proof of the main theorem}

We have introduced all the necessary tools to prove Theorem \ref{thm:refinement}, which we have split in two: Propositions \ref{lem:PHI} and \ref{lem:PSI}.

\begin{prop} \label{lem:PHI}
For any category equipped with an action of $\bP$, the functor
$$
\Phi:= \Av_*^{\bH_k, \chi} : \C^{\bN,\chi} \longrightarrow \C^{\bH_k, \chi}
$$
is an equivalence.
\end{prop}

\begin{proof}
We proceed by induction on $n$, the claim being tautologically true for $n=1$. 
First off, $\C^{\bA, \chi}$ retains an action of $\bP' \subseteq \bG'$, so that $\Phi'$ can be applied to $(\C^{\bA, \chi})^{\bN', \chi}$. Secondly, the equivalence
$\sigma : (\uprestr \C {\chi_a} )^{\bH'_k, \chi'}
\xrightarrow{ \, \,}
(\uprestr \C {\bL_k} )^{\bG_k, \chi} \simeq \C^{\bH_k}
$
of Proposition \ref{prop:sigma_equiv}, goes over via Fourier transform to 
$$
\Av_*^{\bH_k, \chi} : \C^{\bH'_k \ltimes \bA, \chi} \to \C^{\bH_k}.
$$
Hence, the composition
$$
\C^{\bN, \chi} \simeq (\uprestr \C {\chi_a })^{\bN', \chi'} \xrightarrow{ \, \, \Phi' =\Av_*^{\bH'_k, \chi'} \, \,}
(\uprestr \C {\chi_a} )^{\bH'_k, \chi'}
\xrightarrow{ \, \, \sigma \,\,}
(\uprestr \C {\bL_k} )^{\bG_k, \chi} \simeq \C^{\bH_k}
$$
is manifestly isomorphic to $\Phi$ and it is an equivalence by the induction hypothesis.
\end{proof}

This concludes the proof the first part of Theorem \ref{thm:refinement}. Before proceeding with the analysis of $\Psi$, let us observe the following.

\begin{rem} \label{rem:leftleftadjoints}
For any $k \geq 1$, the functor $\Phi^{-1}$ equals
$$
\Av_!^{\bN, \chi} \simeq \uscolim{\ell \geq k} \Av_!^{\bN_\ell, \chi}: \C^{\bH_k,\chi} \longrightarrow \C^{\bN, \chi}.
$$
Likewise, $\sigma^{-1} \simeq \Av_!^{\bA, \chi_a} : \C^{\bH_k, \chi} \to \C^{\bH'_k \ltimes \bA, \chi}$.
\end{rem}

\begin{prop}  \label{lem:PSI}
For any category equipped with an action of $\bP$, the functor
$$
\Psi := \pr_{\bN, \chi}: \C^{\bH_k, \psi} \to \C_{\bN ,\chi}
$$
is an equivalence.
\end{prop}

\begin{proof}
Since $\C_{\bA, \chi}$ retains an action of $\bG'$, the functor 
$$
\Psi': (\C_{\bA, \chi})^{\bH_k', \chi'} \to (\C_{\bA, \chi})_{\bN', \chi'} \simeq \C_{\bN, \chi}
$$
is an equivalence by induction. Proposition \ref{prop:X/H---x/S} provides the equivalence
$$
\rho: 
(\uprestr \C {\bL_k})^{\bG_k, \chi} \to (\restr \C {\chi_a})^{\bH'_k},
$$
which goes over via Fourier transform to 
$$
\pr_{\bA, \chi} : \C^{\bH_k, \chi} \xto{\, \, \simeq \, \,} (\C_{\bA, \chi})^{\bH_k', \chi}.
$$
The resulting equivalence 
$$
\C^{\bH_k, \chi} 
\xrightarrow{\, \, \pr_{\bA, \chi}}
\big( \C_{\bA, \chi}\big)^{\bH'_k, \chi'} 
\xrightarrow{\, \, \Psi'}
\big( \C_{\bA, \chi}\big)_{\bN', \chi'}
\simeq
\C_{\bN, \chi}
$$
is evidently equivalent to $\Psi$.
\end{proof}

\sssec{}

Proposition \ref{prop:comp=T} is what remains to be proven. We have the semi-direct product decomposition $\bN_\ell \simeq \bN_\ell' \ltimes \bA_\ell$, for any $\ell \geq 1$. Fix the dimension theory of $\bA$ determined by $\dim( \bA_k) = 0$ and let $\Theta_\bA : \C_{\bA,\chi} \to \C^{\bA, \chi}$ be the corresponding renormalized $*$-averaging functor.
With these choices,
\begin{equation} \label{blah}
\Theta_\bN \simeq \Theta_{\bN'} \circ \Theta_{\bA}.
\end{equation}
Recall now that $\sigma \circ \rho \simeq \Theta_{\chi_a}$ and that, under Fourier transform, $\Theta_{\chi_a}$ goes over to $\Theta_\bA$ (Lemma \ref{lem:FT-of-omega-ren}).

\medskip

Hence, Proposition \ref{prop:comp=T} is a consequence of the following statement:

\begin{lem}
The square
\begin{gather}
\xy
(0,0)*+{ \C^{\bN,\chi}  }="00";
(32,0)*+{  \C_{\bN,\chi} }="10";
(00,15)*+{ (\C^{\bA,\chi})^{\bH'_k,\chi}}="01";
(32,15)*+{  (\C_{\bA,\chi})^{\bH_k',\chi} }="11";
{\ar@{->}_{ \Phi'   } "00";"01"};
{\ar@{<-}^{ \Theta_{\bN} } "00";"10"};
{\ar@{<-}_{ \Psi'  } "10";"11"};
{\ar@{<-}^{ \Theta_{\bA} } "01";"11"};
\endxy
\end{gather} 
is commutative.
\end{lem}

\begin{proof}
We may assume by induction that $\Theta_{\bN'} \simeq \Phi^{-1} \circ \Psi^{-1}$. Thanks to (\ref{blah}), it suffices to prove that $\Phi'$ (viewed as an endofunctor of $\C$) commutes with $\Theta_{\bA}$ (also viewed as an endofunctor of $\C$). This follows from Lemma \ref{lem:AvAv_commute}. In order to apply the Lemma, one has to write $\bA$ as a colimit of vector spaces on which $\bH_k'$ acts; this can always be done.
\end{proof}


\newcommand{\eprint}[1]{Preprint {\tt #1}}\newcommand{\available}[1]{Available
  from \url{#1}}

\end{document}